\documentclass[twoside, a4paper]{amsart}
\usepackage{amsmath}
\usepackage{amsthm}
\usepackage[utf8]{inputenc}
\usepackage{mathrsfs}
\usepackage{amssymb}
\usepackage{bm}
\usepackage{tikz-cd}
\usetikzlibrary{calc}
\usepackage{url}
\usepackage{etoolbox}
\usepackage{hyperref}
\usepackage{cleveref}

\newtheorem{thm}{Theorem}[section]
\newtheorem{lemma}[thm]{Lemma}
\newtheorem{cor}[thm]{Corollary}
\newtheorem{prop}[thm]{Proposition}
\newtheorem{introthm}{Theorem}

\theoremstyle{remark}
\newtheorem{rk}[thm]{Remark}
\newtheorem{constr}[thm]{Construction}
\newtheorem{warn}[thm]{Warning}
\newtheorem{ex}[thm]{Example}
\newtheorem{nex}[thm]{Non-example}

\newtheorem*{claim*}{Claim}
\AtBeginEnvironment{claim*}{\let\oldqed=\qedsymbol%
\renewcommand\qedsymbol{$\triangle$}}
\AtEndEnvironment{claim*}{\let\qedsymbol=\oldqed}

\theoremstyle{definition}
\newtheorem{defi}[thm]{Definition}

\numberwithin{equation}{section}

\newcommand{\Comm}{\mathscr C\!\mathit{om}}
\newcommand{\iP}{\amalg}

\newcommand{\cat}[1]{\textbf{\textup{#1}}}
\newcommand{\glo}[1]{\edef\arg{#1}\edef\brg{GlobalSpectra}\edef\crg{GlobalSpaces}\edef\drg{NUCA}\edef\erg{Mod}\edef\frg{Comm}\underline{\ifx\arg\brg\mathscr S\!p\else\ifx\arg\crg\mathscr S\else\ifx\arg\drg\mathscr N\else\ifx\arg\erg\mathscr M\else\ifx\arg\frg\Comm\else\textup{#1}\fi\fi\fi\fi\fi}}
\newcommand{\GlSp}{\glo{Sp}}
\newcommand{\Sp}{\textup{Sp}}
\newcommand{\op}{{\textup{op}}}
\newcommand{\blank}{{\textup{--}}}
\newcommand{\pr}{{\textup{pr}}}
\newcommand{\Hom}{{\textup{Hom}}}
\newcommand{\id}{{\textup{id}}}
\newcommand{\ev}{{\textup{ev}}}

\newcommand{\colim}{\mathop{\textup{colim}}\nolimits}
\newcommand{\Fun}{\textup{Fun}}

\newcommand{\maps}{\mathord{\textup{maps}}}

\newcommand{\forget}{\mathop{\textup{forget}}}

\newcommand{\ppo}{\mathbin{\raise.25pt\hbox{\scalebox{.67}{$\square$}}}}

\newcommand{\Ho}{{\textup{Ho}}}

\newcommand{\sh}{\mathop{\textup{sh}}\nolimits}

\newcommand{\extsmash}{\mathbin{\widehat\smashp}}
\newcommand{\extotimes}{\mathbin{\widehat\otimes}}
\newcommand{\cofib}{{\textup{cofib}}}
\newcommand{\pos}{{+}}
\newcommand{\triv}{\mathop{\textup{triv}}}
\newcommand{\End}{{\mathbb E\textup{nd}}}
\newcommand{\unit}{\bm 1}
\newcommand{\power}[1]{\mathbb P^{{#1}}}

\DeclareMathOperator{\TAQ}{TAQ}
\DeclareMathOperator{\Ab}{Ab}
\DeclareMathOperator{\Mod}{Mod}
\DeclareMathOperator{\CAlg}{CAlg}
\DeclareMathOperator{\NUCA}{NUCA}

\let\ul=\underline
\let\del=\partial
\let\smashp=\wedge
\let\setminus=\smallsetminus
\let\phi=\varphi
\let\_=\blank

\def\twocell[#1]{\arrow[#1, dash, phantom, "\Rightarrow"{scale=1.125, yshift=-.4pt, description, allow upside down, sloped, inner sep=0pt}]}

\title[Global model categories and top.~André-Quillen cohomology]{Global model categories and\\ topological André-Quillen cohomology}
\author{Tobias Lenz}
\author{Michael Stahlhauer}
\address{T.L.:  Mathematisches Institut, Rheinische Friedrich-Wilhelms-Universit\"at Bonn, Endenicher Allee 60, 53115 Bonn, Germany\hskip0pt plus 5pt \&\hskip 0pt plus 5pt Mathematical Institute, University of Utrecht, Budapestlaan 6, 3584 CD Utrecht, The Netherlands (\textit{current address})}
\address{M.S.: Mathematisches Institut, Rheinische Friedrich-Wilhelms-Universit\"at Bonn, Endenicher Allee 60, 53115 Bonn, Germany}
 % abstract homotopy theory

\begin{document}
\begin{abstract}
We introduce \emph{global model categories} as a general framework to capture several phenomena in global equivariant homotopy theory. We then construct \emph{genuine stabilizations} of these, generalizing the usual passage from unstable to stable global homotopy theory. Finally, we define the \emph{global topological André-Quillen cohomology} of an ultra-commutative ring spectrum and express it in terms of a genuine stabilization in our framework in analogy with the classical non-equivariant description obtained by Basterra and Mandell.
\end{abstract}

\maketitle
\setcounter{tocdepth}{1}
\tableofcontents

\section*{Introduction}
Cohomology theories like topological $K$-theory and the various flavors of cobordism are among the most fundamental tools of algebraic topology. Many of the examples one encounters in practice come with additional structure in the form of multiplications and power operations, and these can often be exploited fruitfully, as for example in the classical proof that the Hopf maps define non-trivial elements in the stable homotopy groups of spheres via the Steenrod operations on singular cohomology. In terms of the representing spectra, these extra algebraic data are encoded in a highly structured multiplication, making the representing spectra so-called \emph{$E_\infty$-ring spectra}. In good model categories of diagram spectra like symmetric or orthogonal spectra \cite{hss, mmss}, such $E_\infty$-ring spectra are represented by strictly commutative algebras, and they can also be conveniently described in the language of $\infty$-categories.

\emph{$G$-equivariant cohomology theories}, as modelled by \emph{genuine $G$-spectra} in the sense of equivariant stable homotopy theory, are a refinement of cohomology theories to the context of objects endowed with extra symmetries in the form of an action of a (finite) group $G$. Many classical cohomology theories have equivariant analogues, leading for example to $G$-equivariant topological $K$-theory and cobordism.

When we study highly structured multiplications on these, interesting new structure arises in the form of \emph{norm maps}, which can be thought of as twisted multiplications. As a motivating algebraic example, consider an ordinary commutative ring $R$ with an action by a finite group $G$ and an $H$-fixed point $r\in R^H$ for a subgroup $H\subset G$. Then the \emph{norm} of $r$ is the product
\[ N_H^G(r) = \prod_{gH\in G/H} g. r,\]
which is now a \emph{$G$-fixed} point. This yields a multiplicative map $N_H^G\colon R^H\to R^G$, and together with inclusions of fixed points as restrictions and similarly defined additive \emph{transfers}, this gives rise to the structure of a \emph{Tambara functor} \cite{tambara} on the collection $\{ R^H\}_{H\subset G}$ of fixed points of $R$. More generally, the zeroth homotopy groups of any \emph{genuine $G$-$E_\infty$-ring spectrum} naturally admit the structure of a Tambara functor and in particular come with norm maps. These norms have been famously exploited in the solution of the Kervaire-invariant-one problem by Hill, Hopkins, and Ravenel \cite{hhr}, renewing interest in this rich structure.

It turns out that many important equivariant cohomology theories like cobordism and $K$-theory exist in a uniform fashion for large classes of groups, like all finite or all compact Lie groups. This is the perspective taken by \emph{global homotopy theory} \cite{schwede-book, hausmann-global}. Such global objects come with extra structure in the form of \emph{inflations} (restrictions along surjective group homomorphisms), and incorporating this additional information can allow for easier analysis of equivariant phenomena, the most prominent example being the recent proof of an equivariant Quillen theorem linking equivariant bordism to formal group laws by Hausmann \cite{hausmann-formal}, which crucially relies on the global perspective.

The correct notion of `multiplicative global cohomology theories' are the \emph{ultra-commutative ring spectra}, which are commutative algebras in a suitable model category of global spectra. The zeroth homotopy groups of such an ultra-commutative ring spectrum form a so-called \emph{global power functor}, the global analogue of a $G$-Tambara functor for fixed $G$, in particular coming with norm maps for any inclusion $H\subset G$ of finite groups. Currently, no purely $\infty$-categorical description of ultra-commutative ring spectra is known.

\subsection*{Obstruction theory and topological André-Quillen cohomology}
As commutative ring spectra encode such a rich additional structure in their homotopy groups and represented cohomology, their study has received much attention. However, the existence of structured multiplications on spectra is a very subtle question: while many important spectra do come with the structure of an $E_\infty$-ring spectrum, such as the sphere spectrum, Eilenberg-MacLane spectra, $K$-theory and Thom spectra, or spectra of topological modular forms, some other naturally defined spectra do not admit such a structured multiplication. The prime examples of this are Moore spectra, where it has been long known that $\mathbb S/2$ does not admit a unital multiplication in the homotopy category and no $\mathbb S/p$ can have a structured associative \hbox{(i.e.~$A_\infty$-)}\kern0ptmultiplication, see \cite[Theorem 1.1]{araki-toda} and \cite[Example 3.3]{angeltveit-thh}. Similarly, the Brown-Peterson spectrum $BP$ (for any prime $p$) does not support an $E_\infty$-multiplication \cite{lawson-BP, senger-BP}.

In order to derive positive results, in some cases obstruction theoretic methods can be applied, with famous examples being the result of Goerss and Hopkins \cite{goerss-hopkins} that Morava $E$-theory admits a unique $E_\infty$-ring structure, or the usage of Postnikov towers of commutative ring spectra in order to construct an $E_4$-multiplication on $BP$ by Basterra and Mandell \cite{basterra-mandell-BP}. Both of these results rely on \emph{topological André-Quillen cohomology}, a cohomology theory for commutative ring spectra introduced by Basterra \cite{basterra-TAQ} as an adaptation of a cohomology theory originally defined by André and Quillen for ordinary commutative rings \cite{andre-homology, quillen-cohomology-algebras}. The latter is defined as a derived functor of Kähler differentials, and a similar approach is used by Basterra for the topological version, using a model category of spectra in which commutative ring spectra model $E_\infty$-ring spectra.

A crucial observation in the construction of classical \emph{algebraic} André-Quillen cohomology is that for a fixed commutative ring $R$, one can identify the category of $R$-modules with the category of abelian group objects in augmented $R$-algebras, which allows to interpret André-Quillen homology and cohomology as a derived abelianization procedure. In the topological case, a similar result was obtained by Basterra and Mandell \cite{basterra-mandell-stab}, who showed that for a commutative ring spectrum $R$, topological André-Quillen cohomology exhibits the category of $R$-modules as a stabilization of the category of augmented $R$-algebras.

\subsection*{Global topological André-Quillen cohomology}
Given the wealth of structure encoded in the multiplication on an ultra-commutative ring spectrum (even compared to a non-equivariant $E_\infty$-ring spectrum), a \emph{global obstruction theory} would be particularly desirable. As the first step towards this, we adapt the theory of topological André-Quillen cohomology to the context of global homotopy theory in this article. In particular, we define the global topological André-Quillen homology and cohomology of ultra-commutative ring spectra, prove a Hurewicz theorem in this setting, and construct Postnikov towers for ultra-commutative ring spectra. As the main result of this paper (see Theorem~\ref{introthm:global-stab-calg} below) we then identify global topological André-Quillen cohomology as a suitable \emph{global stabilization} in analogy with the non-equivariant result of Basterra--Mandell, justifying that this is indeed the `correct' global analogue of classical topological André-Quillen cohomology.

To put this into perspective, recall that any suitably nice model or $\infty$-category admits a \emph{stabilization} \cite{schwede-stab,higher-algebra}, obtained by inverting the suspension-loop adjunction. However, the passage between unstable and stable equivariant or global homotopy theory is more subtle, related to the existence of transfer maps between equivariant stable homotopy groups mentioned above. In the equivariant setting, we can instead concisely express it as a \emph{genuine stabilization} \cite[Appendix~C]{gepner-meier}: the passage from unstable to stable $G$-equivariant homotopy theory is given by universally inverting the $1$-point compactification $S^G$ of the regular real representation, instead of just the usual sphere $S^1$ (i.e.~the $1$-point compactification of the trivial $1$-dimensional representation). Sadly, however, this approach can not be immediately adapted to the global world---for example, a computation by Schwede shows that only the ordinary non-equivariant spheres are inverted when passing from unstable to stable global homotopy theory.

Similarly to an idea described by Gepner and Nikolaus \cite{gepner-nikolaus}, we solve this issue in the present paper by looking more generally at \emph{$G$-global homotopy theory} in the sense of \cite{g-global} for all finite groups $G$, which for the trivial group $G=1$ recovers usual global homotopy theory. In this setting, we then describe the passage from unstable to stable $G$-global homotopy theory \emph{for all finite groups $G$ simultaneously}: it is given by inverting $S^G$ in $G$-global homotopy theory for all $G$ in a compatible way (see Theorem~\ref{introthm:global-stab-s} below).

\subsection*{Global model categories}
To make this precise, we introduce the notion of a \emph{global model category} (Definition~\ref{defi:global-model-cat}). Roughly speaking, such a global model category $\ul{\mathscr C}$ consists of a category $\mathscr C$ together with two (suitably nice) Quillen equivalent model structures on the category $G\text-\mathscr C$ of $G$-objects in $\mathscr C$, called the \emph{projective} and \emph{flat} model structures, that interact in a prescribed way with restriction along group homomorphisms, formalizing the behaviour established for unstable and stable $G$-global homotopy theory in \cite{g-global}. As our main examples, we introduce and study global model categories $\glo{GlobalSpaces}$ of global spaces, $\glo{GlobalSpectra}$ of global spectra, $\glo{Mod}_R$ of $R$-modules, and $\glo{Comm}_R/R$ of augmented $R$-algebras for an ultra-commutative ring spectrum $R$.

If $\ul{\mathscr C}$ is a \emph{pointed} global model category (i.e.~$\mathscr C$ has a zero object) and $G$ is finite, the homotopy category $\Ho(G\text-\mathscr C)$ comes with an equivariant suspension-loop adjunction
\begin{equation*}
S^G\smashp^{\cat L}\blank\colon \Ho(G\text-\mathscr C)\rightleftarrows\Ho(G\text-\mathscr C) :\!\cat{R}\Omega^G
\end{equation*}
(recovering the usual one for pointed global spaces), and we call $\ul{\mathscr C}$ \emph{(genuinely) stable} if this adjunction is an adjoint equivalence for every $G$.

With this definition, both the global model category of global spectra and of $R$-modules are stable. For a general global model category $\ul{\mathscr C}$ on the other hand, we can construct a \emph{global stabilization} in the form of a homotopy universal map to a stable global model category by considering suitable $G$-global spectrum objects in $\mathscr C$, refining and generalizing the non-equivariant construction of \cite{schwede-stab}. We then compute this global stabilization in two key cases:

First, we show that the global stabilization of global spaces is indeed given by global spectra in this setting:
\begin{introthm}[see \Cref{thm:stabilization-global-spaces}]\label{introthm:global-stab-s}
	The suspension spectrum-loop space adjunction
	\begin{equation*}
	\ul{\Sigma^\bullet_+}\colon\glo{GlobalSpaces}\rightleftarrows\glo{GlobalSpectra} :\!\ul{\Omega^\bullet}
	\end{equation*}
	exhibits the global model category $\glo{GlobalSpectra}$ of global spectra as global stabilization of the global model category $\glo{GlobalSpaces}$ of global spaces.
\end{introthm}
This in particular serves as a sanity check for our framework, but it also allows us to provide a description of the previously elusive passage between unstable and stable global homotopy theory in terms of a universal property.

The main application of this theory, however, lies in the calculation of the global stabilization of the category of augmented $R$-algebras for an ultra-commutative ring spectrum $R$:

\begin{introthm}[see \Cref{thm:stabilization-aug-algebras} for a precise statement]\label{introthm:global-stab-calg}
	Let $R$ be a flat ultra-commutative ring spectrum. Then the global model category $\glo{Mod}_R$ of $R$-modules is the global stabilization of the global model category $\glo{Comm}_R/R$ of augmented $R$-algebras, and the universal map reduces on homotopy categories to the `global abelianization' functor from the construction of global topological André-Quillen cohomology.
\end{introthm}

\subsection*{André-Quillen (co)homology of global power functors} Just like topological André-Quillen cohomology is a `higher algebra' version of a cohomology theory for ordinary rings, there is a notion of \emph{global André-Quillen cohomology} \cite[Chapter~1]{stahlhauer} of `commutative global rings,' i.e.~the aforementioned global power functors; a similar theory of \emph{equivariant André-Quillen cohomology} for Tambara functors has moreover been studied by Hill \cite{hill-EAQ}.

This theory enjoys several properties one would expect of a global analogue of André-Quillen cohomology: for example, it can be described in degree $0$ in terms of derivations of global power functors and in degree $1$ by a global version of Grothendieck's $\text{Exalcomm}$-functor. However, it also suffers from a serious defect: it does not admit a global analogue of Quillen's \emph{transitivity long exact sequence}. As explained in \cite{stahlhauer}, this failure can be traced back to an anomaly of the abelian category of global functors, first observed by Lewis \cite{lewis-proj-flat}: its projective objects are not necessarily flat. In contrast to that, global \emph{topological} André-Quillen cohomology, as considered in this paper, does come with a transitivity exact sequence, see Corollary~\ref{cor:transitivity-seq-homotopy-groups}.

\subsection*{Related work}
Another perspective on the relation between unstable and stable equivariant homotopy theory is provided by \emph{parameterized higher category theory} in the sense of Barwick, Dotto, Glasman, Nardin, and Shah \cite{elements-param,nardin-orbi}, which emphasizes the extra algebraic structure encoded in the additive transfers (or more precisely the so-called \emph{Wirthmüller isomorphisms} underlying them) as a form of `genuine semiadditivity.'

Using this language, Cnossen, Linskens, and the first author \cite{global-param} have concurrently introduced the concept of \emph{global $\infty$-categories}. These again come with a notion of \emph{genuine stability} (now defined via the aforementioned Wirthmüller isomorphisms), and the main result of \emph{op.~cit.} describes a certain explicit global $\infty$-category of global spectra (again built from $G$-global spectra for all finite $G$) as the genuine stabilization in this sense of an analogous global $\infty$-category of global spaces---in fact, global spaces and global spectra admit universal descriptions in this framework as the free presentable and free presentable genuinely stable global $\infty$-category, respectively, which immediately implies the above description.

It is not hard to show that any global model category gives rise to a (presentable) global $\infty$-category via Dwyer-Kan localization, which for $\glo{GlobalSpaces}$ and $\glo{GlobalSpectra}$ precisely recovers the aforementioned global $\infty$-categories of global spaces and global spectra; in particular, the two notions of global stabilization agree in this case. However, while it is natural to expect them to also agree in general, this is not clear a priori; the first author plans to come back to this question in future work.

\subsection*{Outline}
In Section~\ref{sec:g-global} we give a recollection of $G$-equivariant and $G$-global homotopy theory, in particular describing the rich `change of group'-calculus present in the latter. We then formalize this calculus in Section~\ref{sec:global-model-cat} in the notion of a \emph{global model category}.

Section~\ref{sec:global-stab} is devoted to the notion of \emph{stability} for global model categories and the general construction of global stabilizations. In Section~\ref{sec:global-spectra} we compute this stabilization in the case of global spaces, proving Theorem~\ref{introthm:global-stab-s}.

In Section~\ref{section:brave-new-algebra} we develop the theory of modules and algebras in stable $G$-global homotopy theory, and in particular define corresponding global model categories. Afterwards, we introduce ($G$-)global and $G$-equivariant topological André-Quillen cohomology in Section~\ref{sec:gtaq} and express the former via the global stabilization of augmented commutative algebras, proving Theorem~\ref{introthm:global-stab-calg}. This proof in turn relies on a hard technical result about the stabilization of the free-forgetful adjunction between modules and non-unital commutative algebras (so-called `NUCAs'), to which all of Section~\ref{section:stab-nucas} is devoted.

\subsection*{Acknowledgements} The first author would like to thank Bastiaan Cnossen and Sil Linskens for helpful discussions. The second author would like to thank Stefan Schwede for his support and supervision. Both authors would further like to thank the anonymous referee for helpful feedback that in particular led to the inclusion of Subsection~\ref{subsec:equivariant-TAQ}.

This article is partially based on work supported by the Swedish Research Council under grant no. 2016-06596 while the first author was in residence at Institut Mittag-Leffler in Djursholm, Sweden in early 2022.

During the time this article was written, the second author was supported by the Max-Planck Insitute for Mathematics, Bonn, and was an associate member of the Hausdorff Center for Mathematics at the University of Bonn. Subsections~6.1 and~6.2 are adapted from results contained in the dissertation of the second author.

\section{A reminder on $G$-equivariant and $G$-global homotopy theory}\label{sec:g-global}
Throughout, let $G$ be a finite group. To set the stage, we recall several model categorical aspects of \emph{$G$-equivariant} and \emph{$G$-global homotopy theory}.

\subsection{\texorpdfstring{$\bm G$}{G}-equivariant homotopy theory} We begin with the classical unstable equivariant story:

\begin{prop}
Let $\mathcal F$ be a family of subgroups of $G$, i.e.~a non-empty collection of subgroups that is closed under subconjugates. Then there is a unique model structure on $\cat{$\bm G$-SSet}$ in which a map $f$ is a weak equivalence or fibration if and only if $f^H$ is a weak homotopy equivalence or Kan fibration, respectively, of simplicial sets for all $H\in\mathcal F$. We call this the \emph{$\mathcal F$-equivariant model structure} and its weak equivalences the \emph{$\mathcal F$-weak equivalences}. It is simplicial, proper, and combinatorial with generating cofibrations
\begin{equation*}
\{ G/H\times(\del\Delta^n\hookrightarrow\Delta^n) : H\in\mathcal F,n\ge0\}
\end{equation*}
and generating acyclic cofibrations
\begin{equation*}
\{ G/H\times(\Lambda^n_k\hookrightarrow\Delta^n) : H\in\mathcal F,0\le k\le n\}.
\end{equation*}
Moreover, a map is a cofibration in it if and only if it is levelwise injective and every simplex not in its image has isotropy in $\mathcal F$. Finally, filtered colimits in it are homotopical.
\begin{proof}
By \cite[Example~2.14]{cellular} we get a cofibrantly generated (hence combinatorial) model structure with the above weak equivalences, fibrations, and generating (acyclic) cofibrations, while Proposition~2.16 of \emph{op.~cit.} provides the characterization of the cofibrations. The remaining properties are easy to check, also see \cite[Proposition~1.1.2]{g-global} for a complete proof.
\end{proof}
\end{prop}

\begin{ex}
If we take $\mathcal F=\mathcal A\ell\ell$ to be the collection of all subgroups of $G$, then we get a model structure with cofibrations the underlying cofibrations of simplicial sets, while weak equivalences and fibrations are those maps $f$ such that $f^H$ is a weak equivalence or fibration, respectively, for every subgroup $H\subset G$.

We will refer to this model structure simply as the \emph{$G$-equivariant model structure} and to its weak equivalences as \emph{$G$-equivariant weak equivalences}.
\end{ex}

\begin{ex}\label{ex:graph-model-structure}
Let $G,H$ be finite groups. We write $\mathcal G_{G,H}$ for the family of \emph{graph subgroups} of $G\times H$, i.e.~groups of the form $\Gamma_{K,\phi}\mathrel{:=}\{(k,\phi(k)) : k\in K\}$ for a subgroup $K\subset G$ and a homomorphism $\phi\colon K\to H$. Note that $K$ and $\phi$ are actually uniquely determined for any graph subgroup, and a subgroup $L\subset G\times H$ is a graph subgroup if and only if it intersects $H$ trivially, i.e.~$L\cap (1\times H)=1$.

We now apply the proposition for $\mathcal F=\mathcal G_{G,H}$ to get a \emph{graph model structure} on $\cat{$\bm{(G\times H)}$-SSet}$, which in the case $H=1$ recovers the previous model structure. For general $H$, a map is a cofibration in this model structure if and only if it is levelwise injective and $H$ acts freely outside the image.
\end{ex}

We will also need a variant of the above model structure with more cofibrations:

\begin{prop}
Let $G$ be a finite group and let $\mathcal F$ be a family of subgroups. Then there is a unique model structure on $\cat{$\bm G$-SSet}$ with weak equivalences the $\mathcal F$-equivariant weak equivalences and cofibrations the injective cofibrations, i.e.~the underlying cofibrations of simplicial sets. We call this the \emph{injective $\mathcal F$-equivariant model structure}. It is simplicial, proper, combinatorial with generating cofibrations
\begin{equation*}
\{G/H\times(\del\Delta^n\hookrightarrow\Delta^n) : H\subset G,n\ge0\},
\end{equation*}
and filtered colimits in it are homotopical.
\end{prop}

Observe that for $\mathcal F=\mathcal A\ell\ell$ this recovers the $G$-equivariant model structure again.

\begin{proof}
This is a folklore result; the earliest appearance in the literature (with $\cat{SSet}_*$ instead of $\cat{SSet}$) we are aware of is \cite[Proposition~1.3]{shipley-mixed} where this is already referred to as a `well-known' model structure. A full proof in our setting can be found as \cite[Proposition~1.1.15]{g-global}.
\end{proof}

We will frequently use the following well-known `change of group' properties of the above model structures, all of which can also be found in \cite[1.1.4]{g-global}:

\begin{lemma}\label{lemma:graph-target}
Let $G$ be a finite group and let $\alpha\colon H\to H'$ be a homomorphism of finite groups. Then the adjunction
\begin{equation*}
\alpha_!\mathrel{:=}(G\times\alpha)_!\colon \cat{$\bm{(G\times H)}$-SSet}_{\mathcal G_{G,H}}\rightleftarrows \cat{$\bm{(G\times H')}$-SSet}_{\mathcal G_{G,H
}} :(G\times\alpha)^*\mathrel{=:}\alpha^*
\end{equation*}
is a Quillen adjunction. If $\alpha$ is injective, then also
\begin{equation*}
\alpha^*\mathrel{:=}(G\times\alpha)^*\colon \cat{$\bm{(G\times H')}$-SSet}_{\mathcal G_{G,H'}}\rightleftarrows \cat{$\bm{(G\times H)}$-SSet}_{\mathcal G_{G,H}} :(G\times\alpha)_*\mathrel{=:}\alpha_*
\end{equation*}
is a Quillen adjunction.\qed
\end{lemma}

\begin{lemma}\label{lemma:graph-source}
Let $\alpha\colon G\to G'$ be an injective homomorphism of finite groups and let $H$ be a finite group. Then
\begin{equation*}
\alpha^*\mathrel{:=}(\alpha\times H)^*\colon\cat{$\bm{(G'\times H)}$-SSet}_{\mathcal G_{G',H}}\to \cat{$\bm{(G\times H)}$-SSet}_{\mathcal G_{G,H}}
\end{equation*}
is both left and right Quillen.\qed
\end{lemma}

Next, we come to the stable situation, where we will use Hausmann's model \cite{hausmann-equivariant} based on symmetric spectra, which we briefly recall:

\begin{constr}
We write $\bm\Sigma$ for the $\cat{SSet}_*$-enriched category whose objects are finite sets and with morphism spaces
\begin{equation*}
\maps_{\bm\Sigma}(A,B)\mathrel{:=}\bigvee_{i\colon A\to B\text{ injective}} S^{B\setminus i(A)}.
\end{equation*}
Composition is given by smashing, i.e.~if $C$ is yet another object, then the composition $\maps_{\bm\Sigma}(A,B)\smashp\maps_{\bm\Sigma}(B,C)\to\maps_{\bm\Sigma}(A,C)$ is given on the wedge summands corresponding to $i\colon A\to B$ and $j\colon B\to C$ as
\begin{equation*}
S^{B\setminus i(A)}\smashp S^{C\setminus j(B)}\cong S^{j(B)\setminus ji(A)}\smashp S^{C\setminus j(B)}\cong S^{C\setminus ji(A)}\hookrightarrow \bigvee_{k\colon A\to C\text{ injective}} S^{C\setminus k(A)}
\end{equation*}
where the first isomorphism is induced by $j$, the second one is the canonical isomorphism, and the final map is the inclusion of the summand indexed by $ji$.
\end{constr}

\begin{defi}
A \emph{symmetric spectrum} (in simplicial sets) is an $\cat{SSet}_*$-enriched functor $\bm\Sigma\to\cat{SSet}_*$. We write $\cat{Spectra}$ for the corresponding category of enriched functors and enriched natural transformations, which is itself enriched, tensored, and cotensored over $\cat{SSet}_*$ with (co)tensoring defined levelwise.

If $G$ is any group, then we write $\cat{$\bm G$-Spectra}$ for the category of $G$-objects in $\cat{Spectra}$ and call its objects \emph{$G$-symmetric spectra} or simply (by slight abuse of language) \emph{$G$-spectra}.
\end{defi}

\begin{rk}
Symmetric spectra are often instead defined in a `coordinatized' fashion as sequences $(X_n)_{n\ge0}$ of based $\Sigma_n$-simplicial sets together with maps $S^1\smashp X_n\to X_{n+1}$ that interact suitably with the actions. For the equivalence to the above definition we refer the reader to \cite[2.4]{hausmann-equivariant}.
\end{rk}

We will now construct equivariant model structures on the category of $G$-spectra. Just like non-equivariantly, these come in a \emph{projective} and a \emph{flat} version, and will be obtained by Bousfield localizing suitable level model structures:

\begin{prop}
There is a unique model structure on $\cat{$\bm G$-Spectra}$ in which a map $f$ is a weak equivalence or fibration if and only if $f(A)$ is a $\mathcal G_{G,\Sigma_A}$-weak equivalence or fibration, respectively, for every finite set $A$. We call this the \emph{$G$-equivariant projective level model structure} and its weak equivalences the \emph{$G$-equivariant level weak equivalences}. It is combinatorial with generating cofibrations
\begin{equation*}
\big\{\big(G_+\smashp\bm\Sigma(A,\blank)\big)/H\smashp(\del\Delta^n\hookrightarrow\Delta^n)_+ : \text{$A$ finite set}, H\in\mathcal G_{G,\Sigma_A}, n\ge0\big\}.
\end{equation*}
\end{prop}
More precisely, for the above to be a set (as opposed to a proper class), we should restrict to a set of finite sets hitting all isomorphism classes; in all what follows we will ignore this and similar technicalities.
\begin{proof}
See \cite[Corollary~2.26 and discussion afterwards]{hausmann-equivariant}.
\end{proof}

\begin{prop}
There is a unique model structure on $\cat{$\bm G$-Spectra}$ in which a map $f$ is a weak equivalence or fibration if and only if $f(A)$ is a weak equivalence or fibration, respectively, in the \emph{injective} $\mathcal G_{G,\Sigma_A}$-model structure for every finite set $A$. We call this the \emph{$G$-equivariant flat level model structure}. Its weak equivalences are precisely the $G$-equivariant level weak equivalence; moreover, a map is a cofibration in the flat level model structure on $\cat{$\bm G$-Spectra}$ if and only if it is a cofibration in the flat level model structure on $\cat{Spectra}$ (i.e.~for $G=1$); we will refer to these maps as \emph{flat cofibrations}.

Finally, the $G$-equivariant flat level model structure is combinatorial with generating cofibrations
\begin{equation*}
\big\{\big(G_+\smashp\bm\Sigma(A,\blank)\big)/H\smashp(\del\Delta^n\hookrightarrow\Delta^n)_+ : \text{$A$ finite set}, H\subset G\times\Sigma_A, n\ge0\big\}.
\end{equation*}
\begin{proof}
The construction of the model structure and the identification of the generating cofibrations is \cite[Corollary~2.25 and discussion afterwards]{hausmann-equivariant}, while the characterization of the flat cofibrations is Remark~2.20 of \emph{op.~cit.}
\end{proof}
\end{prop}

\begin{rk}
We will never need to know how the generating \emph{acyclic} cofibrations of the above two model structures look like; the curious reader can find them in Hausmann's treatment referred to above.
\end{rk}

\begin{defi}
A $G$-spectrum $X$ is called a \emph{$G$-$\Omega$-spectrum} if for every $H\subset G$ and all finite $H$-sets $A\subset B$ the derived adjoint structure map $X(A)\to \cat{R}\Omega^{B\setminus A} X(B)$ is an $H$-equivariant weak equivalence, where $H$ acts on $X$, $A$, and $B$.
\end{defi}

Here we are deriving $\Omega^{B\setminus A}$ with respect to the $H$-equivariant model structure on $\cat{$\bm H$-SSet}_*$; in particular, if $X$ is fibrant in either of the above level model structures, then the above is already represented by the ordinary adjoint structure map.

We will also frequently reexpress the above condition as saying that for all finite $H$-sets $A,C$ the map $X(A)\to\cat{R}\Omega^C X(A\amalg C)$ is a weak equivalence.

\begin{thm}\label{thm:equiv-stable}
The projective $G$-equivariant level model structure on $\cat{$\bm G$-Spectra}$ admits a Bousfield localization with fibrant objects precisely the projectively level fibrant $G$-$\Omega$-spectra. Similarly, the flat $G$-equivariant level model structure admits a Bousfield localization with fibrant objects the flatly level fibrant $G$-$\Omega$-spectra. Both of these model structures are combinatorial, and they have the same weak equivalences, which we call the \emph{$G$-equivariant weak equivalences}.
\begin{proof}
See \cite[Theorems~4.7 and~4.8]{hausmann-equivariant}.
\end{proof}
\end{thm}

\begin{rk}\label{rk:equivariant-proj-gen-cof}
Let us say something about the generating acyclic cofibrations of the above model structure, see \cite[Example~2.37 and discussion after Theorem~4.8]{hausmann-equivariant}: for any finite set $A$, the spectrum $\bm\Sigma(A,\blank)$ corepresents evaluation at $A$ by the enriched Yoneda lemma; similarly $S^B\smashp\bm\Sigma(A\amalg B,\blank)$ corepresents $X\mapsto\Omega^B X(A\amalg B)$. By another application of the Yoneda Lemma we therefore get a map $\lambda_{A,B}\colon S^B\smashp\bm\Sigma(A\amalg B,\blank)\to\bm\Sigma(A,\blank)$ such that $\maps(\lambda_{A,B},X)$ agrees up to conjugation by natural isomorphisms with the adjoint structure map $X(A)\to\Omega^B X(A\amalg B)$ for any symmetric spectrum $X$.

If now $H\subset G$ acts on $A,B$, then $\lambda_{A,B}$ becomes a map of $H$-spectra (denoted $\lambda_{H,A,B}$) with respect to the induced actions, and we factor it as a projective cofibration $\kappa_{H,A,B}$ followed by a level weak equivalence $\rho_{H,A,B}$. Then a set of generating acyclic cofibrations is given by taking a set of generating acyclic \emph{level} cofibrations and adding the pushout product maps $(G_+\smashp_H\kappa_{H,A,B})\ppo i$ for all $H\subset G$, all finite $H$-sets $A,B$ (up to isomorphism), and all generating cofibrations $i$.
\end{rk}

\begin{rk}\label{rk:equivariant-stabilization}
As a teaser for the things to come, we recall that any combinatorial simplicial (left) proper model category admits a \emph{stabilization} \cite{schwede-stab}. We remark without proof that on associated $\infty$-categories this models the universal stabilization in the sense of \cite[Corollary~1.4.4.5]{higher-algebra}, i.e.~the initial example of an adjunction to a presentable stable $\infty$-category, or equivalently the result of universally inverting $\Sigma\mathrel{:=} S^1\smashp\blank$ in the presentable world.

However, while we have a Quillen adjunction $\Sigma^\infty_+\colon\cat{$\bm G$-SSet}\rightleftarrows\cat{$\bm G$-Spectra}:\!\Omega^\infty$ for either of the above model structures, as already mentioned in the introduction this does \emph{not} model the stabilization in the above na\"ive sense. Rather, this defines a `genuine stabilization' universally inverting the functor $S^G\smashp\blank$ (in presentable $\infty$-categories), where $G$ acts via permuting the smash functors of $S^G=\bigwedge_GS^1$, or equivalently the functors $S^A\smashp\blank$ for all finite $G$-sets $A$, see~\cite[Appendix C]{gepner-meier} or \cite[Theorem~A.2]{clausen-mathew-naumann-noel}.
\end{rk}

The main disadvantage of the approach via symmetric spectra is that the weak equivalences are only indirectly defined in terms of specifying the local objects. However, there is a notion of \emph{$\ul\pi_*$-isomorphism}, which we will now introduce, that while not accounting for all $G$-equivariant weak equivalences is at least coarse enough for many purposes:

\begin{constr}
Let $H\subset G$ and let $\mathcal U_H$ be a complete $H$-set universe, i.e.~a countable $H$-set into which any other countable $H$-set embeds equivariantly, and write $s(\mathcal U_H)$ for the poset of finite $H$-subsets of $\mathcal U_H$.

For every $G$-spectrum $X$ and every $k\ge0$ we then define
\begin{equation*}
\pi_k^H(X)=\colim_{A\in s(\mathcal U_H)} [S^{A\amalg\{1,\dots,k\}}, |X(A)|]^H_*
\end{equation*}
where $[\,{,}\,]^H_*$ denotes the set of $H$-equivariant based homotopy classes (for maps of $H$-topological spaces) and the transition maps are given by
\begin{align*}
[S^{A\amalg\{1,\dots,k\}}, |X(A)|]^H_*&\xrightarrow{S^{B\setminus A}\smashp\blank}
[S^{B\setminus A}\smashp S^{A\amalg\{1,\dots,k\}}, S^{B\setminus A}\smashp|X(A)|]^H_*\\
&\cong[S^{B\amalg\{1,\dots,k\}},S^{B\setminus A}\smashp |X(A)|]^H_*\\
&\xrightarrow{\sigma} [S^{B\amalg\{1,\dots,k\}},|X(B)|]^H_*
\end{align*}
for all $A\subset B$, where $\sigma$ denotes the structure map of the symmetric spectrum and the unlabelled isomorphism is the canonical one. Similarly, for $k<0$ we define
\begin{equation*}
\pi_k^H(X)=\colim_{A\in s(\mathcal U_H)}[S^A, |X(A\amalg\{1,\dots,-k\})|]^H_*
\end{equation*}
with the analogously defined transition maps.

For every $k\in\mathbb Z$ and $H\subset G$, $\pi_k^HX$ is naturally an abelian group \cite[Definition~3.1]{hausmann-equivariant}; however, we will not need this group structure below.
\end{constr}

\begin{defi}
A map $f\colon X\to Y$ of $G$-spectra is called a \emph{($G$-equivariant) $\ul\pi_*$-isomorphism} if $\pi_k^Hf$ is an isomorphism for all $H\subset G$ and all $k\in\mathbb Z$.
\end{defi}

\begin{rk}
The above homotopy groups are independent of the choice of $\mathcal U_H$ up to natural, but in general non-canonical isomorphism \cite[3.3]{hausmann-equivariant}. In particular, the notion of $\ul\pi_*$-isomorphism is independent of any choices.
\end{rk}

\begin{thm}
Every $\ul\pi_*$-isomorphism of $G$-spectra is a $G$-equivariant weak equivalence.
\begin{proof}
See \cite[Theorem~3.36]{hausmann-equivariant}
\end{proof}
\end{thm}

\begin{warn}
While restriction along \emph{injective} homomorphisms preserves all of the above structure \cite[5.2]{hausmann-equivariant}, the equivariant model structures do not interact reasonably with restrictions along \emph{general} homomorphisms, unlike their unstable siblings. In particular, if $\alpha\colon G\to G'$ is not injective, then $\alpha^*$ will typically \emph{not} send $G'$-equivariant weak equivalences to $G$-equivariant ones.

One nice property of the $G$-global theory we will introduce in the following two subsections is that it comes with homotopically meaningful `change of group' adjunctions (which we will later formalize in the notion of a \emph{global model category}), and in particular that restriction along arbitrary homomorphisms will indeed be homotopical.
\end{warn}

\subsection{Unstable \texorpdfstring{$\bm G$}{G}-global homotopy theory}
Let $G$ continue to denote a finite group. We will now recall \emph{$G$-global homotopy theory} in the sense of \cite{g-global}; this generalizes global homotopy theory (for finite groups) in the sense of Schwede \cite{schwede-book}, and will be the key tool in this article to express and prove properties of the latter. Again, we begin with the unstable story:

\begin{constr}\label{constr:indiscrete}
The forgetful functor $\cat{SSet}\to\cat{Set}$ sending a simplicial set $X$ to its set of vertices admits a right adjoint $E$, given explicitly by $(EX)_n=X^{1+n}$ with the evident functoriality in $X$ and with structure maps induced by the canonical identification $X^{1+n}\cong\Hom_{\cat{Set}}([n], X)$.
\end{constr}

\begin{defi}
We write $I$ for the category of finite sets and injective maps, and we let $\mathcal I$ denote the simplicial category obtained by applying $E$ to all hom sets. We write $\cat{$\bm{\mathcal I}$-SSet}$ for the enriched category of enriched functors $\mathcal I\to\cat{SSet}$ and call its objects \emph{$\mathcal I$-spaces} or \emph{global spaces}. More generally, we write $\cat{$\bm G$-$\bm{\mathcal I}$-SSet}$ for the category of $G$-objects in $\cat{$\bm{\mathcal I}$-SSet}$ and call its objects \emph{$G$-$\mathcal I$-spaces} or \emph{$G$-global spaces}.
\end{defi}

\begin{prop}\label{prop:I-g-glob-level}
There is a unique model structure on $\cat{$\bm G$-$\bm{\mathcal I}$-SSet}$ in which a map is a weak equivalence or fibration if and only if $f(A)$ is a weak equivalence or fibration, respectively, in the $\mathcal G_{\Sigma_A,G}$-equivariant model structure for every finite set $A$. We call this the \emph{$G$-global level model structure} and its weak equivalences the \emph{$G$-global level weak equivalences}. This model structure is proper, simplicial, and combinatorial with generating cofibrations
\begin{equation*}
\{ \mathcal I(A,\blank)\times_\phi G\times (\del\Delta^n\hookrightarrow\Delta^n) : H\subset\Sigma_A, \phi\colon H\to G,n\ge0\},
\end{equation*}
where $\times_\phi$ denotes the quotient of the ordinary product by the diagonal of the right action of $H$ on $\mathcal I(A,\blank)$ via its tautological action on $A$ and the right action on $G$ via $g.h=g\phi(h)$.

Finally, filtered colimits of $G$-global level weak equivalences are again $G$-global level weak equivalences.
\begin{proof}
See \cite[Proposition~1.4.3]{g-global}.
\end{proof}
\end{prop}

Just like for $G$-symmetric spectra, we will now Bousfield localize this to get the model structure we are actually after. However, unlike for $G$-symmetric spectra, we can actually explicitly describe both the weak equivalences and the local objects. We start with the former:

\begin{constr}
Let $A$ be any set, possibly infinite, and let $X$ be an $\mathcal I$-simplicial set. Then we define
\begin{equation*}
X(A)\mathrel{:=}\colim\limits_{B\subset A\text{ finite}} X(B)
\end{equation*}
with transition maps induced by functoriality of $X$. This becomes a functor in $X$ in the obvious way; in particular, if $G$ acts on $X$, then $X(A)$ becomes naturally a $G$-simplicial set.

In addition, the monoid $\textup{End}(A)$ of self-maps of $A$ acts naturally on the above by permuting the terms of the colimit. Thus, if $A$ is an $H$-set and $X$ is a $G$-$\mathcal I$-simplicial set, then $X(A)$ becomes an $(H\times G)$-simplicial set.
\end{constr}

\begin{defi}
A map $f\colon X\to Y$ of $G$-$\mathcal I$-simplicial sets is called a \emph{$G$-global weak equivalence} if for every finite group $H$ and some (hence any) complete $H$-set universe $\mathcal U_H$ the induced map $f(\mathcal U_H)$ is a $\mathcal G_{H,G}$-equivariant weak equivalence, or equivalently (replacing $H$ by a subgroup if necessary) for every $\phi\colon H\to G$ the map $(\phi^* f)(\mathcal U_H)$ is an $H$-equivariant weak equivalence.
\end{defi}

Next, we come to the analogue of the notion of an $\Omega$-spectrum in this setting:

\begin{defi}
A $G$-$\mathcal I$-simplicial set is called \emph{static} if for every finite group $H$ and all finite \emph{faithful} $H$-sets $A\subset B$ the map $X(A)\to X(B)$ induced by the inclusion is a $\mathcal G_{H,G}$-weak equivalence.
\end{defi}

\begin{thm}\label{thm:I-G-glob}
The $G$-global level model structure on $\cat{$\bm G$-$\bm{\mathcal I}$-SSet}$ admits a Bousfield localization with weak equivalences the $G$-global weak equivalences. Its fibrant objects are precisely the level fibrant \emph{static} $G$-$\mathcal I$-simplicial sets.

This model structure is again combinatorial (with the same generating cofibrations), simplicial, proper, and filtered colimits in it are homotopical.
\begin{proof}
See \cite[Theorem~1.4.30]{g-global}.
\end{proof}
\end{thm}

We will also need the following injective variant of the above model structure:

\begin{thm}\label{thm:I-inj-G-glob}
There is a unique model structure on $\cat{$\bm G$-$\bm{\mathcal I}$-SSet}$ with weak equivalences the $G$-global weak equivalences and cofibrations the injective cofibrations. We call this the \emph{injective $G$-global model structure}. It is combinatorial, simplicial, proper, and filtered colimits in it are homotopical.
\begin{proof}
See~\cite[Theorem~1.4.37]{g-global}.
\end{proof}
\end{thm}

\begin{rk}
The category $\cat{$\bm G$-$\bm I$-SSet}$ of $G$-objects in $\Fun(I,\cat{SSet})$ also carries a $G$-global level model structure analogous to Proposition~\ref{prop:I-g-glob-level} and this admits a Bousfield localization with the static objects as local objects such that the resulting model category is Quillen equivalent to $\cat{$\bm G$-$\bm{\mathcal I}$-SSet}$, see~\cite[Theorem~1.4.31]{g-global}. However, the weak equivalences of this model structure are somewhat complicated (similarly to the situation for symmetric spectra), and in particular they cannot just be checked by evaluation at complete $H$-set universes. The passage from $I$ to $\mathcal I$ is precisely what eliminates this subtlety, which is why the above model will be more convenient for us.

In addition to these, \cite[Chapter~1]{g-global} also studies various models of $G$-global homotopy theory based on a certain monoid $\mathcal M$ and the simplicial monoid $E\mathcal M$ obtained via Construction~\ref{constr:indiscrete} from this, that are related to the above via (zig-zags of) Quillen equivalences.
\end{rk}

As promised, the above $G$-global model structures support a rich `change of group' calculus:

\begin{prop}\label{prop:I-functoriality-general}
Let $\alpha\colon G\to G'$ be any group homomorphism. Then the restriction $\alpha^*\colon\cat{$\bm{G'}$-$\bm{\mathcal I}$-SSet}\to\cat{$\bm G$-$\bm{\mathcal I}$-SSet}$ is homotopical and we have Quillen adjunctions
\begin{align*}
\alpha_!\colon\cat{$\bm{G}$-$\bm{\mathcal I}$-SSet}_\textup{$G$-global}&\rightleftarrows\cat{$\bm{G'}$-$\bm{\mathcal I}$-SSet}_\textup{$G'$-global} :\!\alpha^*\\
\alpha^*\colon\cat{$\bm{G'}$-$\bm{\mathcal I}$-SSet}_\textup{injective $G'$-global}&\rightleftarrows\cat{$\bm{G}$-$\bm{\mathcal I}$-SSet}_\textup{injective $G$-global} :\!\alpha_*.
\end{align*}
\begin{proof}
See \cite[Lemma~1.4.40 and Corollary~1.4.41]{g-global}.
\end{proof}
\end{prop}

\begin{prop}\label{prop:I-functoriality-injective}
Let $\alpha\colon G\to G'$ be an \emph{injective} homomorphism. Then also
\begin{align*}
\alpha_!\colon\cat{$\bm{G}$-$\bm{\mathcal I}$-SSet}_\textup{injective $G$-global}&\rightleftarrows\cat{$\bm{G'}$-$\bm{\mathcal I}$-SSet}_\textup{injective $G'$-global} :\!\alpha^*\\
\alpha^*\colon\cat{$\bm{G'}$-$\bm{\mathcal I}$-SSet}_\textup{$G'$-global}&\rightleftarrows\cat{$\bm{G}$-$\bm{\mathcal I}$-SSet}_\textup{$G$-global} :\!\alpha_*.
\end{align*}
are Quillen adjunctions.
\begin{proof}
See \cite[Lemmas~1.4.42 and~1.4.43]{g-global}.
\end{proof}
\end{prop}

As every $G$-$\mathcal I$-simplicial set is injectively cofibrant, the above implies via Ken Brown's Lemma that $\alpha_!$ is homotopical for \emph{injective} $\alpha$. The following generalization of this (which makes precise that `free quotients are homotopical') will be a key input in many arguments, see in particular~Example~\ref{ex:S-Sp-global}.

\begin{prop}\label{prop:I-free-quotients}
Let $\alpha\colon G\to G'$ be any homomorphism and let $f\colon X\to Y$ be a $G$-global weak equivalence such that $\ker(\alpha)$ acts levelwise freely on $X$ and $Y$. Then $\alpha_!f$ is a $G'$-global level weak equivalence.
\begin{proof}
We factor $f$ in the $G$-global model structure as an acyclic cofibration $j\colon X\to Z$ followed by a fibration $p$ (automatically acyclic). Then $\alpha_!j$ is a $G'$-global weak equivalence by Proposition~\ref{prop:I-functoriality-general}, so it suffices to show that also $\alpha_!p$ is a weak equivalence; we will show that it is even a $G'$-global level weak equivalence.

To this end, we observe that for any generating cofibration $i$ and every finite set $A$ the map $i(A)$ is a cofibration in the $\mathcal G_{\Sigma_A,G}$-equivariant model structure since $G$ acts freely on $\mathcal I(B,A)\times_\phi G$ for every finite faithful $H$-set $B$ and homomorphism $\phi\colon H\to G$. As evaluation at $A$ is cocontinuous, we see that the claim holds more generally for all cofibrations, and in particular for the above map $j$. Thus, $\ker(\alpha)$ also acts levelwise freely on $Z$; the claim therefore follows by applying \cite[Proposition~1.1.22]{g-global} levelwise (with $M=\Sigma_A$ and $\mathcal E=\mathcal A\ell\ell$).
\end{proof}
\end{prop}

\subsection{Stable \texorpdfstring{$\bm G$}{G}-global homotopy theory} Finally, we come to models of \emph{stable} $G$-global homotopy theory \cite[Chapter~3]{g-global}; we restrict ourselves to the basics here and will recall further constructions and results (in particular monoidal properties and the tensoring over $G$-global spaces) later when needed.

\subsubsection{Model structures} On the pointset level, our models will again be simply given by symmetric spectra with a $G$-action, and we once more start with suitable level model structures \cite[Propositions~3.1.20 and~3.1.23]{g-global}:

\begin{prop}
There is a unique model structure on $\cat{$\bm G$-Spectra}$ in which a map is a weak equivalence or fibration if and only if $f(A)$ is a $\mathcal G_{\Sigma_A,G}$-weak equivalence or fibration, respectively, for every $A\in\bm\Sigma$. We call this the \emph{$G$-global projective level model structure} and its weak equivalences the \emph{$G$-global level weak equivalences}. It is proper, simplicial, combinatorial with generating cofibrations
\begin{equation*}
\{\bm\Sigma(A,\blank)\smashp_H G_+\smashp(\del\Delta^n\hookrightarrow\Delta^n)_+ : A\in\bm\Sigma, H\in\mathcal G_{\Sigma_A,G},n\ge0\},
\end{equation*}
and filtered colimits in it are homotopical.\qed
\end{prop}

\begin{prop}
There is a unique model structure on $\cat{$\bm G$-Spectra}$ in which a map is a weak equivalence or fibration if and only if $f(A)$ is a weak equivalence or fibration, respectively, in the \emph{injective} $\mathcal G_{\Sigma_A,G}$-model structure. We call this the \emph{$G$-global flat level model structure}; its weak equivalences are precisely the $G$-global level weak equivalences and its cofibrations are the flat cofibrations. This model structure is proper, simplicial, combinatorial with generating cofibrations
\begin{equation*}
\{\bm\Sigma(A,\blank)\smashp_H G_+\smashp(\del\Delta^n\hookrightarrow\Delta^n)_+ : A\in\bm\Sigma, H\subset \Sigma_A\times G,n\ge0\},
\end{equation*}
and filtered colimits in it are homotopical.\qed
\end{prop}

\begin{warn}
Beware that the notion of \emph{$G$-global level weak equivalence} differs from the \emph{$G$-equivariant level weak equivalences}: the former is a condition on the $H$-fixed points for $H\in\mathcal G_{\Sigma_A,G}$ for varying $A$, while the latter is a condition for $H\in\mathcal G_{G,\Sigma_A}$.
\end{warn}

\begin{defi}
A $G$-spectrum $X$ is called a \emph{$G$-global $\Omega$-spectrum} if for all finite groups $H$ and all finite \emph{faithful} $H$-sets $A\subset B$ the derived adjoint structure map
\begin{equation*}
X(A)\to\cat{R}\Omega^{B\setminus A}X(B)
\end{equation*}
is a $\mathcal G_{H,G}$-weak equivalence.
\end{defi}

Again, for a $G$-spectrum that is fibrant in either of the above level model structures, the derived adjoint structure map is already modelled by the ordinary one.

\begin{defi}
A map $f$ in $\cat{$\bm G$-Spectra}$ is called a \emph{$G$-global weak equivalence} if $\phi^*f$ is an $H$-equivariant weak equivalence for every finite group $H$ and every homomorphism $\phi\colon H\to G$.
\end{defi}

\begin{thm}
The $G$-global projective level model structure admits a Bousfield localization with weak equivalences the $G$-global weak equivalences. We call this the \emph{$G$-global projective model structure}; its fibrant objects are precisely those $G$-global $\Omega$-spectra that are fibrant in the $G$-global projective level model structure.

This model structure is again combinatorial (with the same generating cofibrations as before), simplicial, proper, and filtered colimits in it are homotopical.
\begin{proof}
See \cite[Theorem~3.1.41 and Proposition~3.1.47]{g-global}.
\end{proof}
\end{thm}

\begin{thm}
The $G$-global flat level model structure admits a Bousfield localization with weak equivalences the $G$-global weak equivalences. We call this the \emph{$G$-global flat model structure}; its fibrant objects are precisely those $G$-global $\Omega$-spectra that are fibrant in the $G$-global flat level model structure.

This model structure is again combinatorial (with the same generating cofibrations as before), simplicial, proper, and filtered colimits in it are homotopical.
\begin{proof}
See \cite[Theorem~3.1.40 and Proposition~3.1.47]{g-global}.
\end{proof}
\end{thm}

\begin{rk}
For $G=1$ the above two model structures agree and recover Hausmann's \emph{global model structure} \cite[Theorem~2.18]{hausmann-global}.
\end{rk}

Again, there is also an injective version of the above model structures:

\begin{thm}
There is a unique model structure on $\cat{$\bm G$-Spectra}$ with weak equivalences the $G$-global weak equivalences and cofibrations the injective cofibrations. We call this the \emph{$G$-global injective model structure}. It is combinatorial, simplicial, proper, and filtered colimits in it are homotopical.
\begin{proof}
See \cite[Corollary~3.1.46]{g-global}.
\end{proof}
\end{thm}

\subsubsection{Change of group adjunctions} As promised (and unlike their equivariant counterparts), these model structures behave nicely under changing the group:

\begin{prop}\label{prop:sp-functoriality-general}
Let $\alpha\colon G\to G'$ be any homomorphism. Then the restriction $\alpha^*\colon\cat{$\bm{G'}$-Spectra}\to\cat{$\bm G$-Spectra}$ is homotopical and we have Quillen adjunctions
\begin{align*}
\alpha_!\colon\cat{$\bm{G}$-Spectra}_\textup{$G$-gl.~proj.}&\rightleftarrows\cat{$\bm{G'}$-Spectra}_\textup{$G'$-gl.~proj.} :\!\alpha^*\\
\alpha^*\colon\cat{$\bm{G'}$-Spectra}_\textup{$G'$-gl.~flat} &\rightleftarrows\cat{$\bm{G}$-Spectra}_\textup{$G$-gl.~flat} :\!\alpha_*\\
\alpha^*\colon\cat{$\bm{G'}$-Spectra}_\textup{$G'$-gl.~inj.}&\rightleftarrows\cat{$\bm{G}$-Spectra}_\textup{$G$-gl.~inj.} :\!\alpha_*.
\end{align*}
\begin{proof}
Everything except for the statement about the injective model structures appears in \cite[Lemmas~3.1.49 and~3.1.50]{g-global}. For the final statement, it then only remains to show that $\alpha^*$ preserves injective cofibrations, which is immediate from the definition.
\end{proof}
\end{prop}

\begin{prop}\label{prop:sp-functoriality-injective}
Let $\alpha\colon G\to G'$ be an \emph{injective} homomorphism. Then we also have Quillen adjunctions
\begin{align*}
\alpha^*\colon\cat{$\bm{G'}$-Spectra}_\textup{$G'$-gl.~proj.}&\rightleftarrows\cat{$\bm{G}$-Spectra}_\textup{$G$-gl.~proj.} :\!\alpha_*\\
\alpha_!\colon\cat{$\bm{G}$-Spectra}_\textup{$G$-gl.~flat} &\rightleftarrows\cat{$\bm{G'}$-Spectra}_\textup{$G'$-gl.~flat} :\!\alpha^*\\
\alpha_!\colon\cat{$\bm{G}$-Spectra}_\textup{$G$-gl.~inj.}&\rightleftarrows\cat{$\bm{G'}$-Spectra}_\textup{$G'$-gl.~inj.} :\!\alpha^*.
\end{align*}
\begin{proof}
The latter two statements are \cite[Propositions~3.1.52 and~3.1.53]{g-global}. For the first statement it then only remains (as $\alpha^*$ is homotopical) that $\alpha_*$ sends acyclic fibrations of the $G$-global projective (level) model structure to acyclic fibrations in the $G$-global projective (level) model structure. Using that acyclic fibrations are defined levelwise and adjoining, this amounts to saying that
\begin{equation*}
(\Sigma_A\times\alpha)^*\colon\cat{$\bm{(\Sigma_A\times G')}$-SSet}_{\mathcal G_{\Sigma_A,G'}}\to \cat{$\bm{(\Sigma_A\times G)}$-SSet}_{\mathcal G_{\Sigma_A,G}}
\end{equation*}
preserves cofibrations for every finite set $A$. This is immediate from Lemma~\ref{lemma:graph-target}.
\end{proof}
\end{prop}

Again, suitably free quotients are homotopical in our setting:

\begin{prop}\label{prop:free-quotient-spectra}
Let $\alpha\colon G\to G'$ be any homomorphism, and let $f\colon X\to Y$ be a $G$-global weak equivalence in $\cat{$\bm G$-Spectra}$ such that $\ker(\alpha)$ acts levelwise freely on $X$ and $Y$ outside the basepoint. Then $\alpha_!f$ is a $G'$-global weak equivalence.
\begin{proof}
See \cite[Proposition~3.1.54]{g-global}.
\end{proof}
\end{prop}

\subsubsection{The smash product} The usual smash product of symmetric spectra gives us a smash product on $\cat{$\bm G$-Spectra}$ by pulling through the $G$-actions. This is compatible with the above model structures:

\begin{thm}\label{thm:smash-g-global}
The smash product defines left Quillen bifunctors
\begin{align*}
\cat{$\bm G$-Spectra}_\textup{$G$-global flat}\times\cat{$\bm G$-Spectra}_\textup{$G$-global flat}&\to\cat{$\bm G$-Spectra}_\textup{$G$-global flat}\\
\cat{$\bm G$-Spectra}_\textup{$G$-global proj.}\times\cat{$\bm G$-Spectra}_\textup{$G$-global flat}&\to\cat{$\bm G$-Spectra}_\textup{$G$-global proj.}
\end{align*}
\end{thm}
Note that in the second adjunction, indeed only \emph{one} of the input factors is equipped with the projective model structure (the corresponding statement where both factors are equipped with the projective model structures follows immediately).
\begin{proof}
See \cite[Propositions~3.1.63 and~3.1.64]{g-global}.
\end{proof}

In particular, smashing with a fixed flat $G$-spectrum is left Quillen for either of the above model structures, so it preserves weak equivalences between cofibrant objects by Ken Brown's Lemma. The following (easy) $G$-global analogue of the equivariant \emph{Flatness Theorem} \cite[Proposition~6.2]{hausmann-equivariant} strengthens this result:

\begin{prop}\label{prop:flatness-theorem}
\begin{enumerate}
\item Let $X$ be any $G$-spectrum. Then $X\smashp\blank$ preserves $G$-global weak equivalences between \emph{flat} $G$-spectra.
\item Let $X$ be a \emph{flat} $G$-spectrum. Then $X\smashp\blank$ preserves $G$-global weak equivalences.
\end{enumerate}
\begin{proof}
See~\cite[Proposition~3.1.62]{g-global}.
\end{proof}
\end{prop}

\subsubsection{Relation to stable equivariant homotopy theory} On the pointset level, $G$-global and $G$-equivariant spectra are the same objects, and every $G$-global weak equivalence is in particular a $G$-equivariant weak equivalence. Thus, the identity of $\cat{$\bm G$-Spectra}$ descends to exhibit the $G$-equivariant stable homotopy category as a localization of the $G$-global one. This localization admits both adjoints (fully faithful for formal reasons), which has the following model categorical manifestation:

\begin{prop}\label{prop:equivariant-vs-global}
The adjunctions
\begin{align*}
\id\colon\cat{$\bm G$-Spectra}_\textup{$G$-equivariant projective}&\rightleftarrows\cat{$\bm G$-Spectra}_\textup{$G$-global flat} :\!\id\\
\id\colon\cat{$\bm G$-Spectra}_\textup{$G$-global flat}&\rightleftarrows\cat{$\bm G$-Spectra}_\textup{$G$-equivariant flat} :\!\id
\end{align*}
are Quillen adjunctions.
\begin{proof}
See \cite[Proposition~3.3.1 and Corollary~3.3.3]{g-global}.
\end{proof}
\end{prop}

\subsubsection{Suspension spectra} Finally, we come to the relation between unstable and stable $G$-global homotopy theory:

\begin{constr}
Let $X$ be an $\mathcal I$-simplicial set (or an $I$-simplicial set). We define a symmetric spectrum $\Sigma^\bullet_+X$ via $(\Sigma^\bullet_+X)(A) = S^A\smashp X(A)_+$; if $i\colon A\to B$ is an injection of finite sets, then the structure map is given by
\begin{equation*}
S^{B\setminus i(A)}\smashp (S^A\smashp X(A)_+)\cong S^B\smashp X(A)_+\xrightarrow{S^B\smashp X(i)_+} S^B\smashp X(B)_+
\end{equation*}
where the unlabelled isomorphism is induced by $i$. This becomes an enriched functor in $X$ as follows: given a map $f\colon X\times\Delta^n\to Y$ of $\mathcal I$-simplicial sets (or $I$-simplicial sets), we have a map $(\Sigma^\bullet_+ X)\smashp\Delta^n_+\to\Sigma^\bullet_+ Y$ of symmetric spectra given in degree $A$ by
\begin{equation*}
	(S^A\smashp X(A)_+)\smashp\Delta^n_+\cong S^A\smashp (X(A)\times\Delta^n)_+\xrightarrow{S^A\smashp f_A} S^A\smashp Y(A)_+
\end{equation*}
where the unlabelled isomorphism is the obvious one. We can then lift this to a functor
\begin{equation}\label{eq:suspension}
\Sigma^\bullet_+\colon\cat{$\bm G$-$\bm{\mathcal I}$-SSet}\to\cat{$\bm G$-Spectra}
\end{equation}
by pulling through the $G$-actions.
\end{constr}

\begin{prop}\label{prop:suspension-loop-G-gl}
The functor $(\ref{eq:suspension})$ admits a simplicial right adjoint $\Omega^\bullet$. We have Quillen adjunctions
\begin{align*}
\Sigma^\bullet_+\colon\cat{$\bm G$-$\bm{\mathcal I}$-SSet}_\textup{$G$-gl.~proj.}&\rightleftarrows\cat{$\bm G$-Spectra}_\textup{$G$-gl.~proj.} :\!\Omega^\bullet\\
\Sigma^\bullet_+\colon\cat{$\bm G$-$\bm{\mathcal I}$-SSet}_\textup{$G$-gl.~inj.}&\rightleftarrows\cat{$\bm G$-Spectra}_\textup{$G$-gl.~inj.} :\!\Omega^\bullet,
\end{align*}
and in particular $\Sigma^\bullet_+$ is homotopical.
\end{prop}
Beware that \cite{g-global} uses $\Omega^\bullet$ for the corresponding right adjoint in $I$-simplicial sets instead and introduces more complicated notation for the above right adjoint in $\mathcal I$-simplicial sets. As we will only need the latter, we have decided to change notation here.
\begin{proof}
See \cite[Corollary~3.2.6 and Remark~3.2.7]{g-global}.
\end{proof}

\begin{rk}\label{rk:Omega-bullet-on-fibrant}
We briefly remark on the above right adjoint. As a functor from $\cat{$\bm I$-SSet}$, $\Sigma^\bullet_+$ has a left adjoint $\omega^\bullet$ defined via $(\omega^\bullet X)(A)=\Omega^AX(A)$ with the evident functoriality in each variable. The functor $\Omega^\bullet$ is accordingly obtained by postcomposing $\omega^\bullet$ with the right adjoint $\cat{$\bm I$-SSet}\to\cat{$\bm{\mathcal I}$-SSet}$ of the forgetful functor, and this as usual gives the right adjoint for general $G$ by pulling through the action. As the forgetful functor $\cat{$\bm G$-$\bm{\mathcal I}$-SSet}\to\cat{$\bm G$-$\bm I$-SSet}$ is the left half of a Quillen equivalence for the projective $G$-global model structures \cite[Theorem~1.4.49]{g-global}, we obtain a natural $G$-global level weak equivalence between the restriction of
\begin{equation*}
\cat{$\bm G$-Spectra}\xrightarrow{\Omega^\bullet}\cat{$\bm{G}$-$\bm{\mathcal I}$-SSet}\xrightarrow{\textup{forget}}\cat{$\bm G$-$\bm I$-SSet}
\end{equation*}
to projectively fibrant objects and the corresponding restriction of $\omega^\bullet$; this is all we will need about $\Omega^\bullet$ below.
\end{rk}

\section{Global model categories}\label{sec:global-model-cat}
In this section, we introduce the framework of \emph{global model categories} which will then in particular allow us later to express the universal property of the passage from global spaces to global spectra.

\subsection{Preglobal model categories} We begin by describing a slightly more general notion:

\begin{defi}
A \emph{preglobal model category} consists of a locally presentable category $\mathscr C$, which is enriched, tensored, and cotensored over $\cat{SSet}$, together with two model structures on the category $G\text{-}\mathscr C$ of $G$-objects in $\mathscr C$ for each finite group $G$, called the \emph{projective} and \emph{flat $G$-global model structures}, such that the following conditions are satisfied:
\begin{enumerate}
\item The projective and the flat $G$-global model structure have the same weak equivalences (which we call the \emph{$G$-global weak equivalences}) and the adjunction $\id\colon G\text-\mathscr C_{\text{proj}}\rightleftarrows G\text{-}\mathscr C_\text{flat} :\!\id$ is a Quillen adjunction (i.e.~every projective cofibration is also a flat cofibration, or equivalently every fibration of the flat model structure is also a fibration in the projective one).
\item Both the projective and the flat model structure on $G\text{-}\mathscr C$ are left proper, combinatorial, and simplicial.
\item For every homomorphism $\alpha\colon H\to G$ of finite groups the restriction functor $\alpha^*\colon G\text{-}\mathscr C\to H\text-\mathscr C$ preserves weak equivalences, flat cofibrations, and projective fibrations. Moreover, if $\alpha$ is injective, then $\alpha^*$ also preserves projective cofibrations and flat fibrations.

In particular, $\alpha^*$ is always left Quillen for the flat model structures and right Quillen for the projective ones; if $\alpha$ is injective, then $\alpha^*$ is also right Quillen for the flat model structures and left Quillen for the projective ones.
\end{enumerate}
\end{defi}

\begin{ex}
Let $\mathscr C=\cat{$\bm{\mathcal I}$-SSet}$. For every finite group $G$, we can equip $\cat{$\bm G$-$\bm{\mathcal I}$-SSet}$ with the \emph{$G$-global model structure} and the \emph{injective $G$-global model structure} as projective and flat model structures, respectively, and these are left proper, combinatorial, and simplicial (Theorems~\ref{thm:I-G-glob} and~\ref{thm:I-inj-G-glob}). Propositions~\ref{prop:I-functoriality-general} and~\ref{prop:I-functoriality-injective} then show that this yields a preglobal model category, which we denote by $\glo{GlobalSpaces}$ and call the \emph{preglobal model category of global spaces}.
\end{ex}

\begin{ex}\label{ex:global-spectra}
Let $\mathscr C=\cat{Spectra}$ be the category of symmetric spectra. For every finite $G$, we can equip $\cat{$\bm G$-Spectra}$ with the projective and injective $G$-global model structure. Propositions~\ref{prop:sp-functoriality-general} and~\ref{prop:sp-functoriality-injective} then show that this yields a preglobal model category $\glo{GlobalSpectra}$, which we call the \emph{preglobal model category of global spectra}.
\end{ex}

\begin{rk}
By the same argument we could have used the \emph{flat} instead of the \emph{injective} model structures to yield another preglobal model category (which is the motivation for the above terminology). However, for now the above choices will be more convenient; later, when we consider algebraic structures on global spectra, we will see a return of the flat model structures (or more precisely their `positive' cousins).
\end{rk}

\begin{nex}
	If we instead consider the \emph{$G$-equivariant} projective and flat model structures from Theorem~\ref{thm:equiv-stable}, this does \emph{not} make $\cat{Spectra}$ into a preglobal model category: for example, the restriction $\alpha^*$ along any non-injective homomorphism $\alpha$ does not preserve equivariant stable weak equivalences, and it is not right Quillen for the equivariant projective model structures.

	We emphasize that this should not be seen as some pointset level artifact---instead, it is a model categorical manifestation of the fact that the left derived functor $\cat{L}\alpha^*$ does not admit a left adjoint \cite[Warning~9.8]{clefts}. In the alternative higher categorical framework of \cite{global-param, clefts} this leads to the global $\infty$-category of equivariant spectra not being `globally presentable,' but only `equivariantly presentable.'
\end{nex}

\begin{ex}\label{ex:shifts}
Let $\underline{\mathscr C}$ be a preglobal model category and let $G$ be a finite group. Then we have a preglobal model category $\underline{G\text{-}\mathscr C}$ with underlying category $G\text{-}\mathscr C$ and with the model structures on $G'\text{-}(G\text{-}\mathscr C)$ transported from $(G'\times G)\text{-}\mathscr C$ along the evident isomorphism of categories. In particular, specializing to the previous examples we get preglobal model categories of $G$-global spaces and of $G$-global spectra.
\end{ex}

\begin{ex}\label{ex:exotic}
Finally, we introduce a more exotic example. Let $\mathcal F$ be a \emph{global family}, i.e.~a non-empty collection of finite groups closed under subquotients (hence in particular under isomorphisms). We make $\mathscr E\mathrel{:=}\cat{SSet}$ into a preglobal model category as follows: for every finite group $G$, we write $G\cap\mathcal F$ for the family of subgroups of $G$ that belong to $\mathcal F$, and we equip $\cat{$\bm G$-SSet}$ with the $(G\cap\mathcal F)$-model structure and the injective $(G\cap\mathcal F)$-model structure, respectively.

If now $\alpha\colon G\to G'$ is any group homomorphism, then $\alpha^*$ clearly preserves injective cofibrations, and it preserves weak equivalences and projective fibrations as $(\alpha^*X)^H=X^{\alpha(H)}$ and $\alpha(H)$ is a quotient of $H$. Moreover, if $\alpha$ is injective, then $\alpha^*$ also preserves projective cofibrations as for any $G'$-simplicial set $X$ and any simplex $x$ the isotropy groups satisfy $\textup{Iso}_{\alpha^*X}(x)=\alpha^{-1}(\textup{Iso}_X(x))$, which is isomorphic to a subgroup of $\textup{Iso}_X(x)$. Finally, $\alpha_!$ always preserves injective cofibrations (as quotients by group actions preserve injections of sets), and if $\alpha$ is injective, then this moreover sends $(G\cap\mathcal F)$-weak equivalences to $(G'\cap\mathcal F)$-weak equivalences by \cite[Proposition~1.1.18]{g-global} (for $M=1$), so it is left Quillen for the injective equivariant model structures as claimed.

We call the resulting preglobal model category $\ul{\mathscr E}_{\mathcal F}$ the \emph{preglobal model category of $\mathcal F$-equivariant spaces} (the reader is invited to choose for themselves whether $\mathscr E$ is an abbreviation for `equivariant' or `exotic').
\end{ex}

The simplicial enrichment, tensoring, and cotensoring of $\mathscr C$ make $G\text{-}\mathscr C$ not only enriched, tensored, and cotensored over $\cat{SSet}$ but over all of $\cat{$\bm G$-SSet}$: the tensoring and cotensoring are just given by equipping the non-equivariant tensoring or cotensoring with the diagonal and conjugation $G$-action, and the enrichment is given by taking the mapping space in $\mathscr C$ (i.e.~without regards to $G$-actions) and equipping it with the conjugation $G$-action. In the same way, we obtain for any finite group $H$ a functor $\blank\times\blank\colon\cat{$\bm{(G\times H)}$-SSet}\times G\text-\mathscr C\to (G\times H)\text-\mathscr C$ generalizing the tensoring.

\begin{lemma}\label{lemma:g-sset-enriched}\label{lemma:tensoring-generalized}
Let $G,H$ be finite groups. Then the above functors
\begin{align*}
\cat{$\bm{(G\times H)}$-SSet}_{\mathcal G_{G,H}}\times G\text-\mathscr C_\textup{proj}&\to (G\times H)\text-\mathscr C_\textup{proj}\\
\cat{$\bm{(G\times H)}$-SSet}\times G\text-\mathscr C_\textup{flat}&\to (G\times H)\text-\mathscr C_\textup{flat}
\end{align*}
are left Quillen bifunctors.

In particular, specializing to $H=1$, both the projective and flat model structure on $G\text{-}\mathscr C$ are enriched in the model categorical sense over $\cat{$\bm G$-SSet}$ with the usual equivariant model structure.
\begin{proof}
Let us consider the case of the projective model structures first. By adjointness, it is enough to show that
\begin{align*}
G\text-\mathscr C^\op_\textup{proj}\times (G\times H)\text-\mathscr C_\textup{proj}&\to\cat{$\bm{(G\times H)}$-SSet}_{\mathcal G_{G,H}}\\
(X,Y)&\mapsto \maps(\textup{triv}_H X,Y)
\end{align*}
is a right Quillen bifunctor. By definition of the model structure on the right this amounts to saying that for every $K\subset G$ and $\phi\colon K\to H$ the functor $(X,Y)\mapsto \maps^K(i^*X,(i,\phi)^*Y)$ to $\cat{SSet}$ is a right Quillen bifunctor, where $i\colon K\hookrightarrow G$ is the inclusion. However, as $(i,\phi)^*$ is right Quillen and $i^*$ is left Quillen for the projective model structures, this follows at once from the fact that $K\text{-}\mathscr C_\textup{proj}$ is a simplicial model category.

For the flat model structures, we observe that $\text{triv}_H\colon G\text-\mathscr C_\textup{flat}\to (G\times H)\text-\mathscr C_\textup{flat}$ is left Quillen. Replacing $G$ by $G\times H$ if necessary, we may therefore assume without loss of generality that $H=1$. However, in this case we are similarly reduced by adjointness to proving that $X,Y\mapsto\maps^K(X,Y)$ is a right Quillen bifunctor for any subgroup $K\subset G$. This in turns follows again from the fact that the restriction $G\text-\mathscr C_\textup{flat}\to K\text-\mathscr C_\text{flat}$ is both left and right Quillen and that $K\text-\mathscr C_\textup{flat}$ is simplicial.
\end{proof}
\end{lemma}

\subsection{Global model categories} In order to support a good theory of \emph{genuine stabilizations} we will need one extra condition in addition to the axioms of a preglobal model category. To motivate this we first recall:

\begin{lemma}\label{lemma:Beck-Chevalley}
Let
\begin{equation}\label{diag:pb-surj}
\begin{tikzcd}
A\arrow[d, "q"']\arrow[r, "g"] & B\arrow[d, "p"]\\
C\arrow[r, "f"'] & D
\end{tikzcd}
\end{equation}
be a pullback square of groups such that $p$ (whence also $q$) is surjective, and let $\mathscr C$ be a complete category. Then the Beck-Chevalley transformation
\begin{equation*}
\begin{tikzcd}
A\text-\mathscr C\arrow[d, "q_*"']\twocell[from=dr] & \arrow[l, "g^*"'] B\text-\mathscr C\arrow[d, "p_*"]\\
C\text-\mathscr C & \arrow[l, "f^*"] D\text-\mathscr C
\end{tikzcd}
\end{equation*}
(i.e.~the canonical mate
\begin{equation*}
f^*p_*\xrightarrow{\eta} q_*q^*f^*p_*=q_*g^*p^*p_*\xrightarrow{\epsilon} q_*g^*
\end{equation*}
of the identity transformation) is an isomorphism.
\begin{proof}
This is well-known (see e.g.~\cite[Proposition~11.6]{joyal-book} for a result in much greater generality), but also easy enough to prove directly: if $X\in B\text-\mathscr C$, then $f^*p_*X=X^{\ker p}$ as objects of $\mathscr C$, while $q_*g^*X=(g^*X)^{\ker q}$, and the Beck-Chevalley transformation is the unique map over $X$. The claim follows as $g(\ker q)=\ker p$ by virtue of $(\ref{diag:pb-surj})$ being a pullback.
\end{proof}
\end{lemma}

\begin{prop}\label{prop:global-tfae}
Let $\underline{\mathscr C}$ be a preglobal model category. Then the following are equivalent:
\begin{enumerate}
\item For every pullback square $(\ref{diag:pb-surj})$ of finite groups in which all maps are surjections, the Beck-Chevalley transformation $f^*\circ\cat{R}p_*\Rightarrow \cat{R}q_*\circ g^*$ is an isomorphism in $\Ho(C\text-\mathscr C)$.
\item For every pullback square $(\ref{diag:pb-surj})$ of finite groups and surjections, the Beck-Chevalley transformation $\cat{L}g_!\circ q^*\Rightarrow p^*\circ\cat{L}f_!$ is an isomorphism in $\Ho(B\text-\mathscr C)$.
\item \label{item:gtfae-cofree-fixed-points} For every diagram
\begin{equation}\label{diag:span-global}
\begin{tikzcd}
C & \arrow[l, "q"'] A\arrow[r, "g"] & B
\end{tikzcd}
\end{equation}
of finite groups such that $\ker q\cap\ker g=1$, every fibrant $X\in B\text-\mathscr C_\textup{flat}$, and some (hence any) fibrant replacement $g^*X\to Y$ in $A\text-\mathscr C_\textup{flat}$ the induced map $q_*g^*X\to q_*Y$ is a $C$-global weak equivalence.
\item \label{item:gtfae-free-quotient} For every diagram $(\ref{diag:span-global})$, every cofibrant $X\in C\text{-}\mathscr C_\textup{proj}$, and some (hence any) cofibrant replacement $Y\to q^*X$ in $A\text-\mathscr C_\textup{proj}$ the induced map $g_!Y\to g_!q^*X$ is a $B$-global weak equivalence.
\end{enumerate}
\end{prop}

\begin{defi}\label{defi:global-model-cat}
A preglobal model category satisfying the equivalent conditions of Proposition~\ref{prop:global-tfae} is called a \emph{global model category}.
\end{defi}

\begin{proof}[Proof of Proposition~\ref{prop:global-tfae}]
The equivalence $(1)\Leftrightarrow(2)$ follows at once from the fact that the two maps in question are total mates of each other. We will now show that $(1)\Leftrightarrow(3)$; the argument that $(2)\Leftrightarrow(4)$ is then analogous.

For the proof of $(3)\Rightarrow(1)$ we consider a pullback as in $(\ref{diag:pb-surj})$ and we fix a fibrant object $X\in B\text-{\mathscr C}_\text{flat}$ and a fibrant replacement functor $\iota\colon\id\Rightarrow P$ for $A\text-\mathscr C_\text{flat}$. By naturality of $\iota$, we then have a commutative diagram
\begin{equation*}
\begin{tikzcd}
f^*p_*X\arrow[r,"\eta"] & q_*q^*f^*p_*X\arrow[d, "q_*\iota"'] \arrow[r,equal] & q_*g^*p^*p_*X\arrow[r, "q_*g^*\epsilon"]\arrow[d, "q_*\iota"] &[1em] q_*g^*X\arrow[d, "q_*\iota"]\\
& q_*Pq^*f^*p_*X\arrow[r,equal] & q_*Pg^*p^*p_*X\arrow[r, "q_*Pg^*\epsilon"'] & q_*Pg^*X
\end{tikzcd}
\end{equation*}
in $C\text{-}\mathscr C$. Here the top horizontal composite is an isomorphism by the previous lemma; moreover, the bottom composite $f^*p_*X\to q_*Pg^*X$ represents the Beck-Chevalley transformation $f^*\cat{R}p_*\Rightarrow \cat{R}q_*g^*$, so this is the map we want to be a $C$-global weak equivalence. However, the right hand vertical map is a $C$-global weak equivalence by assumption, so the claim follows by $2$-out-of-$3$.

Conversely, for the proof of $(1)\Rightarrow(3)$ we first observe that the above shows that for any pullback square $(\ref{diag:pb-surj})$ of surjections and any fibrant $X\in B\text{-}\mathscr C_\text{flat}$ the canonical map $q_*g^*X\to\cat{R}q_*g^*X$ is a $C$-global weak equivalence.

We will now use this to prove the special case of $(3)$ in which both $q$ and $g$ are surjective: namely, we set $D=A/\ker(g)\ker(q)$ and we define $p\colon B\to D$ as the composite
\begin{equation*}
B\xrightarrow[\cong]{g^{-1}} A/\ker(q)\twoheadrightarrow A/\ker(g)\ker(q)=D
\end{equation*}
and analogously we define a surjective homomorphism $f\colon C\to D$. We now claim that the commutative square
\begin{equation*}
\begin{tikzcd}
A\arrow[d, "q"']\arrow[r, "g"] & B\arrow[d, "p"]\\
C\arrow[r, "f"'] & D
\end{tikzcd}
\end{equation*}
of surjections is a pullback square; with this established, the above observation will then complete the proof of the special case.

To prove that this is indeed a pullback, first note that $(g,q)\colon A\to B\times_DC\subset B\times C$ is injective as $\ker(g,q)=\ker(g)\cap\ker(q)=1$. Thus, it only remains to prove surjectivity. For this we let $(b,c)\in B\times C$ with $p(b)=f(c)$; thus, if $a\in A$ with $g(a)=b$, then $fq(a)=pg(a)=p(b)=f(c)$, so $c=q(a)\bar x$ for some $\bar x\in\ker(f)$. By surjectivity of $q$, we can then write $c=q(ax)$ for some $x\in q^{-1}(\ker(f))=\ker(g)$; but we also have $g(ax)=g(a)g(x)=b$, so $ax$ is the desired preimage of $(b,c)\in B\times_DC$.

Now we can prove the general case of $(3)$: we factor $q$ as a surjection $q'\colon A\to C'$ followed by an injection $i\colon C'\to C$, and we factor $g$ as a surjection $g'\colon A\to B'$ followed by an injection $j\colon B'\to B$. If now $X\in B\text{-}\mathscr C_\text{flat}$ is fibrant, then also $j^*X\in B'\text-\mathscr C_{\text{flat}}$ is fibrant as $j$ is injective. Thus, if $g^*X=(g')^*j^*X\to Y$ is any fibrant replacement in $A\text{-}\mathscr C_\text{flat}$, then applying the above special case to $j^*X$ shows that $q'_*g^*X\to q'_*Y$ is a $C'$-global weak equivalence. The right hand side is fibrant in $C'\text-\mathscr C_\text{flat}$, and we claim that the left hand side is at least fibrant in $C'\text-\mathscr C_\text{proj}$; as $i_*$ is right Quillen for the projective model structures by injectivity, Ken Brown's Lemma will then show that also the induced map $i_*q'_*g^*X\to i_*q'_*Y$ is a $C$-global weak equivalence, and as this agrees with $q_*g^*X\to q_*Y$ up to conjugation by isomorphisms, this will then complete the proof of the proposition.

To prove the claim, we note that $\ker(g')\cap\ker(q')=\ker(g)\cap\ker(q)=1$, so the above argument yields a pullback square
\begin{equation*}
\begin{tikzcd}
A\arrow[r, "g'"]\arrow[d,"q'"'] & B'\arrow[d, "f'"]\\
C'\arrow[r, "p'"'] & D'
\end{tikzcd}
\end{equation*}
of finite groups and surjections. Thus, another application of Lemma~\ref{lemma:Beck-Chevalley} shows $q'_*g^*X=q'_*(g')^*i^*X\cong (p')^*f'_*i^*X$. However, $f'_*i^*X$ is fibrant in $D'\text-\mathscr C_\text{flat}$, hence in particular in $D'\text-\mathscr C_\text{proj}$, so $(p')^*f'_*i^*X$ is fibrant in $C'\text-\mathscr C_\text{proj}$ as desired.
\end{proof}

\begin{rk}
The first two formulations above are the `morally correct ones,' and they correspond to the notions of \emph{global continuity} and \emph{global cocontinuity} in parameterized higher category theory \cite{elements-param, martini-wolf}. In contrast to that, Condition~$(\ref{item:gtfae-cofree-fixed-points})$ is the statement we will actually use later (namely, in the proof of Theorem~\ref{thm:sp-global-model-cat}), while the final formulation is the one that is easiest to check in our examples below. Intuitively, we can think of a projectively cofibrant object in $G\text-\mathscr C$ as a flat one for which the $G$-action is `free' in some sense (cf.~for example Lemma~\ref{lemma:tensoring-generalized}); the final formulation of the above axiom can then be viewed as a very abstract incarnation of the slogan that `free quotients are homotopical.'
\end{rk}

\begin{rk}
	In the above proof we have seen how a pointset level condition (certain fibrant replacements being sent to weak equivalences) encodes a Beck-Chevalley condition for derived functors. Similarly, the pointset level conditions for a preglobal model category can be seen to encode various compatibility conditions between the derived functors. As a concrete example, consider an injective homomorphism $i\colon H\to G$ and a surjection $p\colon K\to G$. If we write $j\colon p^{-1}(H)\to K$ and $q\colon p^{-1}(H)\to H$ for the restricted maps, then Lemma~\ref{lemma:Beck-Chevalley} shows after dualizing that $j_!q^*\cong p^*i_!$. As all functors are left Quillen for the flat model structures (using that $i$ and $j$ are injective), this then induces an equivalence $\cat{L}j_!q^*\cong p^*\cat{L}i_!$ of derived functors. Similarly, one can give a `Mackey double coset formula' expressing $k^*\cat{L}i_!$ for injective $k$ as a (derived) coproduct of maps of the form $\cat{L}\ell_!r^*$, again coming from an analogous formula for non-derived functors.

	Note that conversely this type of argument tells us that we cannot get away with a single model structure for each $G\text{-}\mathscr C$ for which restrictions are always both left and right Quillen. For example, for any $G$ the projection $s\colon G\to 1$ satisfies $s_!s^*\cong\id$, so if both $s_!$ and $s^*$ were left Quillen we would deduce $\cat{L}s_!s^*\cong\id$. This is however not true for $G$-global spaces: we have $\cat{L}s_!s^*(*)\cong\mathcal I(G,\blank)/G$, whose underlying non-equivariant space is a $K(G,1)$.
\end{rk}

\begin{ex}\label{ex:S-Sp-global}
Both $\glo{GlobalSpaces}$ and $\glo{GlobalSpectra}$ are global model categories: this follows from Propositions~\ref{prop:I-free-quotients} and~\ref{prop:free-quotient-spectra}, respectively: if $X\in B\text-\mathscr C$ is projectively cofibrant, then $B$ acts freely on it, hence $\ker(q)$ acts freely on $g^*X$ as $g$ is injective when restricted to $\ker(q)$.
\end{ex}

\begin{ex}
Any shift of a global model category is again a global model category. This follows immediately from the third formulation.
\end{ex}

\begin{nex}
Let $p$ be a prime and let $\mathcal F$ be a global family of groups containing $\mathbb Z/p$ but not $\mathbb Z/p\times\mathbb Z/p$, for example the global family of groups of order at most $p$. We claim that the preglobal model category $\ul{\mathscr E}_{\mathcal F}$ from Example~\ref{ex:exotic} is \emph{not} a global model category; in particular, the conditions of the above proposition are not vacuous.

To this end, we will show that Condition~$(\ref{item:gtfae-free-quotient})$ is not satisfied. We set $B=C=\mathbb Z/p$ and $A=B\times C$ with $q$ and $g$ the respective projections. Then the $(B\cap\mathcal F)$-model structure is just the $\mathcal A\ell\ell$-model structure, and in particular $X=*$ is cofibrant in it. We now let $Y\to *=g^*(*)$ be a cofibrant replacement in the $(A\cap\mathcal F)$-model structure, and we claim that $q_!(Y)=Y/B$ has non-connected $C$-fixed points, so it is in particularly not weakly equivalent to $q_!(*)=*$.

For this, let us observe that the isotropy $I_\sigma$ of a simplex $\sigma$ of $Y$ only depends on the class $[\sigma]$ in $Y/B$ (as $A$ is abelian). Moreover, $[\sigma]$ is $C$-fixed if and only if $I_\sigma$ contains an element of the form $(b,1)$ with $b\in B$ arbitrary. Conversely, for any $b\in B$ there exists a vertex $y_b$ of $Y$ with $I_{y_b}\ni (b,1)$: namely, $(b,1)$ generates a subgroup $K\subset A$ isomorphic to $\mathbb Z/p$, so $Y^{K}$ is weakly contractible, hence in particular non-empty.

We now claim that $[y_0]$ and $[y_1]$ cannot be joined by a sequence of $C$-fixed edges in $Y/B$. For this, let us consider any $C$-fixed edge $[e]$ with vertices $[x],[y]$. Then $I_x\supset I_e\subset I_y$. On the other hand, $I_e\not=1$ as $[e]$ is $C$-fixed, while $I_x,I_y\not=A$ as $Y^A=\varnothing$ since $A\notin\mathcal F$. Thus, it follows for cardinality reasons that all of the above inclusions are equalities, and in particular $I_x=I_y$.

Now assume $[y_0]$ and $[y_1]$ were actually connected by a sequence of $C$-fixed edges. Then it would follow from the above by induction that $I_{y_0}=I_{y_1}$, whence in particular $(0,1),(1,1)\in I_{y_0}$. However, then $I_{y_0}=A$ which is impossible by the same argument as above.
\end{nex}

\subsection{Global Quillen adjunctions} Finally, let us discuss the appropriate notion of Quillen adjunctions in this context:

\begin{defi}
Let $\ul{\mathscr C},\ul{\mathscr D}$ be preglobal model categories. A \emph{global Quillen adjunction} $\underline F\colon\underline{\mathscr C}\rightleftarrows\underline{\mathscr D} :\!\underline U$ is a simplicially enriched adjunction $F\colon\mathscr C\rightleftarrows\mathscr D:\!U$ of the underlying categories such that for every finite group $G$ the induced adjunction $G\text{-}\mathscr C\rightleftarrows G\text{-}\mathscr D$ is a Quillen adjunction for both model structures.

We call $\underline{F}\dashv\underline{G}$ a \emph{global Quillen equivalence} if in addition each $G\text{-}\mathscr C\rightleftarrows G\text{-}\mathscr D$ is a Quillen equivalence (for the projective or, equivalently, for the flat model structures).
\end{defi}

\begin{ex}\label{ex:S-Sp-adjunction}
The suspension-loop adjunction $\Sigma^\bullet_+\colon\cat{$\bm{\mathcal I}$-SSet}\rightleftarrows \cat{Spectra}:\!\Omega^\bullet$ defines a global Quillen adjunction $\glo{GlobalSpaces}\rightleftarrows\glo{GlobalSpectra}$, see Proposition~\ref{prop:suspension-loop-G-gl}.
\end{ex}

\begin{lemma}
Let $\underline F\colon\underline{\mathscr C}\rightleftarrows\underline{\mathscr D} :\!\underline U$ be a global Quillen equivalence of preglobal model categories. Then $\underline{\mathscr C}$ is a global model category if and only if $\underline{\mathscr D}$ is so.
\begin{proof}
Associated to any pullback square $(\ref{diag:pb-surj})$ of finite groups and surjective maps, we obtain a coherent cube
\begin{equation*}
\begin{tikzcd}[row sep=small, column sep=small]
& \Ho(A\text-\mathscr C)\arrow[from=dd, "q^*"{near start}] &&\arrow[ll, "g^*"'] \Ho(B\text-\mathscr C)\arrow[from=dd,"p^*"']\\
\Ho(A\text-\mathscr D)\arrow[ur] &&\arrow[ll, "g^*"'{near start}, crossing over] \Ho(B\text-\mathscr D)\arrow[ur]\\
& \Ho(C\text-\mathscr C) &&\arrow[ll, "f^*"'{near end}] \Ho(D\text-\mathscr C)\\
\Ho(C\text-\mathscr D)\arrow[ur]\arrow[uu, "q^*"] && \arrow[ll, "f^*"] \Ho(D\text-\mathscr D)\arrow[ur]\arrow[uu, "p^*"'{near start}, crossing over]
\end{tikzcd}
\end{equation*}
where all front-to-back maps are given by $\cat{R}U$, the front and back face commute strictly, and the remaining squares are filled by the natural isomorphisms coming from the fact that all functors are right Quillen for the projective model structures and commute strictly on the pointset level.

Passing to canonical mates with respect to the adjunctions $q^*\dashv\cat{R}q_*$ and $p^*\dashv\cat{R}p_*$ we then get a coherent cube
\begin{equation*}
\begin{tikzcd}[row sep=small, column sep=small]
& \Ho(A\text-\mathscr C)\arrow[dd, "\cat{R}q_*"'{near end}] &&\arrow[ll, "g^*"'] \Ho(B\text-\mathscr C)\arrow[dd,"\cat{R}p_*"]\\
\Ho(A\text-\mathscr D)\arrow[ur]\arrow[dd, "\cat{R}q_*"'] &&\arrow[ll, "g^*"'{near start}, crossing over] \Ho(B\text-\mathscr D)\arrow[ur]\\
& \Ho(C\text-\mathscr C) &&\arrow[ll, "f^*"'{near end}] \Ho(D\text-\mathscr C)\\
\Ho(C\text-\mathscr D)\arrow[ur] && \arrow[ll, "f^*"] \Ho(D\text-\mathscr D)\arrow[ur]\arrow[from=uu, "\cat{R}p_*"{near end}, crossing over]
\end{tikzcd}
\end{equation*}
in which the transformations in the top, bottom, left, and right face are isomorphisms (the latter two use that $\cat{R}U$ is an equivalence). Using again that all front-to-back maps are equivalences, it follows that the natural transformation filling the front square is an isomorphism if and only if the one filling the back square is so, which immediately yields the claim.
\end{proof}
\end{lemma}

\section{Global stability}\label{sec:global-stab}
In this section we will introduce a notion of \emph{genuine stability} for global model categories and construct universal stabilizations in this setting.

\subsection{Pointed (pre)global model categories} As usual, in order to talk about stability we first have to talk about pointedness:

\begin{defi}
A preglobal model category $\mathscr C$ is called \emph{pointed} if the underlying category $\mathscr C$ is pointed in the usual sense, i.e.~has a zero object.
\end{defi}

\begin{rk}\label{rk:basepoint}
Recall \cite{hirschhorn, hirschhorn-slice} that for a model category $\mathscr C$ the category $\mathscr C_*$ of pointed objects (i.e.~the slice $*/\mathscr C$ under our favourite terminal object) carries a model structure in which a map is a weak equivalence, fibration, or cofibration, if and only if it is so in $\mathscr C$. If $\mathscr C$ is left proper, right proper, or combinatorial, then so is $\mathscr C_*$, with generating (acyclic) cofibrations in the latter case being given by applying the left adjoint $(\blank)_+=(\blank)\amalg*\colon\mathscr C\to\mathscr C_*$ to a set of generating (acyclic) cofibrations of $\mathscr C$. Moreover, if $\mathscr C$ is simplicial, then $\mathscr C_*$ is enriched as a model category over $\cat{SSet}_*$ (hence in particular over $\cat{SSet}$) by taking the basepoints of the mapping spaces to be the zero maps, while the tensoring $K\smashp X$ is induced by the tensoring $K\times X$ in $\mathscr C$ over $\cat{SSet}$ by collapsing $*\times X\amalg K\times *$. Analogously, we again get $\blank\smashp\blank\colon\cat{$\bm{(G\times H)}$-SSet}_*\times G\text-\mathscr C_*\to (G\times H)\text-\mathscr C_*$.

If now $F\colon \mathscr C\rightleftarrows\mathscr D :\!U$ is a Quillen adjunction, then $U$ lifts to $U_*\colon\mathscr D_*\to \mathscr C_*$ (as it preserves terminal objects); $F$ does not necessarily lift directly, but we get $\mathscr C_*\to  F(*)/\mathscr D$, which we can postcompose with pushforward along the unique map $F(*)\to *$ to get a functor $F_*\colon\mathscr C_*\to\mathscr D_*$ left adjoint to $U_*$. As fibrations and weak equivalences are defined in the underlying categories, we immediately see that $F_*\dashv U_*$ is a Quillen adjunction again.

In particular, given a preglobal model category $\ul{\mathscr C}$ we can make the category $\mathscr C_*$ into a preglobal model category $\ul{\mathscr C}_*$ this way, coming with a global Quillen adjunction $\ul{(\blank)_+}\colon\ul{\mathscr C}\rightleftarrows\ul{\mathscr C}_* :\!\ul{\text{forget}}$. If $\ul{\mathscr C}$ is actually a global model category, then so is $\ul{\mathscr C}_*$ (see the third formulation in Proposition~\ref{prop:global-tfae}). Moreover, if
\begin{equation}\label{eq:gQa-pointed}
\ul F\colon\ul{\mathscr C}\rightleftarrows\ul{\mathscr D}:\!\ul{U}
\end{equation}
is a global Quillen adjunction, then $F_*\dashv U_*$ defines a global Quillen adjunction $\ul{F}_*\dashv\ul{U}_*$.
\end{rk}

\begin{lemma}\label{lemma:pointed-equivalence}
If $(\ref{eq:gQa-pointed})$ is a global Quillen \emph{equivalence}, then so is the induced global Quillen adjunction $\ul{F}_*\dashv\ul{U}_*$.
\begin{proof}
We prove more generally that for any Quillen equivalence \hbox{$L\colon\mathscr A\rightleftarrows\mathscr B :\!R$} of \emph{left proper} model categories the induced Quillen adjunction $L_*\dashv R_*$ is a Quillen equivalence. This is well-known, and can also be deduced with a bit of work from \cite[Proposition~2.7]{rezk-proper}, but we do not know of an explicit reference in the literature, so we provide a direct argument.

For this, we first observe that $R_*$ still reflects weak equivalences between fibrant objects (as everything is defined in underlying categories), i.e.~its right derived functor is conservative. To complete the proof it suffices that for every cofibration $*\to X$ and some (hence any) fibrant replacement $L_*(*\to X)\to Z$ the induced map $X\to R_*L_*(*\to X)\to RZ$ in $\mathscr C$ is a weak equivalence.

For this, we pick a cofibrant replacement $Q\to *$. Factoring the induced map $Q\to *\to X$ as a cofibration followed by a weak equivalence we then get a commutative diagram
\begin{equation}\label{diag:pseudo-po}
\begin{tikzcd}
Q\arrow[d]\arrow[r] & Y\arrow[d]\\
*\arrow[r] & X
\end{tikzcd}
\end{equation}
in which both vertical maps are weak equivalences. In particular, this is a homotopy pushout, i.e.~(as $Q\to Y$ is a cofibration), the induced map $Y/Q\to X$ is a weak equivalence of cofibrant objects in $\mathscr A_*$. It therefore suffices to prove the claim for $Y/Q$ instead of $X$, i.e.~we may assume without loss of generality that $(\ref{diag:pseudo-po})$ is an \emph{honest} pushout. We then consider the diagram
\begin{equation*}
\begin{tikzcd}
LQ\arrow[d]\arrow[r] & LY\arrow[d]\\
L*\arrow[r]\arrow[d] & LX\arrow[d]\\
*\arrow[r] & L_*X
\end{tikzcd}
\end{equation*}
where the upper half is the image of $(\ref{diag:pseudo-po})$ under $L$ (whence a pushout) and the lower half is the pushout defining $L_*(*\to X)$; in particular, also the total rectangle is a pushout. As $L(Q\to Y)$ is a cofibration ($L$ being left Quillen) and the left hand vertical map is a weak equivalence, we conclude that the right hand vertical composite $LY\to L_*(*\to X)$ is a weak equivalence because $\mathscr D$ is left proper. Thus, if $L_*(*\to X)\to Z$ is any fibrant replacement (in $\mathscr C_*$), then the composite $LY\to Z$ is a fibrant replacement (in $\mathscr C$). We conclude that the composite $Y\to RLY\to RZ$ represents the derived unit of $L\dashv R$ (as $Y$ is cofibrant), so it is a weak equivalence. The claim follows by $2$-out-of-$3$ as the composite $X\to Y\to RZ$ represents the derived unit for $L_*\dashv R_*$.
\end{proof}
\end{lemma}

Finally, arguing as in Remark~\ref{rk:basepoint}, we deduce from Lemma~\ref{lemma:g-sset-enriched}:

\begin{cor}\label{cor:pointed-g-sset-enriched}
Let $G,H$ be finite groups and let $\ul{\mathscr C}$ be a pointed preglobal model category. Then the smash product defines left Quillen bifunctors
\begin{align*}
\big(\cat{$\bm{(G\times H)}$-SSet}_*\big)_{\mathcal G_{G,H}}\times G\text-\mathscr C_\textup{proj}&\to (G\times H)\text-\mathscr C_\textup{proj}\\
\cat{$\bm{(G\times H)}$-SSet}_*\times G\text-\mathscr C_\textup{flat}&\to (G\times H)\text-\mathscr C_\textup{flat}.
\end{align*}
In particular, both the projective and flat model structure on $G\text{-}\mathscr C$ are enriched as model categories over $\cat{$\bm G$-SSet}_*$.\qed
\end{cor}

\subsection{Genuine stability} Recall from Remark~\ref{rk:equivariant-stabilization} that the passage from $G$-spaces to $G$-spectra can be understood as inverting smashing with the spheres $S^A$ for all finite $G$-sets $A$. While we cannot apply this directly to the passage from global spaces to global spectra in that way (as smashing with $S^A$ for a non-trivial $G$-set $A$ does not make sense directly), smashing with $S^A$ makes sense for \emph{$G$-global} spaces, and following an idea of \cite{gepner-nikolaus} (or more generally the philosophy of parameterized higher category theory) we can then try to characterize the passage from unstable to stable global homotopy theory by looking more generally at what happens in $G$-global homotopy theory for all $G$ simultaneously:

\begin{defi}\label{def:global-stability}
A global model category $\ul{\mathscr C}$ is called \emph{genuinely stable} (or simply \emph{stable}) if it is pointed and the Quillen adjunction $S^A\smashp\blank\colon G\text{-}\mathscr C\rightleftarrows G\text{-}\mathscr C :\!\Omega^A$ is a Quillen equivalence for every finite group $G$ and every finite $G$-set $A$ (for the flat or, equivalently, the projective model structures).
\end{defi}

Note that specializing to $A=*$ with trivial $G$-action shows that each $G\text{-}\mathscr C$ is in particular stable in the usual sense.

\begin{prop}\label{prop:GlobalSpectra-stable}
The global model category $\glo{GlobalSpectra}$ from Example~\ref{ex:global-spectra} is stable.
\begin{proof}
The usual smash product of symmetric spectra defines a left Quillen bifunctor
\begin{equation}\label{eq:equivariant-global-bifunctor}
\cat{$\bm G$-Spectra}_\text{$G$-equiv.~proj.}\times \cat{$\bm G$-Spectra}_\text{$G$-gl.~flat}\to  \cat{$\bm G$-Spectra}_\text{$G$-gl.~flat}
\end{equation}
by Theorem~\ref{thm:smash-g-global} together with Proposition~\ref{prop:equivariant-vs-global}, making the $G$-global flat model structure tensored over the $G$-equivariant projective one.

If now $A$ is any finite $G$-set, then $\Sigma^\infty S^A$ is cofibrant in the $G$-equivariant projective model structure, so $\Sigma^\infty S^A\smashp^{\cat L}\blank$ agrees with the left derived functor of $S^A\smashp\blank$; it therefore suffices to show that $\Sigma^\infty S^A\smashp^{\cat L}\blank$ is an autoequivalence of the $G$-global stable homotopy category. But by \cite[Proposition~4.9]{hausmann-equivariant} the analogous functor is an autoequivalence of the $G$-\emph{equivariant} stable homotopy category, so we get a projectively cofibrant $G$-equivariant spectrum $D$ together with a zig-zag of weak equivalences of projectively cofibrant objects between $D\smashp \Sigma^\infty S^A$ and $\mathbb S$. As $(\ref{eq:equivariant-global-bifunctor})$ is a left Quillen bifunctor, this then shows that $D\smashp^{\cat L}\blank$ is the desired quasi-inverse.
\end{proof}
\end{prop}

\begin{rk}
In the above argument we crucially use that $\Sigma^\infty S^A$ is actually cofibrant in the equivariant \emph{projective} model structure (and not just flat); in particular, $(\ref{eq:equivariant-global-bifunctor})$ is not left Quillen for the $G$-equivariant \emph{flat} model structure. On the other hand, we could have just as well used the $G$-global projective model structure instead of the flat one everywhere.
\end{rk}

\begin{lemma}\label{lemma:stability-stable}
Let $\ul{F}\colon \ul{\mathscr C}\rightleftarrows\ul{\mathscr D} :\!\ul{U}$ be a global Quillen equivalence of pointed global model categories. Then $\ul{\mathscr C}$ is stable if and only if $\ul{\mathscr D}$ is so.
\begin{proof}
Let $G$ be a finite group and $A$ a finite $G$-set. As $F$ is a simplicial and hence also $\cat{SSet}_*$-enriched left adjoint, there is a (canonical) isomorphism filling
\begin{equation*}
\begin{tikzcd}
G\text-\mathscr C\arrow[d, "F"']\arrow[r, "S^A\smashp\blank"] &[1em] G\text-\mathscr C\arrow[d, "F"]\\
G\text-\mathscr D\arrow[r, "S^A\smashp\blank"'] & G\text-\mathscr D,
\end{tikzcd}
\end{equation*}
and as all participating functors are left Quillen (say, for the projective model structures), this induces an isomorphism $\cat{L}(S^A\smashp\blank)\circ\cat{L}F\cong\cat{L}F\circ\cat{L}(S^A\smashp\blank)$ of left derived functors. The claim follows by $2$-out-of-$3$.
\end{proof}
\end{lemma}

\begin{rk}
Let $\ul{\mathscr C}$ be a stable global model category and let $G$ be a finite group. Then the shift $\ul{G\text-\mathscr C}$ is again stable: if $G'$ is any finite group and $A$ is a finite $G'$-set, then viewing it as a $(G'\times G)$-set with trivial $G$-action shows that the left Quillen functor $S^A\smashp\blank\colon G'\text-G\text-\mathscr C\to G'\text-G\text-\mathscr C$ is a Quillen equivalence.
\end{rk}

\subsection{Spectrification}\label{subsec:spectrification}
Non-equivariantly, the universal stabilization of a suitably nice simplicial model category $\mathscr C$ can be computed by spectrum objects in $\mathscr C$. In this subsection, we introduce a refinement of this construction to our framework; the universal property will then be established in the next subsection.

\begin{constr}
Let $\underline{\mathscr C}$ be a pointed preglobal model category. We write $\Sp(\underline{\mathscr C})$ for the category of $\cat{SSet}_*$-enriched functors $\bm\Sigma\to\mathscr C$. A map $f$ in $\Sp(\underline{\mathscr C})$ is called a \emph{global level weak equivalence} if $f(A)$ is a weak equivalence in $\Sigma_A\text{-}\mathscr C$ for every finite set $A$. Moreover, $f$ is called a \emph{projective global level fibration} or \emph{flat global level fibration} if each $f(A)$ is a fibration in the projective $\Sigma_A$-global model structure or flat $\Sigma_A$-global model structure, respectively.

More generally, if $\ul{\mathscr C}$ is an arbitrary preglobal model category, then we define $\Sp(\ul{\mathscr C})\mathrel{:=}\Sp(\ul{\mathscr C}_*$).
\end{constr}

\begin{prop}\label{prop:level-model-structures}
Assume $\ul{\mathscr C}$ is pointed. The global level weak equivalences and global projective (flat) level fibrations are part of a simplicial, combinatorial, and left proper model structure on $\Sp(\underline{\mathscr C})$. We call this the \emph{global projective (flat) level model structure}. A possible set of generating cofibrations is given by
\begin{equation*}
\{\bm\Sigma(A,\blank)\smashp_{\Sigma_A}i : \text{$A\in\bm\Sigma$, $i\in I_{\Sigma_A}$}\}
\end{equation*}
where $I$ denotes a set of generating cofibrations of the projective (flat) model structure on $\Sigma_A\text{-}\mathscr C$, and similarly a set of generating acyclic cofibrations is given by
\begin{equation*}
\{\bm\Sigma(A,\blank)\smashp_{\Sigma_A}j : \text{$A\in\bm\Sigma$, $j\in J_{\Sigma_A}$}\}
\end{equation*}
for sets $J_{\Sigma_A}$ of generating acyclic cofibrations of the respective model structure on $\Sigma_A\text-\mathscr C$.
\end{prop}

Here $\bm\Sigma(A,\blank)\smashp_{\Sigma_A} f$ is the map obtained from the levelwise tensoring over $\cat{SSet}_*$ by dividing out the diagonal $\Sigma_A$-action.

\begin{rk}
Replacing $\ul{\mathscr C}$ by $\ul{\mathscr C}_*$ we immediately get an analogous statement for unpointed $\mathscr C$, with generating (acyclic) cofibrations now of the form $\bm\Sigma(A,\blank)\smashp_{\Sigma_A} f_+$ for generating (acyclic) cofibrations $f$ of $\Sigma_A\text-\mathscr C$.
\end{rk}

For the proof of the proposition we will use the following easy criterion, cf.~\cite[Proposition~C.23]{schwede-book} or \cite[Definition~2.21]{hausmann-equivariant}:

\begin{lemma}
Let $\mathbb I$ be a small $\cat{SSet}_*$-enriched category, and let $\mathscr C$ be a locally presentable category enriched over $\cat{SSet}_*$. Assume we are given for each $X\in\mathbb I$ a combinatorial model structure on $\End(X)\text-\mathscr C$ (the category of enriched functors from the full $\cat{SSet}_*$-subcategory spanned by $X$ to $\mathscr C$) with generating cofibrations $I_X$ and generating acyclic cofibrations $J_X$, such that the following `consistency condition' is satisfied: for every $Y\in\mathbb I$, any relative $\{{\mathbb I}(X,Y)\smashp_{\End(X)} j : X\in\mathbb I,j\in J_X\}$-cell complex is a weak equivalence in $\End(Y)\text-\mathscr C$.

Then there is a unique model structure on the category $\mathbb I\text-\mathscr C$ of enriched functors $\mathbb I\to\mathscr C$ in which a map $f$ is a weak equivalence if and only if $f(X)$ is a weak equivalence or fibration, respectively, in the given model structure on $\End(X)\text-\mathscr C$ for all $X\in\mathbb I$. This model structure is combinatorial with generating cofibrations
\begin{equation*}
\{\mathbb I(X,\blank)\smashp_{\End(X)} i : X\in\mathbb I, i\in I_X\}
\end{equation*}
and generating acyclic cofibrations
\begin{equation*}
\{\mathbb I(X,\blank)\smashp_{\End(X)} j : X\in\mathbb I, j\in J_X\}.
\end{equation*}
\begin{proof}
The forgetful functor $\mathbb I\text-\mathscr C\to\prod_{X\in\mathbb I}\End(X)\text-\mathscr C$ has a left adjoint given by $F\mathrel{:=}\coprod_{X\in\mathbb I}\mathbb I(X,\blank)\smashp_{\End(X)}\pr_{X}$. We will verify the conditions of the Crans-Kan Transfer Criterion \cite[Theorem~11.3.2]{hirschhorn} for this adjunction, which will then provide the desired model structure and show that it is cofibrantly generated (hence combinatorial) with the above sets of generating (acyclic) cofibrations. By local presentability, every set permits the small object argument, so we only have to show that every relative $F(J)$-cell complex is a weak equivalence where $J$ is a set of generating acyclic cofibrations of the right hand side. But for the standard choice of generating acyclic cofibrations this precisely amounts to the above consistency condition.
\end{proof}
\end{lemma}

\begin{proof}[Proof of Proposition~\ref{prop:level-model-structures}]
Let us consider the case of the flat model structures first. To verify the above consistency condition, it suffices to show that for all finite sets $A\subset B$ the map $\bm\Sigma(A,B)\smashp_{\Sigma_A}\blank\colon\Sigma_A\text{-}\mathscr C_\text{flat}\to\Sigma_B\text{-}\mathscr C_\text{flat}$ is left Quillen. To this end, we observe that we can identify $\bm\Sigma(A,B)$ with $(\Sigma_B)_+\smashp_{\Sigma_{B\setminus A}}S^{B\setminus A}$ as simplicial sets with left $\Sigma_B$- and right $\Sigma_A$-action, see \cite[proof of Proposition~2.24]{hausmann-equivariant}, so $\bm\Sigma(A,B)\smashp_{\Sigma_A}\blank$ factors as the composite
\begin{equation*}
\Sigma_A\text{-}\mathscr C_\text{flat}\xrightarrow{S^{B\setminus A}\smashp\blank}(\Sigma_A\times\Sigma_{B\setminus A})\text{-}\mathscr C_\text{flat}\xrightarrow{k_!}\Sigma_{B}\text{-}\mathscr C_\text{flat}
\end{equation*}
where $k\colon\Sigma_A\times\Sigma_{B\setminus A}\to\Sigma_B$ is the evident embedding. As $k$ is injective, the second arrow is left Quillen, and so is the first one by Corollary~\ref{cor:pointed-g-sset-enriched}.

The consistency condition for the projective model structure follows immediately from the one for the flat model structure as it has fewer cofibrations and the same weak equivalences, proving the existence of the projective level model structure.

As (acyclic) fibrations and the cotensoring on $\Sp(\ul{\mathscr C})$ are simply defined levelwise, we immediately see that these model structures are again simplicial. Similarly, for left properness it is enough to observe that all generating cofibrations of either model structure are levelwise flat cofibrations (as $\bm\Sigma(A,B)\smashp_{\Sigma_A}\blank$ is left Quillen for the flat model structures by the above), so that every cofibration in $\Sp(\ul{\mathscr C})$ is in particular a levelwise flat cofibration.
\end{proof}

Applying this to the shifts $\underline{G\text-\mathscr C}$ of $\underline{\mathscr C}$ (Example~\ref{ex:shifts}) gives us two left proper, simplicial, and combinatorial model structures on $G\text{-}\Sp(\underline{\mathscr C})$ that we call the \emph{$G$-global projective level model structure} and \emph{$G$-global flat level model structure}. Their weak equivalences and fibrations are those maps $f$ such that $f(A)$ is a weak equivalence or fibration, respectively, in the corresponding model structure on $(G\times\Sigma_A)\text{-}\mathscr C$ for every $A\in\bm\Sigma$.

In contrast to this, the cofibrations are not just simply defined levelwise, but rather in terms of a left lifting property. Nevertheless we can say something about the individual levels of the above cofibrations; we begin with the flat case where we have already noticed in the above proof:

\begin{cor}\label{cor:flat-implies-level-flat}
Let $f$ be a cofibration in $G\text-\Sp(\ul{\mathscr C})_\textup{flat}$ and let $B$ be a finite set. Then $f(B)$ is a cofibration in $(G\times\Sigma_B)\text-\mathscr C_\textup{flat}$.\qed
\end{cor}

In the projective case we get a slightly weaker statement:

\begin{lemma}\label{lemma:projective-cofibrations-levelwise}
Let $i$ be a projective cofibration in $G\text-\Sp(\ul{\mathscr C})$. Then $i(B)$ is a cofibration in $G\text-\mathscr C_\textup{proj}$ for every finite $G$-set $B$.
\begin{proof}
It is enough to prove this for the generating cofibrations. But for any finite set $A$ and any $(G\times\Sigma_A)$-global projective cofibration $i$
\begin{equation*}
(\bm\Sigma(A,\blank)\smashp_{\Sigma_A} i)(B)= (\bm\Sigma(A,B)\smashp i)/\Sigma_A
\end{equation*}
and $\bm\Sigma(A,B)\smashp i$ is a $(G\times\Sigma_A)$-global projective cofibration by Corollary~\ref{cor:pointed-g-sset-enriched}; the claim follows as quotients preserve projective cofibrations.
\end{proof}
\end{lemma}

\begin{lemma}\label{lemma:level-model-structures-functoriality}
The projective and flat $G$-global level model structures define a preglobal model category $\underline{\Sp}(\underline{\mathscr C})_\textup{level}$. If $\ul{\mathscr C}$ is actually a global model category, then so is $\underline{\Sp}(\underline{\mathscr C})_\textup{level}$.
\begin{proof}
We may assume without loss of generality that $\ul{\mathscr C}$ is pointed.

By the above, both model structures are simplicial, combinatorial, and left proper. Moreover, it is clear that they have the same weak equivalences and that every (generating) projective cofibration is also a flat cofibration. Thus, it only remains to verify the change-of-group properties.

For this let $\alpha\colon G\to G'$ be any group homomorphism. Then the restriction $(\alpha\times\id)^*\colon (G'\times\Sigma_A)\text-\mathscr C\to (G\times\Sigma_A)\text-\mathscr C$ preserves weak equivalences and projective fibrations. Thus, $\alpha^*\colon G'\text-\Sp(\underline{\mathscr C})\to G\text-\Sp(\underline{\mathscr C})$ preserves level weak equivalences as well as projective level fibrations. Moreover, if $\alpha$ is injective, then each $(\alpha\times\id)^*\colon(G'\times\Sigma_A)\text-\mathscr C\to (G\times\Sigma_A)\text-\mathscr C$ also preserves flat fibrations, so that $\alpha^*$ preserves flat level fibrations.

Similarly, one shows that $\alpha_*\colon G\text-\Sp(\underline{\mathscr C})\to G'\text-\Sp(\underline{\mathscr C})$ preserves acyclic flat level fibrations for any $\alpha$, so that $\alpha^*$ preserves flat level cofibrations, and that $\alpha_*$ also preserves acyclic projective level fibrations if $\alpha$ is injective, so that $\alpha^*$ preserves projective level cofibrations in this case.

Altogether we have shown that $\underline{\Sp}(\underline{\mathscr C})$ is a preglobal model category. Now assume $\ul{\mathscr C}$ is actually a global model category, let $g\colon A\to B$ and $q\colon A\to C$ be homomorphisms with $\ker g\cap\ker q = 1$, let $X\in B\text-\Sp(\underline{\mathscr C})_\text{flat level}$ be fibrant, and let $i\colon g^*X\to Y$ be a fibrant replacement in $A\text-\Sp(\underline{\mathscr C})_\text{flat level}$; we have to show that $q_*(i)$ is a $C$-global level weak equivalence. But for any finite set $D$, $X(D)$ is fibrant in $(B\times\Sigma_D)\text-\mathscr C_\text{flat}$ and $i(D)$ is a fibrant replacement in $(A\times\Sigma_D)\text-\mathscr C_\text{flat}$. Thus, $(g_*i)(D)=(g\times\Sigma_D)_*(i(D))$ is a $(C\times\Sigma_D)$-global weak equivalence as $\underline{\mathscr C}$ is a global model category and $\ker(g\times\id)\cap \ker(q\times\id)=1$, whence $g_*i$ is a $G$-global level weak equivalence as desired.
\end{proof}
\end{lemma}

\begin{lemma}\label{lemma:sp-level-induced-adjunction}
Let $\ul F\colon\ul{\mathscr C}\rightleftarrows\ul{\mathscr D} :\!\ul U$ be a global Quillen adjunction. Then also
\begin{equation*}
\ul\Sp(\ul{F})\colon \ul\Sp(\ul{\mathscr C})_\textup{level}\rightleftarrows\Sp(\ul{\mathscr D})_\textup{level} :\!\ul\Sp(\ul U)
\end{equation*}
is a global Quillen adjunction. If $\ul F\dashv\ul G$ is a global Quillen equivalence, then so is $\ul\Sp(\ul F)\dashv\ul\Sp(\ul U)$.
\begin{proof}
By Lemma~\ref{lemma:pointed-equivalence}, we may assume without loss of generality that $\mathscr C$ and $\mathscr D$ are pointed. Passing to shifts it further suffices to prove that $\Sp(\ul F)\dashv\Sp(\ul U)$ is a Quillen adjunction for both level model structures and a Quillen equivalence if $\ul F\dashv\ul U$ is a global Quillen equivalence.

For the first statement it is enough to observe that $\Sp(\ul U)$ preserves (acyclic) fibrations in either model structure as they are simply defined levelwise.

For the second statement, it suffices to prove that this is a Quillen equivalence for the flat model structures. But indeed, if $X\in\Sp(\ul{\mathscr C})_\text{flat}$ is flat, then $X(A)$ is flat for every $A$ by Corollary~\ref{cor:flat-implies-level-flat}; similarly, if $FX\to Y$ is a fibrant replacement, then each $FX(A)\to Y(A)$ is a fibrant replacement in $\Sigma_A\text-\mathscr C_\text{flat}$. Now the composite $X\to UFX\to UY$ represents the derived unit for $X$ by definition, but at the same time each $X(A)\to UY(A)$ represents the derived unit for $\Sigma_A\text-F\dashv \Sigma_A\text-U$ by the above, so it is a $\Sigma_A$-global weak equivalence by assumption. Analogously one shows that also the derived \emph{co}unit is a global level weak equivalence, finishing the proof of the lemma.
\end{proof}
\end{lemma}

As before, we now want to restrict to a suitable notion of {$\Omega$-spectra} via Bousfield localization.

\begin{defi}
Let $\underline{\mathscr C}$ be a preglobal model category. An object $X\in\Sp(\underline{\mathscr C})$ is called a \emph{global $\Omega$-spectrum} if for every finite group $H$, every finite $H$-set $A$ and every finite $H$-set $B$ the derived adjoint structure map
\begin{equation}\label{eq:der-adj-str-global}
X(A)\to\cat{R}\Omega^BX(A\amalg B)
\end{equation}
is an $H$-global weak equivalence.
\end{defi}

Here we are deriving $\Omega^B$ with respect to the $H$-global projective model structure; in particular, if $X$ is fibrant in either of the above level model structures, then $(\ref{eq:der-adj-str-global})$ is already represented by the ordinary adjoint structure map. Note that there are \emph{no} faithfulness assumptions on $A$ and $B$ here.

Again we can apply this to the shifts of $\mathscr C$ by $G$; we will refer to the corresponding objects of $G\text{-}\Sp(\underline{\mathscr C})$ as \textit{$G$-global $\Omega$-spectra}.

\begin{rk}\label{rk:g-global-omega-g-h-sets}
In the definition of a $G$-global $\Omega$-spectrum we can equivalently ask for the adjoint structure map $X(A)\to\cat{R}\Omega^BX(A\amalg B)$ to be a $(G\times H)$-global weak equivalence for all finite $(G\times H)$-sets $A, B$ (instead of just for $H$-sets). Namely, we can simply replace $H$ by $G\times H$ in the above and then restrict along the diagonal $G\times H\to G\times H\times G$.
\end{rk}

In order to construct the corresponding Bousfield localizations we need some additional notation:

\begin{constr}
As already used in the description of the generating (acyclic) cofibrations, performing the tensoring over $\cat{SSet}_*$ levelwise gives us a bifunctor
\begin{equation}\label{eq:tensoring-spectra}
\blank\smashp\blank\colon\cat{Spectra}\times\mathscr C\to\Sp(\underline{\mathscr C}),
\end{equation}
which preserves colimits in each variable separately. In particular, $X\smashp\blank$ has a right adjoint $F(X,\blank)\colon\Sp(\underline{\mathscr C})\to\mathscr C$ for every spectrum $X$. From this, we get bifunctors
\begin{align}
\blank\smashp\blank\colon \cat{$\bm G$-Spectra}\times G\text-{\mathscr C} &\to G\text-\Sp(\underline{\mathscr C})\label{eq:tensoring-g-spectra}\\
F\colon\cat{$\bm G$-Spectra}^\op\times G\text-\Sp(\underline{\mathscr C})&\to G\text-\mathscr C\label{eq:cotensoring-g-spectra}
\end{align}
by pulling through the $G$-actions everywhere.
\end{constr}

\begin{lemma}\label{lemma:tensoring-g-global-level}
The smash product $(\ref{eq:tensoring-g-spectra})$ is a left Quillen bifunctor with respect to the $G$-\emph{equivariant} projective level model structure on $\cat{$\bm G$-Spectra}$ and the projective $G$-global (level) model structures elsewhere. Dually, $(\ref{eq:cotensoring-g-spectra})$ is a right Quillen bifunctor for these model structures.
\begin{proof}
Again, we may assume $\ul{\mathscr C}$ to be pointed. By adjunction, it will be enough to show that the functor
\begin{equation*}
G\text-\mathscr C_\text{proj.}^\op\times G\text-\Sp(\underline{\mathscr C})_\text{proj.~level}\to \cat{$\bm G$-Spectra}_\text{$G$-equiv.~proj.~level}
\end{equation*}
given by taking mapping spaces levelwise is a right Quillen bifunctor. As also the weak equivalences and fibrations are defined levelwise, it suffices to show that
\begin{equation*}
\maps\colon G\text-\mathscr C_\text{proj.}^\op\times (G\times \Sigma_A)\text-\mathscr C_\text{proj.}\to \big(\cat{$\bm{(G\times \Sigma_A)}$-SSet}_*\big)_\text{$\mathcal G_{G,\Sigma_A}$-equiv.}
\end{equation*}
is a right Quillen bifunctor for every finite set $A$. By another adjointness argument, this then follows from Corollary~\ref{cor:pointed-g-sset-enriched}.
\end{proof}
\end{lemma}

\begin{prop}\label{prop:projective-model-structure-G-gl}
The $G$-global projective level model structure on $G\text{-}\Sp(\underline{\mathscr C})$ admits a Bousfield localization whose fibrant objects are precisely the $G$-globally projectively level fibrant $G$-global $\Omega$-spectra. We call the resulting model structure the \emph{$G$-global projective model structure}. It is again left proper, combinatorial, and simplicial.
\begin{proof}
As before we may assume that $\ul{\mathscr C}$ is pointed, and replacing $\ul{\mathscr C}$ by $\underline{G\text{-}\mathscr C}$ if necessary, it suffices to prove this for $G=1$.

As $\Sp(\underline{\mathscr C})_\text{proj.~level}$ is left proper, combinatorial, and simplicial, it will be enough by the localization machinery of \cite[Proposition A.3.7.3]{htt} to give a set $S$ of cofibrations such that a fibrant $X\in\Sp(\underline{\mathscr C})_\text{proj.~level}$ is a global $\Omega$-spectrum if and only if for every $f\in S$ the restriction $\maps(f,X)$ is an acyclic Kan fibration.

For this we recall from Remark~\ref{rk:equivariant-proj-gen-cof} for any finite group $H$ and any finite $H$-sets $A,B$ the map $\lambda_{H,A,B}\colon S^B\smashp\bm\Sigma(A\amalg B,\blank)\to\bm\Sigma(A,\blank)$ corepresenting $X(A)^H\to(\Omega^B X(A\amalg B))^H$ and its factorization $\lambda_{H,A,B}=\rho_{H,A,B}\kappa_{H,A,B}$ into an $H$-equivariant projective cofibration followed by a level weak equivalence. Moreover, let us pick for each finite group $H$ a set $I_H$ of generating cofibrations of $H\text-\mathscr C_\text{proj}$. We now claim that the set
\begin{equation}\label{eq:defining-weak-equivalences}
S\mathrel{:=}\{\kappa_{H,A,B}\ppo_H i : H,A,B,i\in I_H\}
\end{equation}
of (balanced) pushout products, where $H$ runs through all finite groups and $A$ and $B$ through finite $H$-sets, has the desired properties.

To this end we first observe that each ordinary pushout product $\kappa_{H,A,B}\ppo i$ is a cofibration in $H\text-\Sp(\ul{\mathscr C})_\text{proj.~level}$ by Lemma~\ref{lemma:tensoring-g-global-level}, so that $\kappa_{H,A,B}\ppo_Hi$ is indeed a projective cofibration in $\Sp(\ul{\mathscr C})$ as $\ul{\Sp}(\ul{\mathscr C})_\textup{level}$ is a preglobal model category. Now by adjointness $\maps(\kappa_{H,A,B}\ppo_H i, X)$ is an acyclic Kan fibration if and only if the map $F(\kappa_{H,A,B},\triv_HX)$ has the right lifting property in $H\text{-}\mathscr C$ against all maps of the form $i\ppo(\del\Delta^n\hookrightarrow\Delta^n)_+$. Letting $i$ vary, we conclude (as on the one hand $H\text{-}\mathscr C$ is tensored over $\cat{SSet}$ and as on the other the pushout product with $(\del\Delta^0\to\Delta^0)_+$ gives back the original map up to isomorphism) that $\maps(\kappa_{H,A,B}\ppo_Hi,X)$ is an acyclic Kan fibration for all $i\in I_H$ if and only if $F(\kappa_{H,A,B},\triv_HX)$ is an acyclic fibration in $H\text{-}\mathscr C_\text{proj}$. On the other hand, $F(\kappa_{H,A,B},\triv_HX)$ is always a fibration because $F$ is a right Quillen bifunctor (Lemma~\ref{lemma:tensoring-g-global-level} again) and $X$ was assumed to be projectively level fibrant (so that $\triv_HX$ is $H$-globally projectively level fibrant). Thus, $\maps(\kappa_{H,A,B}\ppo_H i, X)$ is an acyclic fibration for all $i$ if and only if $F(\kappa_{H,A,B},\triv_HX)$ is a weak equivalence in $H\text{-}\mathscr C$. However, $\rho_{H,A,B}$ is an $H$-equivariant level weak equivalence between projectively cofibrant $H$-equivariant spectra (as $\kappa_{H,A,B}$ is a cofibration and the source and target of $\lambda_{H,A,B}$ were cofibrant), so Ken Brown's Lemma shows that $F(\rho_{H,A,B},\triv_HX)$ is an $H$-global weak equivalence. Thus, $F(\lambda_{H,A,B},\triv_HX)$ is an $H$-global weak equivalence if and only if $F(\kappa_{H,A,B},\triv_HX)$ is so. Since $F(\lambda_{H,A,B},\triv_HX)$ is conjugate to the adjoint structure map $X(A)\to\Omega^B X(A\amalg B)$, the claim follows.
\end{proof}
\end{prop}

\begin{rk}
We also explicitly describe the set of maps used to obtain the $G$-global projective model structure by localization in the above proof. Considering the shift $\ul{G\text-\mathscr C}$, we see that this is given as
\begin{equation}\label{eq:localizing-set-G-gl}
S_G\mathrel{:=}\{\kappa_{H,A,B}\ppo_H i : H\text{ finite group}, A,B \text{ finite $H$-sets}, i\in I_{G\times H}\}.
\end{equation}
\end{rk}

\begin{cor}
For any finite group $G$, there is a unique model structure on $G\text-\Sp(\underline{\mathscr C})$ whose cofibrations are the $G$-global flat level cofibrations and whose weak equivalences are the $G$-global weak equivalences. This model structure is simplicial, combinatorial, and left proper. Moreover, its fibrant objects are precisely the $G$-global $\Omega$-spectra that are fibrant in the $G$-global flat level model structure.
\begin{proof}
Again, we may assume that $\ul{\mathscr C}$ is pointed and $G=1$. As every global projective level cofibration is also a global flat level cofibration, we can localize the global flat level model structure at the maps $(\ref{eq:defining-weak-equivalences})$; as every fibration in the global flat level model structure is also a fibration in the global projective level model structure, the above argument then shows that a level fibrant spectrum is fibrant in this new model structure if and only if it is a global $\Omega$-spectrum. It only remains to show that the weak equivalences of this model structure are again the global weak equivalences. But by abstract nonsense about Bousfield localizations, a map $f$ in $\Sp(\underline{\mathscr C})$ is a weak equivalence in this global flat model structure if and only if $[f,X]$ is bijective for every global $\Omega$-spectrum, where $[\,{,}\,]$ denotes the hom set in the localization of $\Sp(\underline{\mathscr C})$ at the global level weak equivalences. As the same characterization applies to the global projective model structure, the claim follows immediately.
\end{proof}
\end{cor}

\begin{warn}
It might be tempting to assume that the generating acyclic cofibrations of the above $G$-global model structures are given by adding the set $S_G$ from $(\ref{eq:localizing-set-G-gl})$ to the generating acyclic cofibrations of the level model structure, analogously to the equivariant situation (Remark~\ref{rk:equivariant-proj-gen-cof}). We warn the reader however that while all of these are clearly acyclic cofibrations, it is not clear whether they generate, and more generally the localization machinery employed above does \emph{not} provide any control about the generating acyclic cofibrations. In fact, the explicit identification of the generating acyclic cofibrations in the equivariant case referred to above crucially relies on \emph{right} properness of the level model structure, which is not assumed in our setting.
\end{warn}

Despite these words of warning, the following result will often allow us to \emph{pretend} that the generating acyclic cofibrations are of the above form:

\begin{prop}\label{prop:proj-level-to-global}
Let $G$ be a finite group and let $\mathscr D$ be a left proper simplicial model category with a simplicial Quillen adjunction $F\colon G\text-\Sp(\ul{\mathscr C})_\textup{proj.~level}\rightleftarrows\mathscr D :\!U$. Then the following are equivalent:
\begin{enumerate}
\item $F$ is left Quillen as a functor $G\text-\Sp(\ul{\mathscr C})_\textup{proj}\to\mathscr D$.
\item $F$ sends the maps in $S_G$ to weak equivalences.
\item $U$ sends fibrant objects to $G$-global $\Omega$-spectra.
\end{enumerate}
The analogous statement for the flat $G$-global (level) model structure also holds.
\end{prop}

The proof will in turn rely on the following general result:

\begin{lemma}\label{lemma:check-QA-fibrant}
Let $F\colon\mathscr C\rightleftarrows\mathscr D :\!U$ be a simplicial adjunction of left proper simplicial model categories. Then $F\dashv U$ is a Quillen adjunction if and only if $F$ preserves cofibrations and $U$ preserves fibrant objects.
\begin{proof}
See \cite[Corollary~A.3.7.2]{htt}.
\end{proof}
\end{lemma}

\begin{proof}[Proof of Proposition~\ref{prop:proj-level-to-global}]
The equivalence $(1)\Leftrightarrow(3)$ is an instance of the above lemma, while for $(1)\Rightarrow(2)$ it suffices to observe that the maps in $S_G$ are acyclic cofibrations. Thus, it only remains to prove $(2)\Rightarrow(3)$.

For this, let $X$ be fibrant and let $f\in S_G$. Then $\maps(f,UX)$ agrees by enriched adjointness up to conjugation with $\maps(Ff,X)$, so it is an acyclic fibration as $Ff$ is an acyclic cofibration and $X$ was assumed to be fibrant. But $UX$ is level fibrant, so letting $f$ vary this implies by the proof of Proposition~\ref{prop:projective-model-structure-G-gl} that $UX$ is a $G$-global $\Omega$-spectrum as desired.
\end{proof}

\begin{thm}\label{thm:sp-global-model-cat}
Let $\ul{\mathscr C}$ be a global model category. Then the projective and flat $G$-global model structures on $G\text{-}\Sp(\underline{\mathscr C})$ assemble into a global model category $\GlSp(\underline{\mathscr C})$.
\begin{proof}
Again, we may assume that $\ul{\mathscr C}$ is pointed. We have already seen that all of these model structures are left proper and combinatorial. Moreover, the projective and the flat $G$-global model structure have the same weak equivalences (namely, the $G$-global weak equivalences), and every projective cofibration is also a flat one by Lemma~\ref{lemma:level-model-structures-functoriality}.

For the functoriality properties, let $\alpha\colon G\to G'$ be any homomorphism. Then $\alpha_!\dashv\alpha^*$ is a Quillen adjunction for the projective \emph{level} model structures, hence in particular with respect to the projective \emph{level} model structure on $G\text-\Sp(\ul{\mathscr C})$ and the actual projective model structure on $G'\text-\Sp(\ul{\mathscr C})$. By Proposition~\ref{prop:proj-level-to-global} it will therefore suffice to prove that $\alpha^*$ sends any fibrant $X\in G'\text-\Sp(\underline{\mathscr C})_\text{proj}$ to a $G$-global $\Omega$-spectrum. Indeed, as we already know that $\alpha^*X$ is projectively level fibrant, this amount to saying that the adjoint structure map
\begin{equation*}
(\alpha^*X)(A)\to\Omega^B(\alpha^*X)(A\amalg B)
\end{equation*}
is a $(G\times H)$-global weak equivalence for all finite $H$-sets $A,B$. However, this map agrees with the restriction of the adjoint structure map along $\alpha\times\id\colon G\times H\to G'\times H$; as $(\alpha\times \id)^*$ is homotopical, the claim follows.

Next, we will show that $\alpha^*\dashv\alpha_*$ is a Quillen adjunction for the flat model structures; arguing as above, this amounts to saying that if $X\in G\text-\Sp(\underline{\mathscr C})_\text{flat}$ is fibrant, then $\alpha_*(X)(A)\to\Omega^B(\alpha_*X)(A\amalg B)$ is a $(G'\times H)$-global weak equivalence for all finite $H$-sets $A,B$. This is where we will need that $\underline{\mathscr C}$ is a \emph{global} (and not just a preglobal) model category: namely, we pick functorial fibrant replacements in $(G\times H)\text-\mathscr C_\text{flat}$ to get a commutative diagram
\begin{equation*}
\begin{tikzcd}
X(A)\arrow[d, "\sim"'] \arrow[r, "\sim"] & \Omega^B X(A\amalg B)\arrow[d, "\sim"]\\
Y_1\arrow[r] & Y_2\rlap.
\end{tikzcd}
\end{equation*}
Here the top arrow is a $(G\times H)$-global weak equivalence by assumption on $X$, while the vertical arrows are so by construction; thus, also the lower horizontal arrow is a weak equivalence by $2$-out-of-$3$. We want to show that applying $(\alpha\times\id)_*$ sends the top arrow to a $(G'\times H)$-global weak equivalence, for which it is then enough to show by another application of $2$-out-of-$3$ that it sends all the remaining arrows to $(G'\times H)$-global weak equivalences. For the lower horizontal arrow this is simply an instance of Ken Brown's Lemma. We will now show that also the left hand vertical arrow is sent to a $(G'\times H)$-global weak equivalence; the argument for the right hand vertical arrow will then be analogous.

For this, we observe that $X(A)$ is fibrant in $(G\times\Sigma_A)\text-\mathscr C_\text{flat}$ by definition of the $G$-global flat model structure. Thus, if $\rho\colon H\to\Sigma_A$ classifies the $H$-action on $A$ and we define $g\mathrel{:=}(\id\times\rho)\colon G\times H\to G\times\Sigma_A, q\mathrel{:=}(\alpha\times \id)\colon G\times H\to G'\times H$, then we precisely want to show that $q_*$ sends the fibrant replacement $g^*X(A)\to Y_1$ in $(G\times H)\text-{\mathscr C}_\text{flat}$ to a weak equivalence. However, as $\ker g=1\times(\ker\rho)$ has trivial intersection with $(\ker\alpha)\times1=\ker q$, this is simply an instance of what it means for a preglobal model category to be global.

As $\alpha^*$ is left Quillen for the flat global model structures, it in particular sends acyclic cofibrations to $G$-global weak equivalences. However, any acyclic \emph{fibration} in the $G'$-global flat model structure is in particular a $G'$-global level weak equivalence, and as $\underline\Sp(\underline{\mathscr C})_\text{level}$ is a preglobal model category, it follows that $\alpha^*$ sends these to $G$-global (level) weak equivalences. Thus, $\alpha^*$ is actually homotopical.

Now assume that $\alpha$ is injective. Then $\alpha^*$ is left Quillen for the projective model structures as it preserves projective (level) cofibrations by Lemma~\ref{lemma:level-model-structures-functoriality} and is homotopical by the above. Moreover, $\alpha_!\dashv\alpha^*$ is a Quillen adjunction for the flat level model structures and $\alpha^*$ sends fibrant objects to $G$-global $\Omega$-spectra by the above, so $\alpha_!\dashv\alpha^*$ is also a Quillen adjunction for the flat model structures.

Finally, let $g\colon A\to C$ and $q\colon A\to B$ be homomorphisms with $\ker(g)\cap\ker(q)=1$, let $X$ be fibrant in $C\text-\Sp(\underline{\mathscr C})_\text{flat}$, and let $\iota\colon g^*X\to Y$ be a fibrant replacement in $A\text-\Sp(\underline{\mathscr C})_\text{flat}$; we have to show that $q_*\iota$ is a $B$-global weak equivalence. However, by the above $g^*X$ is an $A$-global $\Omega$-spectrum, so $\iota$ is actually an $A$-global \emph{level} weak equivalence. The claim therefore follows from $\underline{\Sp}(\underline{\mathscr C})_\textup{level}$ being a global model category.
\end{proof}
\end{thm}

The above spectrification construction is compatible with global Quillen adjunctions and global Quillen equivalences:

\begin{prop}\label{prop:sp-induced-adjunction}
Let $\underline{F}\colon\underline{\mathscr C}\rightleftarrows\underline{\mathscr D}:\!\underline{U}$ be a global Quillen adjunction of global model categories. Then $\GlSp(\underline{F})\colon\GlSp(\underline{\mathscr C})\rightleftarrows\GlSp(\underline{\mathscr D}):\!\GlSp(\underline{U})$ is a global Quillen adjunction. If $\underline{F}\dashv\underline{U}$ is a global Quillen equivalence, then so is $\GlSp(\underline{F})\dashv\GlSp(\underline{U})$.
\begin{proof}
By Lemma~\ref{lemma:pointed-equivalence} we may assume that $\ul{\mathscr C}$ and $\ul{\mathscr D}$ are pointed. Moreover, it suffices as usual to show that $\Sp(\ul{F})\colon\Sp(\ul{\mathscr C})\rightleftarrows\Sp(\ul{\mathscr D}):\!\Sp(\ul{U})$ is a Quillen adjunction in the usual sense for the projective and flat model structures, and a Quillen equivalence (say, for the flat ones) if $\ul{F}\dashv\ul{U}$ is a global Quillen equivalence.

For the first statement, we observe that this holds for the level model structures by Lemma~\ref{lemma:sp-level-induced-adjunction}, so that it suffices that the right adjoint sends projectively fibrant objects to global $\Omega$-spectra, which is a direct consequence of Ken Brown's Lemma.

For the second statement, we let $Y\in\Sp(\ul{\mathscr D})_\text{flat}$ fibrant and $X\to GY$ a cofibrant replacement in the global flat level model structure. Then $FX\to FUY\to Y$ represents the derived counit $\cat{L}F\cat{R}UY\to Y$ for the global flat model structure, but also for the global flat \emph{level} model structure; thus, it is a global weak equivalence by Lemma~\ref{lemma:sp-level-induced-adjunction}.

The proof that also the derived unit is an isomorphism is more involved. The crucial observation for this is the following:

\begin{claim*}
Let $W\in\Sp(\ul{\mathscr C})$ be a flat global $\Omega$-spectrum. Then $\Sp(\ul F)(W)$ is again a global $\Omega$-spectrum.
\begin{proof}
Let $H$ be a finite group. As $F\dashv U$ is an $\cat{SSet}_*$-enriched adjunction, there are natural comparison isomorphisms $K\smashp F(X)\to F(K\smashp X)$ for all $K\in\cat{$\bm H$-SSet}_*, X\in H\text-\mathscr C$. Specializing to $K=S^B$ for some finite $H$-set $B$ and using that all functors in sight are left Quillen for the flat model structures, these derive to isomorphisms $S^B\smashp^{\cat L}\cat{L}F(X)\to \cat{L}F(S^B\smashp^{\cat L}X)$ in the homotopy category. Passing to canonical mates gives us a natural comparison map $\cat{L}F\cat{R}\Omega^B X\to \cat{R}\Omega^B\cat{L}F(X)$ which is again an isomorphism as $\cat{L}F$ is assumed to be an equivalence. Plugging in the definitions, this comparison map is represented for an $X$ that is cofibrant in the \emph{flat} model structure and fibrant in the \emph{projective} one by
\begin{equation}\label{eq:comparison-Omega-LF}
F(Y)\xrightarrow{F(\pi)} F(\Omega^BX) \to \Omega^B F(X) \xrightarrow{\Omega^B\iota} \Omega^B Z
\end{equation}
where $\pi\colon Y\to\Omega^BX$ is a cofibrant replacement in the flat model structure, the unlabelled arrow is the comparison map coming from the enrichment, and $\iota\colon F(X)\to Z$ is a fibrant replacement (say, in the flat model structure); in particular, the composite $(\ref{eq:comparison-Omega-LF})$ is an $H$-global weak equivalence.

With this at hand, we can now prove the claim. As $\Sp(\ul{F})$ is left Quillen with respect to the flat level model structures, we may assume without loss of generality that $W$ is also fibrant in the flat level model structure. Let now $H$ be a finite group, and let $A,B$ be finite $H$-sets; we have to show that the composite
\begin{equation}\label{eq:derived-adjunct-struc-map-LF}
F(W(A))\xrightarrow{\tilde\sigma}\Omega^B F(W(A\amalg B))\xrightarrow{\Omega^B\iota} \Omega^B Z
\end{equation}
is an $H$-global weak equivalence, where the first map is the ordinary adjoint structure map and $\iota$ is a fibrant replacement $F(W(A\amalg B))\to Z$ in $H\text-\mathscr C_\text{flat}$. For this we pick a factorization
\begin{equation*}
\begin{tikzcd}[cramped]
W(A)\arrow[r, tail, "\kappa"] & Y\arrow[r, "\pi", "\sim"'] & \Omega^B W(A\amalg B)
\end{tikzcd}
\end{equation*}
in $H\text-\mathscr C_\text{flat}$ of $\tilde\sigma$ into a cofibration $\kappa$ followed by a weak equivalence $\pi$; note that $\kappa$ is actually an acyclic cofibration by $2$-out-of-$3$ and that $Y$ is cofibrant as $W(A)$ is, so that $\pi$ is a cofibrant replacement. Then we have a commutative diagram
\begin{equation*}
\begin{tikzcd}
&[1.5em] \Omega^B F(W(A\amalg B))\arrow[r, "\Omega^B\iota"] & \Omega^B Z\\
F(W(A))\arrow[r, "F(\tilde\sigma)"{description}]\arrow[dr, "F(\kappa)"', bend right=17pt]\arrow[ur, "\tilde\sigma", bend left=15pt] & F(\Omega^B W(A\amalg B))\arrow[u] & \\
& F(Y)\arrow[u, "F(\pi)"']
\end{tikzcd}
\end{equation*}
in which $F(\kappa)$ is a weak equivalence as $F$ is left Quillen, while the composite $F(Y)\to\Omega^BZ$ agrees with $(\ref{eq:comparison-Omega-LF})$ for $X\mathrel{:=}W(A\amalg B)\in H\text-\mathscr C$, so it is a weak equivalence by the above. Thus, also the top composite $(\ref{eq:derived-adjunct-struc-map-LF})$ is a weak equivalence by $2$-out-of-$3$, finishing the proof of the claim.
\end{proof}
\end{claim*}

It suffices now to show that the derived unit $X\to\cat{R}U\cat{L}FX$ is an isomorphism whenever $X$ is a flat \emph{global $\Omega$-spectrum}. But indeed, this is represented by the composite $X\to UFX\to UY$ where $FX\to Y$ is a fibrant replacement in the global projective model structure. By the claim, $FX$ is a global $\Omega$-spectrum, so this is actually a level fibrant replacement; thus, this composite also represents the derived unit for the corresponding level model structures and the claim follows again from Lemma~\ref{lemma:sp-level-induced-adjunction}.
\end{proof}
\end{prop}

Finally, let us show that the above construction actually fulfills its purpose:

\begin{thm}\label{thm:sp-is-stable}
Let $\ul{\mathscr C}$ be a global model category. Then $\GlSp(\underline{\mathscr C})$ is stable.
\end{thm}

For the proof of the theorem, we may as usual assume without loss of generality that $\ul{\mathscr C}$ is pointed. We now want to proceed in the same way as for $\glo{GlobalSpectra}$ (see Proposition~\ref{prop:GlobalSpectra-stable} above), so we first need to introduce a suitable smash product with $G$-equivariant symmetric spectra:

\begin{constr}
Let us write $\smashp$ for the essentially unique $\cat{SSet}_*$-enriched functor $\cat{Spectra}\times\Sp(\underline{\mathscr C})\to\Sp(\underline{\mathscr C})$
that preserves tensors and colimits in each variable and such that $\bm\Sigma(A,\blank)\smashp (\bm\Sigma(B,\blank)\smashp X)=\bm\Sigma(A\amalg B,\blank)\smashp X$ with the evident functoriality in $A,B\in\bm\Sigma$ and $X\in\mathscr C$. For any finite group $G$, we then obtain a pairing
\begin{equation}\label{eq:tensoring-g-global}
\blank\smashp\blank\colon\cat{$\bm G$-Spectra}\times G\text{-}\Sp(\underline{\mathscr C})\to G\text{-}\Sp(\underline{\mathscr C})
\end{equation}
by pulling through the $G$-actions.
\end{constr}

There is then a unique way to extend the evident isomorphisms
\begin{align*}
\bm\Sigma(\varnothing,\blank)\smashp (\bm\Sigma(B,\blank)\smashp X)&= \bm\Sigma(\varnothing\amalg B,\blank)\smashp X\cong\bm\Sigma(B,\blank)\smashp X
\intertext{and}
(\bm\Sigma(A,\blank)\smashp\bm\Sigma(B,\blank))\smashp(\bm\Sigma(C,\blank)\smashp X)&=\bm\Sigma((A\amalg B)\amalg C,\blank)\smashp X\\&\cong \bm\Sigma(A\amalg(B\amalg C),\blank)\smashp X\\&=\bm\Sigma(A,\blank)\smashp(\bm\Sigma(B,\blank)\smashp (\bm\Sigma(C,\blank)\smashp X))
\end{align*}
to make $G\text-\Sp(\ul{\mathscr C})$ tensored over $\cat{$\bm G$-Spectra}$. Since all adjoints exist by local presentability, $G\text-\Sp(\ul{\mathscr C})$ is then also enriched and cotensored over $\cat{$\bm G$-Spectra}$; the next proposition can therefore be reformulated as saying that $G\text-\Sp(\ul{\mathscr C})_\text{proj}$ is enriched in the model categorical sense over $\cat{$\bm G$-Spectra}_\text{$G$-equiv.~proj.}$.

\begin{prop}\label{prop:tensoring-omega}
The pairing $(\ref{eq:tensoring-g-global})$ is a left Quillen bifunctor with respect to the $G$-equivariant projective model structure on $\cat{$\bm G$-Spectra}$ and the $G$-global projective model structure on $G\text-\Sp(\ul{\mathscr C})$.
\end{prop}

For the proof we will need:

\begin{lemma}\label{lemma:shift-right-Quillen}
Let $A$ be any finite $G$-set. Then the shift $\sh^A\colon G\text-\Sp(\underline{\mathscr C})\to G\text-\Sp(\underline{\mathscr C})$ is right Quillen for the $G$-global projective model structures.
\begin{proof}
As before, it suffices to show that $\sh^A$ is right Quillen for the corresponding level model structures and that it sends fibrant objects to $G$-global $\Omega$-spectra.

For the first statement, we simply note that $\ev_B\circ\sh^A\colon G\text{-}\Sp(\mathscr C)\to (G\times\Sigma_B)\text-\mathscr C$ factors as the composition
\begin{equation*}
G\text{-}\Sp(\underline{\mathscr C})_\text{proj.~lev.}\hskip0pt minus 2pt\xrightarrow{\hskip0pt minus 1pt\ev_{A\amalg B}}\hskip0pt minus 2pt(G\times\Sigma_{A\amalg B})\text{-}\mathscr C_\text{proj}\hskip0pt minus 2pt\xrightarrow{\textup{res}}\hskip0pt minus 2pt(G\times\Sigma_A\times\Sigma_B)\text{-}\mathscr C_\text{proj}\hskip0pt minus 2pt\xrightarrow{\phi^*\hskip0pt minus 3pt}\hskip0pt minus 2pt (G\times\Sigma_B)\text-\mathscr C_\text{proj}
\end{equation*}
where $\phi$ is defined via $\phi(g,\sigma)=(g, g.\blank,\sigma)$, and each of these is right Quillen by definition.

Finally, if $X$ is fibrant, then $\sh^A X$ is level fibrant by the above, hence $\sh^AX$ will be a $G$-global $\Omega$-spectrum if and only if $X(A\amalg B)\to\Omega^{C} X(A\amalg B\amalg C)$ is a $(G\times H)$-global weak equivalence for all finite $H$-sets $B,C$. This follows easily from Remark~\ref{rk:g-global-omega-g-h-sets} (letting $H$ act trivially on $A$ and $G$ act trivially on $B,C$).
\end{proof}
\end{lemma}

\begin{proof}[Proof of Proposition~\ref{prop:tensoring-omega}]
Let us first prove the analogous statement where we equip $\cat{$\bm G$-Spectra}$ with the $G$-equivariant projective \emph{level} model structure. Adjoining over, it will be enough here to show that
\begin{align*}
G\text{-}\Sp(\underline{\mathscr C})^\op\times G\text{-}\Sp(\underline{\mathscr C})&\to \cat{$\bm G$-Spectra}\\
X, Y&\mapsto \big(A\mapsto \maps(X, \sh^A Y)\big)
\end{align*}
(with the evident functoriality in $A,X,Y$) is a right Quillen bifunctor. As all structure in sight is defined levelwise, this amounts to saying that $(X,Y)\mapsto \maps(X,\sh^AY)$ is a right Quillen bifunctor to $\cat{$\bm{(G\times\Sigma_A)}$-SSet}_{*}$ with the $\mathcal G_{G,\Sigma_A}$-model structure. However, as $\underline{\Sp}(\underline{\mathscr C})$ is a global model category, $\text{triv}_{\Sigma_A}\colon G\text-\Sp(\underline{\mathscr C})\to(G\times\Sigma_A)\text-\Sp(\underline{\mathscr C})$ is right Quillen for the projective model structures, and so is the endofunctor $\sh^A$ of $(G\times\Sigma_A)\text-\Sp(\underline{\mathscr C})$ (where $G$ acts trivially on $A$ and $\Sigma_A$ acts tautologically) by the previous lemma. Thus, the claim follows from Corollary~\ref{cor:pointed-g-sset-enriched} applied to the global model category $\ul{\Sp}(\ul{\mathscr C})$.

For the original statement, all that remains now is to show that the pushout product of any standard generating acyclic cofibration $j$ of $\cat{$\bm G$-Spectra}_\text{$G$-equiv.~proj.}$ with any generating cofibration $i$ of $G\text{-}\Sp(\ul{\mathscr C})$ is again a $G$-global weak equivalence. However, if $j$ is even a level weak equivalence, this follows from the above, while the other generating acyclic cofibrations are precisely the maps $(G_+\smashp_H\kappa_{H,A,B})\ppo k$ for subgroups $H\subset G$, finite $H$-sets $A,B$, and (generating) cofibrations $k$ of $G\text-\Sp(\cat{SSet})_\text{$G$-equiv.~proj.}$, see Remark~\ref{rk:equivariant-proj-gen-cof}. The pushout product $(G_+\smashp_H\kappa_{H,A,B}\ppo k)\ppo i$ can then be identified with $\kappa_{H,A,B}\ppo_H(G_+\smashp (k\ppo i))$ where $H$ acts on the second factor via its right action on $G$. By the above, $k\ppo i$ is a (level) cofibration; by construction of the $G$-global projective model structure it therefore only remains to show that
\begin{equation*}
G_+\smashp\blank\colon G\text-\Sp(\underline{\mathscr C})\to (G\times H)\text-\Sp(\underline{\mathscr C})
\end{equation*}
(where $G$ acts diagonally and $H$ acts from the right) is left Quillen for the projective (level) model structures. However, as $\ul{\Sp}(\ul{\mathscr C})$ is a (pre)global model category and $G$ is cofibrant in $\cat{$\bm{(G\times H)}$-SSet}_{\mathcal G_{G,H}}$ ($H$ acting freely from the right), this is simply an instance of Corollary~\ref{cor:pointed-g-sset-enriched} again.
\end{proof}

\begin{proof}[Proof of Theorem~\ref{thm:sp-is-stable}]
This follows from Proposition~\ref{prop:tensoring-omega} in the same way as in the proof of Proposition~\ref{prop:GlobalSpectra-stable}.
\end{proof}

\subsection{The universal property of spectrification} Throughout, let $\ul{\mathscr C}$ and $\ul{\mathscr D}$ be global model categories. We begin by explaining in which sense the above construction $\ul{\Sp}$ is idempotent:

\begin{lemma}
Assume $\ul{\mathscr C}$ is pointed. Then we have a global Quillen adjunction
\begin{equation}\label{eq:stabilization-adjunction}
\ul{\Sigma^\infty}\mathrel{:=}\ul{\mathbb S\smashp\blank}\colon\underline{\mathscr C}\rightleftarrows\GlSp(\underline{\mathscr C}) :\!\ul{\Omega^\infty}\mathrel{:=}\ul{\ev_\varnothing}.
\end{equation}
\begin{proof}
It is clear from the definitions that $\ul{\Omega^\infty}$ is right Quillen (already for the level model structures).
\end{proof}
\end{lemma}

\begin{cor}
Let $\ul{\mathscr C}$ be any global model category. Then we have a global Quillen adjunction
\begin{equation*}
\ul{\Sigma^\infty_+}\mathrel{:=}\ul{\mathbb S\smashp(\blank)_+}\colon\ul{\mathscr C}\rightleftarrows\ul{\Sp}(\ul{\mathscr C}) :\!\ul{\Omega^\infty}\mathrel{:=}\ul{\ev_\varnothing}.\pushQED{\qed}\qedhere\popQED
\end{equation*}
\end{cor}

\begin{thm}\label{thm:stabilization-preserves-stable}
Let $\underline{\mathscr C}$ be a global model category. Then $\underline{\mathscr C}$ is stable if and only if it is pointed and the global Quillen adjunction $(\ref{eq:stabilization-adjunction})$ is a global Quillen equivalence.
\begin{proof}
Stability is preserved under global Quillen equivalences (Lemma~\ref{lemma:stability-stable}), so `$\Leftarrow$' follows from Theorem~\ref{thm:sp-is-stable}. By the usual shifting argument it now suffices to show that $\Sigma^\infty\colon\mathscr C\rightleftarrows\Sp(\underline{\mathscr C}) :\!\Omega^\infty$ is a Quillen equivalence whenever $\ul{\mathscr C}$ is stable.

For this we first observe that $\cat{L}\Sigma^\infty$ takes values in global $\Omega$-spectra. Namely, if $X\in \mathscr C_\text{flat}$ is cofibrant, then the ordinary structure maps $S^B\smashp (\Sigma^\infty X)(A)\to(\Sigma^\infty X)(A\amalg B)$ are isomorphisms for all finite groups $H$ and finite $H$-sets $A$ and $B$. However, $(\Sigma^\infty X)(A)=S^A\smashp X$ is cofibrant in the flat $H$-global model structure, so we can identify this with the derived map $S^B\smashp^{\cat L}(\Sigma^\infty X)(A)\to (\Sigma^\infty X)(A\amalg B)$. By stability, the adjoint map $(\Sigma^\infty X)(A)\to \cat{R}\Omega^B(\Sigma^\infty X)(A\amalg B)$ is then also an $H$-global weak equivalence as desired.

It follows immediately that for every cofibrant $X\in \mathscr C_\text{flat}$ the ordinary unit $X\to\Omega^\infty\Sigma^\infty X$ already represents the derived unit. As the former is an isomorphism, the derived unit is a global weak equivalence.

To complete the proof, it is now enough to show that $\Omega^\infty$ reflects weak equivalences between global $\Omega$-spectra. But indeed, if $f\colon X\to Y$ is a map of global $\Omega$-spectra such that $f(\varnothing)$ is a global weak equivalence, then $\cat{R}\Omega^A f(A)$ is a $\Sigma_A$-global weak equivalence as it is conjugate to $f(\varnothing)$ (equipped with the trivial $\Sigma_A$-action). Thus, $f(A)$ must be a $\Sigma_A$-global weak equivalence as $\Omega^A\colon \Sigma_A\text-\mathscr C\to\Sigma_A\text-\mathscr C$ is part of a Quillen equivalence (say, for the flat model structures).
\end{proof}
\end{thm}

\begin{defi}\label{defi:global-stabilization}
A global Quillen adjunction $\underline{\mathscr C}\rightleftarrows\underline{\mathscr D}$ is called a \emph{global stabilization} if $\underline{\mathscr D}$ is stable and the induced global Quillen adjunction $\GlSp(\underline{\mathscr C})\rightleftarrows\GlSp(\underline{\mathscr D})$ is a global Quillen equivalence.
\end{defi}

\begin{thm}\label{thm:sp-is-stabilization}
Let $\underline{\mathscr C}$ be any global model category. Then $\ul{\Sigma^\infty_+}\colon\underline{\mathscr C}\rightleftarrows\GlSp(\underline{\mathscr C}):\!\ul{\Omega^\infty}$ is a global stabilization.
\begin{proof}
We have already shown in Theorem~\ref{thm:sp-is-stable} that $\GlSp(\underline{\mathscr C})$ is stable. Replacing $\underline{\mathscr C}$ by $\underline{G\text{-}\mathscr C}$ as usual, it then only remains that $\Sp(\underline{\Sigma^\infty_+})\colon\Sp(\underline{\mathscr C})\rightleftarrows\Sp(\GlSp(\underline{\mathscr C})):\Sp(\underline{\Omega^\infty})$ is a Quillen equivalence. Comparing the right adjoints we immediately see that this agrees with $\Sp(\ul{\Sigma^\infty}):\Sp(\ul{\mathscr C}_*)\rightleftarrows\Sp(\Sp(\ul{\mathscr C}_*)):\Sp(\ul{\Omega^\infty})$, so we may further assume that $\ul{\mathscr C}$ is pointed.

For the proof, we identify the $\cat{SSet}_*$-category $\Sp(\GlSp(\underline{\mathscr C}))$ with the category of $\cat{SSet}_*$-enriched functors $\bm\Sigma\otimes\bm\Sigma\to\mathscr C$ where $\bm\Sigma\otimes\bm\Sigma$ denotes the $\cat{SSet}_*$-category with objects the pairs $(A,B)$ of finite sets and with mapping spaces
\begin{equation*}
\maps_{\bm\Sigma\otimes\bm\Sigma}((A,B),(A',B'))=\maps_{\bm\Sigma}(A,A')\smashp\maps_{\bm\Sigma}(B,B')
\end{equation*}
with the evident composition. We take the convention that the first factor of $\bm\Sigma\otimes\bm\Sigma$ corresponds to the `outer $\Sp$', so $\Sp(\underline{\Omega^\infty})$ corresponds to restriction along the inclusion $i_1\colon\bm\Sigma\to \bm\Sigma\otimes\bm\Sigma,A\mapsto(A,\varnothing)$. On the other hand, restricting along the other inclusion $i_2\colon\bm\Sigma\to\bm\Sigma\otimes\bm\Sigma$ corresponds to the Quillen equivalence $\Omega^\infty$.

We now write $\Pi\colon\bm\Sigma\otimes\bm\Sigma\to\bm\Sigma$ for the usual symmetric monoidal structure (used to construct the smash product of spectra), given on objects by $(A,B)\mapsto A\amalg B$. By restricting, this gives rise to a simplicial functor $\Pi^*\colon\Sp(\underline{\mathscr C})\to\Sp(\GlSp(\underline{\mathscr C}))$, which admits a simplicial left adjoint $\Pi_!$ by enriched Kan extension.

\begin{claim*}
The simplicial adjunction $\Pi_!\dashv\Pi^*$ is a Quillen adjunction with respect to the projective model structures.
\begin{proof}
Let us consider the case of the level model structures first, which amounts to saying that $\Sp(\ul{\mathscr C})\to \Sigma_A\text-\Sp(\ul{\mathscr C}), X\mapsto (\Pi^*X)(A)$ is right Quillen for the projective model structures everywhere. However, this can be identified with the composition
\begin{equation*}
\Sp(\underline{\mathscr C})_\text{proj}\xrightarrow{\text{triv}_{\Sigma_A}}\Sigma_A\text-\Sp(\underline{\mathscr C})_\text{proj}\xrightarrow{\sh^A}\Sigma_A\text-\Sp(\underline{\mathscr C})_\text{proj}
\end{equation*}
of which the first functor is right Quillen as $\underline{\Sp}(\underline{\mathscr C})$ is a global model category while the second one is so by Lemma~\ref{lemma:shift-right-Quillen}.

It then only remains to show that if $X$ is fibrant in the global projective model structure on $\Sp(\underline{\mathscr C})$, then $\Pi^*X$ is a global $\Omega$-spectrum, i.e.~for every finite group $H$ and every finite $H$-sets $A,B$, the map $\sh^A X\to\Omega^B\sh^{A\amalg B}X$ is an $H$-global weak equivalence. But indeed, this is even an $H$-global level weak equivalence: if $C$ is any finite set, then after evaluating at $C$ this is simply the adjoint structure map $X(C\amalg A)\to \Omega^B(C\amalg A\amalg B)$, which is an $(H\times\Sigma_C)$-global weak equivalence since $X$ was assumed to be a global $\Omega$-spectrum (letting $\Sigma_C$ act trivially on $A$ and $B$ while $H$ acts trivially on $C$).
\end{proof}
\end{claim*}

The unitality isomorphisms of the symmetric monoidal structure on $\bm\Sigma$ now give us isomorphisms $\Omega^\infty\circ\Pi^*\cong\id\cong \Sp(\underline{\Omega^\infty})\circ\Pi^*$. As all functors are right Quillen, this induces isomorphisms $\cat{R}\Omega^\infty\circ\cat{R}\Pi^*\cong\id\cong \cat{R}\Sp(\ul{\Omega^\infty})\circ\cat{R}\Pi^*$ of derived functors. In particular, as $\cat{R}\Omega^\infty$ is an equivalence (Theorems~\ref{thm:sp-is-stable} and~\ref{thm:stabilization-preserves-stable}), also $\cat{R}\Pi^*$ is an equivalence by $2$-out-of-$3$, and hence so is $\cat{R}\Sp(\underline{\Omega^\infty})$ by the same argument.
\end{proof}
\end{thm}

\begin{rk}
Write $\cat{GLOBMOD}$ for the opposite of the large category of global model categories and global right Quillen functors, localized with respect to the global Quillen equivalences. Then Proposition~\ref{prop:sp-induced-adjunction} implies that $\ul\Sp$ descends to an endofunctor $Q$ of $\cat{GLOBMOD}$. Moreover the global Quillen adjunctions $\smash{\ul{\Sigma^\infty_+}\dashv \ul{\Omega^\infty}}$ induce a natural transformation $\eta\colon\id\Rightarrow Q$.\footnote{Here it comes in handy that we defined the underlying $1$-category to consist of right Quillen functors: $\ul{\Omega^\infty}=\ul{\ev_\varnothing}$ is strictly natural while $\ul{\Sigma^\infty_+}$ is only pseudonatural. However, the approach via left Quillen functors or by encoding both adjoints at the same time could also be made to work, as it is not hard to show using a cocylinder argument that isomorphic functors become \emph{equal} after localizing at the global Quillen equivalences in either case.}

By Theorems~\ref{thm:sp-is-stable} and~\ref{thm:stabilization-preserves-stable}, the map $\eta_{Q\ul{\mathscr C}}$ is an isomorphism in $\cat{GLOBMOD}$ for every $\ul{\mathscr C}$; similarly, Theorem~\ref{thm:sp-is-stabilization} above shows that $Q\eta_{\ul{\mathscr C}}$ is an isomorphism. Thus, it follows by abstract nonsense that $Q$ is a Bousfield localization onto its essential image (i.e.~the stable global model categories) with unit given by $\eta$. Put differently, for any global model category $\ul{\mathscr C}$, $\ul{\Omega^\infty}\colon\Sp(\ul{\mathscr C})\to \ul{\mathscr C}$ (or more generally the right adjoint in any global stabilization in the sense of Definition~\ref{defi:global-stabilization}) is the homotopy universal example of a global right Quillen functor from a \emph{stable} global model category, and dually for the left adjoints.
\end{rk}

\section{The stabilization of global spaces}\label{sec:global-spectra}
In this section we will prove:

\begin{thm}\label{thm:stabilization-global-spaces}
The global Quillen adjunction
\begin{equation}\label{eq:global-stabilization-S-Sp}
\ul{\Sigma^\bullet_+}\colon\glo{GlobalSpaces}\rightleftarrows\glo{GlobalSpectra} :\!\ul{\Omega^\bullet}
\end{equation}
from Example~\ref{ex:S-Sp-adjunction} is a global stabilization.
\end{thm}

This can be seen as a sanity check for our framework, but it also allows us to describe the passage from unstable to stable global homotopy theory via a universal property: $(\ref{eq:global-stabilization-S-Sp})$ is the homotopy-universal example of a global Quillen adjunction from $\glo{GlobalSpaces}$ to a stable global model category.

For the proof, we will compare $\glo{GlobalSpectra}$ to the global stabilization constructed in the previous section. This will require some preparations:

\begin{constr}
We define a functor $\Delta^*\colon\Sp(\glo{GlobalSpaces})\to\cat{Spectra}$ as follows: if $X$ is a spectrum in pointed $\mathcal I$-simplicial sets, then $\Delta^*X(A)=X(A)(A)$ with structure maps
\begin{equation*}
S^{B\setminus i(A)} \smashp X(A)(A) \xrightarrow{\sigma} X(B)(A) \xrightarrow{X(B)(i)} X(B)(B)
\end{equation*}
for every injection $i\colon A\to B$ and with the evident enriched functoriality in $X$.

By the enriched Yoneda Lemma, $\Delta^*$ admits an enriched left adjoint $\Delta_!$ which is characterized up to unique enriched isomorphism by the condition that it preserves colimits and tensors and satisfies $\Delta_!\bm\Sigma(A,\blank)=\bm\Sigma(A,\blank)\smashp \mathcal I(A,\blank)_+$ with the evident functoriality in $A$.
Writing simplices of $\bm\Sigma(A,B)$ for $B\in\bm\Sigma$ as equivalence classes $[i,\sigma]$ of an injection $i\colon A\to B$ and a simplex $\sigma$ of $S^{B\setminus i(A)}$, the unit of $\bm\Sigma(A,\blank)$ is then given in degree $B$ by $\bm\Sigma(A,B)\to\bm\Sigma(A,B)\smashp\mathcal I(A,B)_+, [i,\sigma]\mapsto [i,\sigma]\smashp i$.

By enriched Kan extension, $\Delta^*$ moreover admits a simplicial right adjoint $\Delta_*$.
\end{constr}

Our actual goal now is to prove:

\begin{thm}\label{thm:delta-quillen-equivalence}
For any finite group $G$ the simplicial adjunction
\begin{equation}\label{eq:delta-shriek-star}
\Delta_!\colon \cat{$\bm G$-Spectra}_\textup{$G$-global projective}\rightleftarrows G\text-\Sp(\glo{GlobalSpaces})_\textup{$G$-global projective} :\!\Delta^*
\end{equation}
is a Quillen equivalence.
\end{thm}

The proof will occupy the remainder of this section; for now, let us already remark how it implies Theorem~\ref{thm:stabilization-global-spaces}:

\begin{proof}[Proof of Theorem~\ref{thm:stabilization-global-spaces}]
We already know that $\ul{\Sigma^\bullet_+}\dashv\ul{\Omega^\bullet}$ is a global Quillen adjunction and that $\glo{GlobalSpectra}$ is stable (Proposition~\ref{prop:GlobalSpectra-stable}). Thus, it only remains to show that $G\text-\Sp(\ul{\Omega^\bullet})$ derives to an equivalence of homotopy categories for every finite $G$. To this end, we consider the diagram
\begin{equation}\label{diag:omega-delta}
\begin{tikzcd}[column sep=tiny]
G\text-\Sp(\glo{GlobalSpectra})_\text{proj.~$G$-gl.}\arrow[dr, bend right=10pt, "G\text-\Sp(\underline{\Omega^\bullet})"']\arrow[rr, "\Omega^\infty"] && \cat{$\bm G$-Spectra}_\text{proj.~$G$-gl.}\\
& G\text-\Sp(\glo{GlobalSpaces})_\text{proj.~$G$-gl.}\arrow[ur, bend right=10pt, "\Delta^*"']
\end{tikzcd}
\end{equation}
of right Quillen functors. This does not commute strictly, but after restricting to fibrant objects, Remark~\ref{rk:Omega-bullet-on-fibrant} shows that the lower composite is given up to natural $G$-global (level) weak equivalence by sending a $G$-bispectrum $X$ to the spectrum $\delta(X)$ given by $A\mapsto\Omega^A X(A)(A)$ with the evident functoriality. We moreover have a natural map $\sigma\colon\Omega^\infty X\to \delta(X)$ given in degree $A$ by the map $X(\varnothing)(A)\to \Omega^AX(A)(A)$ induced by the adjunct structure map. If $X$ is fibrant, then each $X(\varnothing)\to\Omega^AX(A)$ is a $(G\times\Sigma_A)$-global weak equivalence of fibrant objects in the projective $(G\times\Sigma_A)$-global model structure, hence a $(G\times\Sigma_A)$-global level weak equivalence. In particular, after evaluating at $A$, this is a $\mathcal G_{\Sigma_A,G\times\Sigma_A}$-equivariant weak equivalence, hence a $\mathcal G_{\Sigma_A,G}$-equivariant weak equivalence with respect to the diagonal $\Sigma_A$-action. Thus, $\sigma$ is a $G$-global (level) weak equivalence for every fibrant $X$. Altogether, we therefore have an isomorphism $\cat{R}\Omega^\infty\cong\cat{R}(G\text-\Sp(\ul{\Omega^\bullet}))\circ\cat{R}\Delta^*$ of right derived functors.

However, the top arrow in $(\ref{diag:omega-delta})$ induces an equivalence of homotopy categories by stability of $\glo{GlobalSpectra}$ (Proposition~\ref{prop:GlobalSpectra-stable}), and so does $\Delta^*$ by the previous theorem. The claim now follows from $2$-out-of-$3$.
\end{proof}

It remains to prove Theorem~\ref{thm:delta-quillen-equivalence}.

\begin{lemma}
The simplicial adjunction $(\ref{eq:delta-shriek-star})$ is a Quillen adjunction.
\begin{proof}
We will first show that $\Delta^*$ preserves (acyclic) level fibrations. Indeed, if $f\colon X\to Y$ is any map in $G\text-\Sp(\glo{GlobalSpaces})$, then $(\Delta^*f)(A)=f(A)(A)$. If now $f$ is a $G$-global (acyclic) level fibration, then $f(A)$ is a $(G\times\Sigma_A)$-global (acyclic) fibration, hence in particular a $(G\times\Sigma_A)$-global (acyclic) level fibration, so that $f(A)(A)$ is an (acyclic) $\mathcal G_{\Sigma_A,G\times\Sigma_A}$-equivariant fibration. As before, it is then in particular a $\mathcal G_{\Sigma_A,G}$-equivariant (acyclic) fibration with respect to the diagonal action, i.e.~$\Delta^*f$ is an (acyclic) level fibration as desired.

To complete the proof, it now only remains to show that $\Delta^*$ sends fibrant objects to $G$-global $\Omega$-spectra. But indeed, if $X$ is fibrant and $A,B$ are finite $H$-sets, then $X(A)\to\Omega^BX(A\amalg B)$ is a $(G\times H)$-global weak equivalence between projectively fibrant $(G\times H)$-global spaces, hence a $(G\times H)$-global level weak equivalence. In particular, $(\Delta^*X)(A)=X(A)(A)\to\Omega^BX(A\amalg B)(A)$ is a $\mathcal G_{\Sigma_A,G\times H}$-equivariant weak equivalence. Thus, if $A$ is \emph{faithful}, then this is a $\mathcal G_{H,G\times H}$-weak equivalence by Lemma~\ref{lemma:graph-source} and whence a $\mathcal G_{H,G}$-equivariant weak equivalence with respect to the diagonal action.

Now the adjoint structure map of $\Delta^*X$ factors as
\begin{equation*}
X(A)(A)\xrightarrow{\tilde\sigma(A)}\Omega^BX(A\amalg B)(A)\xrightarrow{\Omega^BX(A\amalg B)(\text{incl})} \Omega^BX(A\amalg B)(A\amalg B);
\end{equation*}
if $H$ acts faithfully on $A$, then the first map is a $\mathcal G_{H,G}$-equivariant weak equivalence by the above, and so is the second one by the same computation as $X(A\amalg B)$ is a fibrant $(G\times H)$-$\mathcal I$-space.
\end{proof}
\end{lemma}

\begin{lemma}\label{lemma:Delta-star-conservative}
The functor $\Delta^*\colon G\text-\Sp(\glo{GlobalSpaces})_\textup{$G$-global proj.}\to\cat{$\bm G$-Spectra}_\textup{$G$-global proj.}$ reflects $G$-global weak equivalences between fibrant objects, i.e.~its right derived functor $\cat{R}\Delta^*$ is conservative.
\begin{proof}
Let $f\colon X\to Y$ be a map of fibrant objects such that $\Delta^*f$ is a $G$-global weak equivalence. As $\Delta^*$ is right Quillen, $\Delta^*X$ and $\Delta^*Y$ are fibrant, so $\Delta^*f$ is even a $G$-global level weak equivalence, i.e.~$f(A)(A)$ is a $\mathcal G_{H,G}$-equivariant weak equivalence of $(G\times H)$-simplicial sets for every finite group $H$ and any finite faithful $H$-set $A$. We want to show that $f(A)$ is already a $(G\times\Sigma_A)$-global level weak equivalence, i.e.~that for every finite set $B$ the map $f(A)(B)$ is a $\mathcal G_{\Sigma_B,G\times\Sigma_A}$-equivariant weak equivalence, or put differently that for every finite group $H$ acting arbitrarily on $A$ and faithfully on $B$ the map $f(A)(B)$ is a $\mathcal G_{H,G}$-equivariant weak equivalence.

For this we consider the commutative diagram
\begin{equation*}
\begin{tikzcd}
X(A)(B)\arrow[r,"\tilde\sigma"]\arrow[d,"f(A)(B)"'] & \Omega^B X(A\amalg B)(B)\arrow[d, "\Omega^B f(A\amalg B)(B)"]\arrow[r, "\Omega^B X(A\amalg B)(\textup{incl})"] &[5em] \Omega^B(\Delta^*X)(A\amalg B)\arrow[d, "\Omega^B(\Delta^*f)(A\amalg B)"]\\
Y(A)(B)\arrow[r,"\tilde\sigma"'] & \Omega^B Y(A\amalg B)(B)\arrow[r, "\Omega^B X(A\amalg B)(\textup{incl})"'] & \Omega^B(\Delta^*Y)(A\amalg B)
\end{tikzcd}
\end{equation*}
As $H$ acts faitfully on $A\amalg B$, the right hand vertical map is a $\mathcal G_{H,G}$-equivariant weak equivalence by the above. Moreover, as $X(A)\to\Omega^BX(A\amalg B)$ is a $(G\times H)$-global weak equivalence of fibrant objects and $H$ acts faithfully on $B$, the top left horizontal map is a $\mathcal G_{H,G}$-weak equivalence, and so is the lower left horizontal map by the same argument. Finally, $X(A\amalg B)$ is a fibrant $(G\times H)$-$\mathcal I$-space, so that the top right horizontal map is a $\mathcal G_{H,G}$-equivariant weak equivalence, and likewise for the lower right horizontal map. The claim now follows by $2$-out-of-$3$.
\end{proof}
\end{lemma}

We now want to prove the following strengthening of the above lemma:

\begin{prop}\label{prop:Delta-star-homotopical}
The functor $\Delta^*$ creates $G$-global weak equivalences.
\end{prop}

For the proof, it will be crucial to understand the behaviour of $\Delta^*$ on generating (acyclic) \emph{cofibrations}, and more generally on maps of the form $\bm\Sigma(A,\blank)\smashp_{\Sigma_A}f$ for maps $f$ in $\cat{$\bm{(G\times\Sigma_A)}$-$\bm{\mathcal I}$-SSet}_*$. For this we recall the following standard construction:

\begin{constr}\label{constr:tensor-Sp-I}
Let $X$ be a spectrum and let $Y$ be a pointed $\mathcal I$-simplicial set. Then we write $X\otimes Y$ for the spectrum with $(X\otimes Y)(A)=X(A)\smashp Y(A)$ and structure maps
\begin{equation*}
S^{B\setminus i(A)}\smashp (X\otimes Y)(A)=S^{B\setminus A}\smashp X(A)\smashp Y(A)\xrightarrow{\sigma\smashp Y(i)} X(B)\smashp Y(B)=(X\otimes Y)(B)
\end{equation*}
for every injection $i\colon A\to B$. This becomes a functor $\cat{Spectra}\times\cat{$\bm{\mathcal I}$-SSet}_*\to\cat{Spectra}$ in the evident way, which we as usual promote to
\begin{equation}\label{eq:tensor-product}
\blank\otimes\blank\colon\cat{$\bm G$-Spectra}\times\cat{$\bm G$-$\bm{\mathcal I}$-SSet}_*\to\cat{$\bm G$-Spectra}.
\end{equation}
\end{constr}

Put differently we have for any $G$-spectrum $X$ and any $G$-$\mathcal I$-simplicial set $Y$ an equality $X\otimes Y=\Delta^*(X\smashp Y)$ for the levelwise smash product $\cat{$\bm G$-Spectra}\times\cat{$\bm G$-$\bm{\mathcal I}$-SSet}_*\to G\text-\Sp(\glo{GlobalSpaces}_*)=G\text-\Sp(\glo{GlobalSpaces})$, and likewise for maps.

\begin{prop}\label{prop:tensor-homotopical}
The tensor product $(\ref{eq:tensor-product})$ preserves $G$-global weak equivalences in each variable.
\begin{proof}
Let $X$ be a $G$-spectrum and let $g\colon Y\to Y'$ be a $G$-global weak equivalence in $\cat{$\bm G$-$\bm{\mathcal I}$-SSet}_*$. We will show that $X\otimes g$ is even a $\ul\pi_*$-isomorphism, for which we let $\phi\colon H\to G$ be any homomorphism. Then the effect of $\phi^*(X\otimes g)=(\phi^*X)\otimes(\phi^*g)$ on $\pi^H_*$ agrees up to conjugation by isomorphisms with the one of $\phi^*X\smashp (\phi^*g)(\mathcal U_H)$, see \cite[Lemma~3.2.11]{g-global}. But $(\phi^*g)(\mathcal U_H)$ is an $H$-equivariant weak equivalence, so $\phi^*X\smashp (\phi^*g)(\mathcal U_H)$ is even an $H$-equivariant level weak equivalence, in particular a $\ul\pi_*$-isomorphism. This completes the proof that the tensor product preserves $G$-global weak equivalences in the second variable.

On the other hand, let $f\colon X\to X'$ be a $G$-global weak equivalence of $G$-global spectra and let $Y$ be any $G$-global space; we want to show that $f\otimes Y$ is a $G$-global weak equivalence. Arguing precisely as above, we see that $\blank\otimes Y$ preserves $\ul\pi_*$-isomorphisms, so we may assume without loss of generality that $f$ is a map between flat $G$-spectra. Under this assumption, \cite[Proposition~3.2.14]{g-global} provides us with a commutative diagram
\begin{equation*}
\begin{tikzcd}
X\smashp\Sigma^\bullet Y\arrow[r, "\psi", "\sim"']\arrow[d, "f\smashp\Sigma^\bullet Y"'] & X\otimes Y \arrow[d, "f\otimes Y"]\\
X'\smashp\Sigma^\bullet Y\arrow[r, "\psi"', "\sim"] & X'\otimes Y
\end{tikzcd}
\end{equation*}
in which the horizontal maps are $\ul\pi_*$-isomorphisms (in particular $G$-global weak equivalences). By Proposition~\ref{prop:flatness-theorem} the left hand vertical arrow is a $G$-global weak equivalence, so the claim follows by $2$-out-of-$3$.
\end{proof}
\end{prop}

\begin{proof}[Proof of Proposition~\ref{prop:Delta-star-homotopical}]
By Lemma~\ref{lemma:Delta-star-conservative}, it suffices to prove that $\Delta^*$ is homotopical. We already know that $\Delta^*$ preserves acyclic fibrations, so it only remains to show that it also sends acyclic cofibrations to weak equivalences. To this end, we will show that $\Delta^*$ is in fact also \emph{left} Quillen as a functor to $\cat{$\bm G$-Spectra}_\textup{injective $G$-global}$.

Let us first prove that $\Delta^*$ is left Quillen as a functor $G\text-\Sp(\glo{GlobalSpaces})_\text{proj.~$G$-global level}\to\cat{$\bm G$-Spectra}_\text{inj.~$G$-global}$. Indeed, it clearly sends generating cofibrations to injective cofibrations, so we only need to show that it sends generating acyclic cofibrations to weak equivalences. Such a generating acyclic cofibration is now of the form $\bm\Sigma(A,\blank)\smashp_{\Sigma_A} j_+$ for some $A\in\bm\Sigma$ and $j$ a (generating) acyclic cofibration in $\cat{$\bm{(G\times\Sigma_A)}$-$\bm{\mathcal I}$-SSet}$. Then
\begin{equation}\label{eq:Delta-star-gen-acyc-level}
\Delta^*(\bm\Sigma(A,\blank)\smashp_{\Sigma_A}j_+)=(\bm\Sigma(A,\blank)\otimes j_+)/\Sigma_A;
\end{equation}
but $\bm\Sigma(A,\blank)\otimes j_+$ is a $(G\times\Sigma_A)$-global weak equivalence by Proposition~\ref{prop:tensor-homotopical}, and $\Sigma_A$ acts freely on $\bm\Sigma(A,\blank)$, so Proposition~\ref{prop:free-quotient-spectra} shows that the quotient $(\ref{eq:Delta-star-gen-acyc-level})$ is a $G$-global weak equivalence. This completes the argument for the level model structure.

For the actual claim, it suffices now by Proposition~\ref{prop:proj-level-to-global} to show that $\Delta^*$ sends the maps in the set $S_G$ from $(\ref{eq:localizing-set-G-gl})$ to weak equivalences, i.e.~the maps $\kappa_{H,A,B}\ppo_H i_+$ for finite groups $H$, finite $H$-sets $A,B$, and generating cofibrations $i\colon X\to Y$ of the projective model structure on $\cat{$\bm{(G\times H)}$-$\bm{\mathcal I}$-SSet}_*$.

For this, we will pick the generating cofibrations as in Theorem~\ref{thm:I-G-glob} (so that they are maps between cofibrant objects), and we will show that $f\mathrel{:=}\Delta^*(\kappa_{H,A,B}\ppo i_+)$ is a $(G\times H)$-global weak equivalence; the claim will then follow from Proposition~\ref{prop:free-quotient-spectra} again as $f$ is an injective cofibration and $H$ acts levelwise freely on projectively cofibrant $(G\times H)$-global spaces.

To show that $f$ is a $(G\times H)$-global weak equivalence, write $S,T$ for the source and target of $\kappa_{H,A,B}$ and consider the image
\begin{equation*}
\begin{tikzcd}
S\otimes X_+\arrow[dr,phantom, "\ulcorner"{very near end}]\arrow[d, "\kappa_{H,A,B}\otimes X_+"']\arrow[r, "S\otimes i_+"] & S\otimes Y_+\arrow[d]\arrow[ddr, bend left=15pt, "\kappa_{H,A,B}\otimes Y_+"]\\
T\otimes X_+\arrow[r]\arrow[drr, bend right=15pt, "T\otimes i_+"'] & P\arrow[dr, "f"{description}]\\[-1em]
&&[-1em] T\otimes Y_+
\end{tikzcd}
\end{equation*}
under $\Delta^*$ of the diagram defining $\kappa_{H,A,B}\ppo i_+$. Now $\kappa_{H,A,B}$ is an $H$-equivariant weak equivalence of $H$-equivariantly \emph{projectively} cofibrant $H$-spectra, whence an $H$-global weak equivalence by Proposition~\ref{prop:equivariant-vs-global} and thus a $(G\times H)$-global weak equivalence with respect to the trivial $G$-actions. Proposition~\ref{prop:tensor-homotopical} therefore shows that $\kappa_{H,A,B}\otimes X_+$ and $\kappa_{H,A,B}\otimes Y_+$ are $(G\times H)$-global weak equivalences. But on the other hand $\kappa_{H,A,B}\otimes X_+$ is also an injective cofibration by direct inspection, so the pushout $S\otimes Y_+\to P$ is again a $(G\times H)$-global weak equivalence. We conclude by $2$-out-of-$3$ that also $f$ is a $(G\times H)$-global weak equivalence, which then completes the proof that $\Delta^*(\kappa_{H,A,B}\ppo_H i_+)$ is a $G$-global weak equivalence and hence the proof of the proposition.
\end{proof}

\begin{prop}
Let $H$ be a finite group, let $A$ be a finite faithful $H$-set, and let $\phi\colon H\to G$ be any homomorphism. Then the unit
\begin{equation*}
\eta\colon\bm\Sigma(A,\blank)\smashp_\phi G_+ \to \Delta^*\Delta_!(\bm\Sigma(A,\blank)\smashp_\phi G_+)
\end{equation*}
is a $G$-global weak equivalence.
\begin{proof}
By design, $\eta$ is induced by the `diagonal' map $\bm\Sigma(A,\blank)\to \bm\Sigma(A,\blank)\otimes \mathcal I(A,\blank)_+$; in particular, it has a left inverse induced by the unique map $p\colon\mathcal I(A,\blank)\to *$. By $2$-out-of-$3$ it therefore suffices to show that $(\bm\Sigma(A,\blank)\otimes p_+)\smashp_\phi G_+=\phi_!(\bm\Sigma(A,\blank)\otimes p_+)$ is a $G$-global weak equivalence.

For this we note that $p$ is an $H$-global weak equivalence, whence so is $\bm\Sigma(A,\blank)\otimes p_+$ by Proposition~\ref{prop:tensor-homotopical}. By faithfulness, $H$ acts freely on $\bm\Sigma(A,\blank)$ outside the basepoint, so the claim follows again from Proposition~\ref{prop:free-quotient-spectra}.
\end{proof}
\end{prop}

\begin{proof}[Proof of Theorem~\ref{thm:delta-quillen-equivalence}]
As $\cat{R}\Delta^*$ is conservative (Lemma~\ref{lemma:Delta-star-conservative}), it is enough to show that the derived unit is an isomorphism in the homotopy category. As $\Delta^*$ is homotopical (Proposition~\ref{prop:Delta-star-homotopical}), this amounts to saying that the ordinary unit $X\to\Delta^*\Delta_!X$ is a $G$-global weak equivalence for every projectively cofibrant $G$-global spectrum $X$.

This is a standard cell induction argument: namely, $\Delta_!$ is left Quillen while $\Delta^*$ is a left adjoint sending cofibrations to \emph{injective} cofibrations, so that it is enough to prove the claim for the sources and targets of the standard generating cofibrations, see e.g.~\cite[Lemma~1.2.64]{g-global}. However, these are of the form $\bm\Sigma(A,\blank)\smashp_\phi G_+\smashp K_+$ for some simplicial set $K$; as the tensoring over $\cat{SSet}$ is homotopical and the unit is compatible with the tensoring ($\Delta_!\dashv\Delta^*$ being a simplicial adjunction), the claim therefore follows from the previous proposition.
\end{proof}

\section{Global brave new algebra}\label{section:brave-new-algebra}
We now turn our attention to multiplicative structures in stable $G$-global homotopy theory, generalizing results for $G=1$ by Schwede \cite[Chapter~5]{schwede-book} and Hausmann \cite[Section~3]{hausmann-global}.

\subsection{Positive model structures} Already in the non-equivariant setting, the study of commutative ring spectra from a model categorical perspective requires one to introduce suitable `positive' model structures. This subsection is devoted to the construction of positive flat and projective $G$-global model structures; however, as most of this is entirely parallel to the construction of the usual $G$-global model structures \cite[3.1]{g-global}, we will be somewhat terse here.

\begin{prop}
There is a unique model structure on $\cat{$\bm G$-Spectra}$ in which a map $f$ is a weak equivalence or fibration if and only if $f(A)$ is a $\mathcal G_{\Sigma_A,G}$-weak equivalence or fibration, respectively, for every \emph{non-empty} finite set $A$. We call this the \emph{positive $G$-global projective level model structure} and its weak equivalences the \emph{positive $G$-global level weak equivalences}. It is combinatorial with generating cofibrations
\begin{equation}\label{eq:gen-cof-pos-proj}
\{\bm\Sigma(A,\blank)\smashp_\phi G_+\smashp(\del\Delta^n\hookrightarrow\Delta^n)_+ : A\not=\varnothing, H\subset\Sigma_A,\phi\colon H\to G,n\ge0\},
\end{equation}
simplicial, proper, and filtered colimits in it are homotopical.
\begin{proof}
To see that the model structure exists and is cofibrantly generated (hence combinatorial) with the above generating cofibrations it suffices to check the `consistency condition' of \cite[Proposition~C.23]{schwede-book} for the $\mathcal G_{\Sigma_A,G}$-model structure on $\cat{$\bm{(G\times\Sigma_A)}$-SSet}_*$ for $A\not=\varnothing$ and the model structure on $\cat{$\bm{(G\times\Sigma_\varnothing)}$-SSet}_*$ in which only isomorphisms are cofibrations. However, it is clear that for $A\not=\varnothing$ each $\bm\Sigma(A,B)\smashp_{\Sigma_A}\blank$ sends the usual generating acyclic cofibrations of $\cat{$\bm{(G\times\Sigma_A)}$-SSet}_*$ to injective cofibrations and weak equivalences (also see \cite[proof of Proposition~3.1.20]{g-global}) and even to isomorphisms for $A=\varnothing$.

Right properness, the statement about filtered colimits, and the Pullback Power Axiom for simplicial model categories follow immediately as all relevant constructions and notions are levelwise. Finally, all (generating) cofibrations are injective cofibrations, so left properness follows in the same way.
\end{proof}
\end{prop}

Similarly one gets:

\begin{prop}
There is a unique model structure on $\cat{$\bm G$-Spectra}$ in which a map $f$ is a weak equivalence or fibration if and only if $f(A)$ is a weak equivalence or fibration, respectively, in the \emph{injective} $\mathcal G_{\Sigma_A,G}$-model structure for every $A\not=\varnothing$. We call this the \emph{positive flat $G$-global level model structure}; its weak equivalences are again the $G$-global positive level weak equivalences. Moreover, it is combinatorial with generating cofibrations
\begin{equation}\label{eq:gen-cof-pos-flat}
\{(\bm\Sigma(A,\blank)\smashp G_+)/H\smashp(\del\Delta^n\hookrightarrow\Delta^n)_+ : A\not=\varnothing, H\subset\Sigma_A\times G,n\ge0\},
\end{equation}
simplicial, proper, and filtered colimits in it are homotopical.\qed
\end{prop}

\begin{rk}
The above generating cofibrations agree with the generating cofibrations of Hausmann's \emph{$G$-equivariant positive flat model structure} from \cite[discussion after Proposition~2.28]{hausmann-equivariant}. In particular, the above cofibrations are independent of the group $G$ and agree with what Hausmann calls \emph{positive flat cofibrations}, cf.~Remark~2.20 of \emph{op. cit.}
\end{rk}

For later use we record the following relationship to the usual projective and flat cofibrations:

\begin{lemma}\label{lemma:positive-vs-absolute}
Let $f\colon X\to Y$ be a map in $\cat{$\bm G$-Spectra}$. Then:
\begin{enumerate}
\item $f$ is a positive flat cofibration if and only if it is a flat cofibration and $f(\varnothing)$ is an isomorphism.
\item $f$ is a positive $G$-global projective cofibration if and only if it is a $G$-global projective cofibration and $f(\varnothing)$ is an isomorphism.
\end{enumerate}
\begin{proof}
This is immediate from the characterization of cofibrations in terms of latching maps given in \cite[Proposition~C.23]{schwede-book}.
\end{proof}
\end{lemma}

\begin{defi}
A $G$-spectrum $X$ is called a \emph{positive $G$-global $\Omega$-spectrum} if for every finite group $H$, any \emph{non-empty} finite faithful $H$-set $A$, and every finite $H$-set $B$ the adjoint structure map
\begin{equation*}
X(A)\to\cat{R}\Omega^{B}X(A\amalg B)
\end{equation*}
is a $\mathcal G_{H,G}$-weak equivalence.
\end{defi}

As before, if $X$ is fibrant in either of the above positive level model structures, then this is represented by the ordinary adjoint structure map.

\begin{prop}\label{prop:pos-proj}
There is a unique model structure on $\cat{$\bm G$-Spectra}$ whose cofibrations are the positive $G$-global projective cofibrations and whose weak equivalences are the usual $G$-global weak equivalences. We call this the \emph{positive $G$-global projective model structure}. Its fibrant objects are precisely the positively projectively level fibrant positive $G$-global $\Omega$-spectra.

Moreover, it is again combinatorial with generating cofibrations $(\ref{eq:gen-cof-pos-proj})$, simplicial, proper, and filtered colimits in it are homotopical.
\begin{proof}
We will first construct a Bousfield localization with the above fibrant objects. For this we recall the maps $\lambda_{H,A,B}\colon S^B\smashp\bm\Sigma(A\amalg B,\blank)\to\bm\Sigma(A,\blank)$ from Remark~\ref{rk:equivariant-proj-gen-cof} for finite groups $H$ and finite $H$-sets $A,B$. Varying over all homomorphisms $\phi\colon H\to G$ from finite groups to $H$ and restricting $A$ to \emph{non-empty faithful} $H$-sets, the maps $\phi_!\lambda_{H,A,B}$ are then maps between cofibrant objects of the positive $G$-global projective level model structure corepresenting $\phi$-fixed points of the adjoint structure maps. Factoring each of them as a cofibration $\kappa_{\phi,A,B}$ followed by a $G$-global positive level weak equivalence, \cite[Proposition~A.3.7.3]{htt} applied to the set of all $\kappa_{\phi,A,B}$'s then gives a Bousfield localization with the above fibrant objects, and this is automatically again combinatorial, left proper, simplicial, and filtered colimits in it are homotopical (see \cite[Lemma~A.2.4]{g-global} for the final statement).

We now claim that the weak equivalences agree with the $G$-global weak equivalences. For this we first observe that the identity constitutes a Quillen adjunction $\cat{$\bm G$-Spectra}_\text{positive $G$-global proj.}\rightleftarrows\cat{$\bm G$-Spectra}_\text{$G$-global proj.}$ by Lemma~\ref{lemma:check-QA-fibrant} as the left adjoint preserves cofibrations and the right adjoint preserves fibrant objects. On the other hand, a simple cofinality argument shows that positive $G$-global \emph{level} weak equivalences are $\ul\pi_*$-isomorphisms, hence in particular $G$-global weak equivalences, so every weak equivalence in the above model structure is a $G$-global weak equivalence.

Conversely, let $f\colon X\to Y$ be a $G$-global weak equivalence; we want to show that it is a weak equivalence in the above model structure. Using the previous direction and functorial factorizations in the above model structure, we reduce by $2$-out-of-$3$ to the case that $X$ and $Y$ are fibrant in the above sense, i.e.~they are positive $G$-global $\Omega$-spectra and positively projectively level fibrant. Then the natural maps $X\to\Omega\sh X, Y\to\Omega\sh Y$ are positive $G$-global level weak equivalences, and $\Omega\sh f$ is a $G$-global weak equivalence of $G$-global $\Omega$-spectra, hence in particular a (positive) $G$-global level weak equivalence. Thus, another application of $2$-out-of-$3$ shows that also $f$ is a positive $G$-global level weak equivalence, hence in particular a weak equivalence in the above model structure as claimed.

Finally, we observe that despite its definition right properness is independent of the class of fibrations \cite[Proposition~2.5]{rezk-proper}, so right properness of the positive $G$-global projective model structure follows from right properness of the usual $G$-global projective model structure.
\end{proof}
\end{prop}

\begin{prop}\label{prop:pos-flat}
There is a unique model structure on $\cat{$\bm G$-Spectra}$ whose cofibrations are the positive flat cofibrations and whose weak equivalences are the usual $G$-global weak equivalences. We call this the \emph{positive flat $G$-global model structure}. Its fibrant objects are precisely those positive $G$-global $\Omega$-spectra that are fibrant in the positive flat $G$-global level model structure.

Moreover, this model structure is again combinatorial with generating cofibrations $(\ref{eq:gen-cof-pos-flat})$, simplicial, proper, and filtered colimits in it are homotopical.
\begin{proof}
Arguing precisely as before we get a model structure with the desired cofibrations and fibrant objects. By abstract nonsense about Bousfield localizations, a map $f$ is a weak equivalence in this model structure or the one from the previous proposition if and only if $[f,T]$ is bijective for every positive $G$-global $\Omega$-spectrum $T$, where $[\,{,}\,]$ denotes hom sets in the localization at the positive $G$-global level weak equivalences. In particular, its weak equivalences agree with the ones from the previous proposition, i.e.~with the $G$-global weak equivalences.

Finally, all the remaining properties are established precisely as in the previous proposition.
\end{proof}
\end{prop}

\begin{rk}
For $G=1$ the above two model structures again agree, and they recover Hausmann's \emph{positive global model structure} \cite[Theorem~2.18]{hausmann-global}.
\end{rk}

\begin{rk}\label{rk:gen-acyclic-pos-flat}
We will never need to know explicitly how the generating acyclic cofibrations in the above model structures look like. However, we record for later use that the generating cofibrations $(\ref{eq:gen-cof-pos-flat})$ are maps between cofibrant objects, so \cite[Corollary~2.7]{barwick-tractable} shows that we can also find a set of generating acyclic cofibrations for the positive $G$-global flat model structure consisting of maps between cofibrant objects.
\end{rk}

Next, we come to functoriality properties:

\begin{lemma}
Let $\alpha\colon H\to G$ be any homomorphism. Then we have Quillen adjunctions
\begin{align*}
\alpha_!\colon\cat{$\bm H$-Spectra}_\textup{pos.~$H$-global proj.}&\rightleftarrows\cat{$\bm G$-Spectra}_\textup{pos.~$G$-global proj.} :\!\alpha^*\\
\alpha^*\colon\cat{$\bm G$-Spectra}_\textup{pos.~$G$-global flat}&\rightleftarrows
\cat{$\bm H$-Spectra}_\textup{pos.~$H$-global flat} :\!\alpha_*.
\end{align*}
\begin{proof}
For the first statement, we first observe that this holds for the corresponding level model structures as a consequence of Lemma~\ref{lemma:graph-target}. For the actual model structures at hand it suffices then to observe that $\alpha^*$ sends fibrant objects to positive $H$-global $\Omega$-spectrum by direct inspection.

For the second statement, it is clear that $\alpha^*$ preserves positive flat cofibrations and sends $G$-global weak equivalences to $H$-global weak equivalences.
\end{proof}
\end{lemma}

\begin{lemma}
Let $\alpha\colon H\to G$ be \emph{injective}. Then we also have Quillen adjunctions
\begin{align*}
\alpha_!\colon\cat{$\bm H$-Spectra}_\textup{pos.~$H$-global flat}&\rightleftarrows\cat{$\bm G$-Spectra}_\textup{pos.~$G$-global flat} :\!\alpha^*\\
\alpha^*\colon\cat{$\bm G$-Spectra}_\textup{pos.~$G$-global proj.}&\rightleftarrows
\cat{$\bm H$-Spectra}_\textup{pos.~$H$-global proj.} :\!\alpha_*.
\end{align*}
\begin{proof}
For the first statement we observe that $\alpha_!$ sends generating cofibrations to generating cofibrations by direct inspection and that it is homotopical by Proposition~\ref{prop:sp-functoriality-injective}.

For the second statement, the corresponding statement for level model structures follows from the fact that $\alpha^*$ sends $\mathcal G_{\Sigma_A,G}$-cofibrations to $\mathcal G_{\Sigma_A,H}$-cofibrations by Lemma~\ref{lemma:graph-target}. The actual claim then follows as $\alpha^*$ is homotopical.
\end{proof}
\end{lemma}

Arguing as for the usual model structures we then conclude from the above:

\begin{cor}
The positive $G$-global projective and flat model structures for varying $G$ make $\cat{Spectra}$ into a global model category $\glo{GlobalSpectra}^\pos$.\qed
\end{cor}

On the other hand we have straight from the definition of cofibrations and weak equivalences:

\begin{cor}
The identity defines a global Quillen equivalence $\glo{GlobalSpectra}^\pos\rightleftarrows\glo{GlobalSpectra}$.\qed
\end{cor}

Finally let us record how the smash product behaves with respect to the above model structures:

\begin{cor}\label{cor:smash-left-Quillen}
The smash product $\cat{$\bm G$-Spectra}\times\cat{$\bm G$-Spectra}\to\cat{$\bm G$-Spectra}$ is a left Quillen bifunctor in each of the following cases:
\begin{enumerate}
\item the $G$-global positive flat model structure and the $G$-global flat model structure on the source, and the $G$-global positive flat model structure on the target
\item the $G$-global positive flat model structure and the $G$-global projective model structure on the source, and the $G$-global positive projective model structure on the target
\item the $G$-global positive projective model structure and the $G$-global flat model structure on the source, and the $G$-global positive projective model structure on the target.
\end{enumerate}
\begin{proof}
For the ordinary projective and flat model structures this is Theorem~\ref{thm:smash-g-global}. The claims now follow from this via Lemma~\ref{lemma:positive-vs-absolute} and the natural isomorphism $(X\smashp Y)(\varnothing)\cong X(\varnothing)\smashp Y(\varnothing)$ for all symmetric spectra $X,Y$.
\end{proof}
\end{cor}

\subsection{Smash powers and norms} Next, we come to homotopical properties of smash powers.

\begin{constr}
Let $X$ be a $G$-spectrum and let $n\ge1$. Then the $n$-fold smash power $X^{\smashp n}$ carries $n$ commuting $G$-actions as well as a $\Sigma_n$-action (by permuting the factors). Together these assemble into a natural action of the \emph{wreath product} $\Sigma_n\wr G=\Sigma_n\ltimes G^n$, lifting $(\blank)^{\smashp n}$ to a functor $\cat{$\bm G$-Spectra}\to\cat{$\bm{(\Sigma_n\wr G)}$-Spectra}$.
\end{constr}

\begin{constr}
Let $H\subset G$ be finite groups and set $n\mathrel{:=} |G/H|$. Then any choice of right $H$-coset representatives $g_1,\dots,g_n$ defines an injective homomorphism $\iota\colon G\to\Sigma_n\wr H$ as follows: $\iota(g)=(\sigma(g);h_1(g),\dots,h_n(g))$ with $gg_i=g_{\sigma(g)(i)}h_i(g)$. The composite
\begin{equation}\label{eq:hhr-norm}
N^G_H\colon \cat{$\bm H$-Spectra}\xrightarrow{(\blank)^{\smashp n}} \cat{$\bm{(\Sigma_n\wr H)}$-Spectra}\xrightarrow{\iota^*}\cat{$\bm G$-Spectra}
\end{equation}
is then called the (Hill-Hopkins-Ravenel) \emph{norm}.
\end{constr}

For our purposes, the key result on the equivariant behaviour of the norm will be the following:

\begin{thm}\label{thm:norm-equivariant}
The composite $(\ref{eq:hhr-norm})$ sends $H$-equivariant weak equivalences between flat $H$-spectra to $G$-equivariant weak equivalences (of flat $G$-spectra).
\begin{proof}
This is the special case $N=1$ of \cite[Theorem~6.8]{hausmann-equivariant}.
\end{proof}
\end{thm}

We now want to consider the smash powers and norms from a $G$-global perspective. Here we will prove the following stronger result:

\begin{thm}\label{thm:smash-power-wreath}
Let $f\colon X\to Y$ be a $G$-global weak equivalence and assume $X$ and $Y$ are flat. Then $f^{\smashp n}$ is a $(\Sigma_n\wr G)$-global weak equivalence.
\end{thm}

Restricting along the above homomorphism $\iota$ this immediately implies:

\begin{cor}
Let $H\subset G$ be finite groups. Then $N^G_H\colon\cat{$\bm H$-Spectra}\to\cat{$\bm G$-Spectra}$ sends $H$-global weak equivalences of flat $H$-spectra to $G$-global weak equivalences (of flat $G$-spectra).\qed
\end{cor}

\begin{proof}[Proof of Theorem~\ref{thm:smash-power-wreath}]
Let $f\colon X\to Y$ be a $G$-global weak equivalence of flat $G$-spectra. As a first step, we will show that $f^{\smashp n}$ is a $(\Sigma_n\wr G)$-\emph{equivariant} weak equivalence. For this, the key observation will be that while the norm is defined in terms of the smash power, in the global setting we can also go the other way round, cf.~\cite[Remark~5.1.7-(iv)]{schwede-book}. Namely, write $K\subset\Sigma_n\wr G$ for the subgroup of those $(\sigma;g_1,\dots,g_n)$ with $\sigma(1)=1$, which comes with a projection homomorphism $\pi\colon K\to G,\pi(\sigma;g_1,\dots,g_n)=g_1$. If we now fix for each $i=1,\dots,n$ a permutation $\sigma_i$ with $\sigma_i(1)=i$, then the $(\sigma_i;1,\dots,1)$ form a system of coset representatives of $(\Sigma_n\wr G)/K$, and one easily checks from the definitions that the resulting homomorphism $\iota\colon\Sigma_n\wr G\to \Sigma_n\wr K$ is of the form
\begin{equation*}
(\sigma;g_1,\dots,g_n)\mapsto (\sigma; (?;g_1,?,\dots), (?;g_2,?,\dots),\dots)
\end{equation*}
where `$?$' denotes an entry we don't care about.

Now $\pi^*f$ is a $K$-equivariant weak equivalence by assumption on $f$, so $N^{\Sigma_n\wr G}_K(\pi^*f)$ is a $(\Sigma_n\wr G)$-equivariant weak equivalence by Theorem~\ref{thm:norm-equivariant}. But by the above description of $\iota$, this agrees with $f^{\smashp n}$ as map of $(\Sigma_n\wr G)$-spectra, completing the proof of the claim.

Now let $\phi\colon H\to \Sigma_n\wr G$ be any map. We have to show that $\phi^*(f^{\smashp n})$ is an $H$-equivariant weak equivalence. For this we view $f$ as a map of $(G\times H)$-spectra via the trivial $H$-action. Applying the above with $G$ replaced by $G\times H$ then shows that $f^{\smashp n}$ is a $\Sigma_n\wr (G\times H)$-equivariant weak equivalence where all copies of $H$ act trivially. The claim now follows by restricting along the \emph{injective} homomorphism
\begin{equation*}
H\xrightarrow{(\phi,\id)} (\Sigma_n\wr G)\times H\xrightarrow\delta \Sigma_n\wr (G\times H)
\end{equation*}
with $\delta$ given by $\delta(\sigma;g_1,\dots,g_n;h)=(\sigma;(g_1,h),\dots,(g_n,h))$.
\end{proof}

Finally, we come to the key property of the \emph{positive} model structures that will allow us to establish model structures on commutative algebras below:

\begin{lemma}\label{lemma:smash-power-free}
Let $X$ be a \emph{positive} flat spectrum and let $n\ge1$. Then the $\Sigma_n$-action on $X^{\smashp n}$ is levelwise free outside the basepoint.
\begin{proof}
This is a special case of \cite[Proposition~7.7*-(a)]{harper-corrigendum}.
\end{proof}
\end{lemma}

\begin{cor}\label{cor:sym-powers-pos-flat}
Let $f\colon X\to Y$ be a $G$-global weak equivalence and assume $X$ and $Y$ are \emph{positively} flat. Then $f^{\smashp n}/\Sigma_n$ is a $G$-global weak equivalence again.
\begin{proof}
By Theorem~\ref{thm:smash-power-wreath}, $f^{\smashp n}$ is a $(\Sigma_n\wr G)$-global weak equivalence, hence in particular a $(\Sigma_n\times G)$-global weak equivalence. The claim follows as $\Sigma_n$ acts freely on both source and target by the previous lemma.
\end{proof}
\end{cor}

\subsection{Global model categories of modules}
Throughout, let $G$ be a finite group.

\begin{defi}
A \emph{$G$-global ring spectrum} is a monoid $R$ (for the smash product) in $\cat{$\bm G$-Spectra}$. We write $\cat{Mod}_R^G$ (or simply $\cat{Mod}_R$ if $G$ is clear from the context) for the category of (left) modules in $\cat{$\bm G$-Spectra}$ over the monoid $R$.
\end{defi}

\begin{prop}
Let $R$ be a $G$-global ring spectrum. Then the positive projective and positive flat $G$-global model structure on $\cat{$\bm G$-Spectra}$ transfer along the free-forgetful adjunction
\begin{equation*}
R\smashp\blank\colon\cat{$\bm G$-Spectra}\rightleftarrows\cat{Mod}_R :\!\mathbb U.
\end{equation*}
This model structure is proper, simplicial, and combinatorial with generating cofibrations $R\smashp I$ and generating acyclic cofibrations $R\smashp J$ for sets of generating (acyclic) cofibrations $I,J$ of the positive projective/flat $G$-global model structure. Moreover, filtered colimits in it are homotopical.
\begin{proof}
For the existence of the model structure, we verify the assumptions of the Crans-Kan Transfer Criterion. By local presentability, every set admits the small object argument, so it only remains to show that for some (hence any) set $J$ of generating acyclic cofibrations of $\cat{$\bm G$-Spectra}_\textup{$G$-global flat}$ the forgetful functor $\mathbb U$ sends relative $(R\smashp J)$-cell complexes to weak equivalences. However, $\mathbb U$ is also a left adjoint, so it suffices to show that it sends maps in $R\smashp J$ to acyclic cofibrations in the \emph{injective $G$-global model structure}. Taking $J$ to consist of maps between cofibrant objects (see~Remark~\ref{rk:gen-acyclic-pos-flat}), this is immediate from Proposition~\ref{prop:flatness-theorem}.

The model structure is clearly combinatorial with the above generating (acyclic) cofibrations, right proper, and simplicial, and filtered colimits in it are homotopical. To see that it is also left proper it suffices to observe that $\mathbb U$ is also left Quillen as a functor into the \emph{injective $G$-global model structure} on $\cat{$\bm G$-Spectra}$ by the above, so that the claim follows from left properness of $\cat{$\bm G$-Spectra}_\textup{$G$-global injective}$ via \cite[Lemma~A.2.15]{g-global}.
\end{proof}
\end{prop}

\begin{rk}
The corresponding statement for the usual projective and flat model structures hold as well (by the same argument); however, for us the above version will be more convenient as we later want to relate the above to categories of commutative algebras.
\end{rk}

\begin{constr}
Let $H$ be a finite group. Identifying $\cat{$\bm H$-Mod}_R^G$ with the category of modules in $\cat{$\bm{(H\times G)}$-Spectra}$ (with $H$ acting trivially on $R$), we get positive projective and flat model structures on $\cat{$\bm H$-Mod}_R^G$, which are equivalently transferred from $\cat{$\bm{(H\times G)}$-Spectra}$ along the evident forgetful functor.
\end{constr}

\begin{lemma}\label{lemma:module-global}
The above make $\cat{Mod}_R^G$ into a global model category $\glo{Mod}_R^G$.
\begin{proof}
We already observed that all of these model structures are combinatorial, simplicial, and (left) proper. To complete the proof that this defines a \emph{pre}global model category it is then enough to note that $\alpha^*$ commutes with both $\mathbb U$ and $R\smashp\blank$ for any $\alpha\colon H\to H'$, so that $\alpha_*$ commutes with $\mathbb U$ by passing to mates; thus, all the functoriality properties follow from the corresponding statements for $\glo{GlobalSpectra}^+$.

Similarly, the fact that $\glo{Mod}_R^G$ is a \emph{global} model category follows from the corresponding result for $G\text-\glo{GlobalSpectra}^+$ using the third formulation in Proposition~\ref{prop:global-tfae}.
\end{proof}
\end{lemma}

\begin{cor}\label{cor:modules-stable}
The global model category $\glo{Mod}_R^G$ is stable.
\begin{proof}
Let $H$ be a finite group and $A$ a finite $H$-set. Then $S^A\smashp\blank$ commutes with the forgetful functor $\mathbb U$, so it is homotopical. Thus, also the (total or equivalently left) derived functors commute. Similarly, $\Omega^A$ commutes with $\mathbb U$, and as both are right Quillen their right derived functors also commute (via the mate of the above equivalence). The claim now follows immediately from the corresponding statement for $G\text-\glo{GlobalSpectra}$ (Proposition~\ref{prop:GlobalSpectra-stable}) as $\Ho(\mathbb U)$ is conservative.
\end{proof}
\end{cor}

\begin{lemma}\label{lemma:module-adjunctions}
Let $G$ be a finite group and let $R$ be a \emph{flat} $G$-global ring spectrum. Then we have global Quillen adjunctions
\begin{equation*}
\ul{R\smashp\blank}\colon G\text-\glo{GlobalSpectra}^+\rightleftarrows\glo{Mod}^G_R :\!\ul{\mathbb U}
\qquad\text{and}\qquad
\ul{\mathbb U}\colon \glo{Mod}^G_R\rightleftarrows G\text-\glo{GlobalSpectra}^+ :\!\ul{F(R,\blank)}
\end{equation*}
Moreover, both $\ul{R\smashp\blank}$ and $\ul{\mathbb U}$ are homotopical.
\begin{proof}
By definition, $\mathbb U$ preserves weak equivalences as well as fibrations in either model structure; in particular, it is right Quillen. Moreover, it sends generating cofibrations to cofibrations by Proposition~\ref{cor:smash-left-Quillen}, so it is also left Quillen. Finally, also $R\smashp\blank$ is homotopical by the Flatness Theorem (Proposition~\ref{prop:flatness-theorem}).
\end{proof}
\end{lemma}

As $\cat{Mod}^G_R$ is a stable model category, its homotopy category is naturally a triangulated category. For later use, we record a t-structure on this, for which we first have to recall the \emph{true} homotopy groups (as opposed to the na\"ive ones considered before) of a $G$-global spectrum:

\begin{constr}
Let $X$ be a $G$-global spectrum, let $\phi\colon H\to G$ be a homomorphism from a finite group $H$, and let $k\in\mathbb Z$. We define the \emph{true $k$-th $\phi$-equivariant homotopy group} as
\begin{equation*}
\hat\pi_k^\phi(X)\mathrel{:=}[\Sigma^{\bullet+k}_+ I(H,\blank)\times_\phi G,X],
\end{equation*}
where $[\,{,}\,]$ denotes the hom set in the $G$-global stable homotopy category; note that $\hat\pi_k^\phi(X)$ carries a natural abelian group structure by additivity of the latter, and that distinguished triangles induce long exact sequences in $\hat\pi_*^\phi$ by general nonsense about triangulated categories.

We moreover write $\underline{\hat\pi}_k(X)$ for the collection of all $\phi$-equivariant homotopy groups for varying $\phi$, together with all natural maps between them. For $G=1$, this structure can be explicitly described as a \emph{global (Mackey) functor} \cite[Remark~6.5]{schwede-k-theory}; we will never need an explicit description for non-trivial $G$.
\end{constr}

\begin{rk}
The true $\phi$-equivariant homotopy groups of $X$ can be equivalently described as the true $H$-equivariant homotopy groups of the $H$-equivariant spectrum $\phi^*X$, see \cite[Corollary~3.3.4]{g-global}.
\end{rk}

\begin{prop}\label{prop:t-structure-mod}
Let $R$ be a $G$-global ring spectrum.
\begin{enumerate}
\item The triangulated category $\Ho(\cat{Mod}^G_R)$ is compactly generated with generators $R\smashp\Sigma^\bullet_+\big(I(A,\blank)\times_\phi G\big)$ for finite sets $A$, subgroups $H\subset\Sigma_A$, and homomorphisms $\phi\colon H\to G$.
\item Assume that $R$ is \emph{connective}, i.e.~for every $k<0$ the true homotopy groups $\underline{\hat\pi}_k R$ (of the underlying $G$-global spectrum) are trivial. Then $\Ho(\cat{Mod}^G_R)$ carries a t-structure with
\begin{itemize}
	\item connective part the connective $R$-modules
	\item coconnective part those $R$-modules $X$ that are \emph{coconnective} in the sense that $\underline{\hat\pi}_kX=0$ for all $k>0$.
\end{itemize}
Moreover, the connective part is the smallest subcategory containing the above generators that is closed under coproducts, extensions, and suspension.
\end{enumerate}
\begin{proof}
For $R=\mathbb S$, i.e.~the case of $G$-global spectra, these are part of \cite[Theorem~7.1.12]{global-param} and its proof.

We will now deduce the general case from this. For this we first note that we have an exact adjunction
\begin{equation}\label{eq:exact-adj}
R\smashp^{\cat L}\blank\colon\Ho(\cat{$\bm G$-Spectra}_\textup{$G$-global})\rightleftarrows \Ho(\cat{Mod}^G_R) :\!\mathbb U
\end{equation}
with conservative right adjoint by construction, so the $R\smashp\Sigma^\bullet_+I(A,\blank)\times_\phi G$ are a set of generators (here we secretly used Proposition~\ref{prop:flatness-theorem} to identify $R\smashp^{\cat L}\Sigma^\bullet_+I(A,\blank)\times_\phi G$ with the underived smash product). We now observe that coproducts of $G$-global spectra are homotopical \cite[Lemma~3.1.43]{g-global} and that $\mathbb U$ preserves them on the pointset level; it follows that also coproducts in $\cat{Mod}^G_R$ are homotopical, and that the right adjoint in $(\ref{eq:exact-adj})$ preserves coproducts, so that its left adjoint preserves compact objects. This completes the proof of the first statement.

For the second statement, we use the above together with \cite[Theorem~A.1]{t-struc-compact} to obtain a t-structure on $\Ho(\cat{Mod}^G_R)$ whose connective part is the smallest subcategory closed under coproducts, suspensions, and extensions containing the $R\smashp\Sigma_+^\bullet I(A,\blank)\times_\phi G$'s; we claim that this is the t-structure described above. To see this, we first observe that an $R$-module $X$ is coconnective in this t-structure if and only if $[\Sigma T,X]=0$ for all $T\in\Ho(\cat{Mod}^G_R)_{\ge 0}$. Specializing to $T=R\smashp\Sigma^{\bullet+k}_+ I(A,\blank)\times_\phi G$ and using the adjunction isomorphisms shows that $X$ is then coconnective in the above sense. Conversely, for fixed $X$ the class of objects $T$ such that $[\Sigma^kT,X]=0$ for all $k>0$ is easily seen to be closed under coproducts, suspension, and extensions (using the long exact sequence for the last statement), so if $X$ has trivial positive homotopy groups, then $X\in\Ho(\cat{Mod}^G_R)_{\le0}$.

To identify the connective part, we will first show:

\begin{claim*}
	Let $X,Y$ be connective $G$-global spectra. Then the derived smash product $X\smashp^{\cat L}Y$ is again connective.
	\begin{proof}
		The derived smash product comes from a left Quillen bifunctor, so it preserves coproducts in each variable and is exact in each variable. By the t-structure on $G$-global spectra recalled above, it therefore suffices to show that the (derived, or equivalently underived) smash product of any two of the above generators of the $G$-global model structure is again connective. However, a simple Yoneda argument shows that $\Sigma^\bullet_+I(A,\blank)\smashp\Sigma^\bullet_+ I(B,\blank)\cong\Sigma^\bullet_+ I(A\amalg B,\blank)$ naturally in $A$ and $B$, so
		\begin{equation*}
		\big(\Sigma^\bullet_+ I(A,\blank)\times_\phi G\big)\smashp\big(\Sigma^\bullet_+ I(B,\blank)\times_\psi G\big)\cong \Sigma^\bullet_+\big(I(A\amalg B)\times (G\times G)\big)/(H\times K)
		\end{equation*}
		(where $H\subset\Sigma_A,K\subset\Sigma_B$ are the sources of $\phi$ and $\psi$, respectively), which is clearly connective again.
	\end{proof}
\end{claim*}

As $R$ is connective, the claim shows that the exact coproduct preserving functor $\mathbb U\colon\Ho(\cat{Mod}^G_R)\to\Ho(\cat{$\bm G$-Spectra})$ sends the chosen generators to connective spectra; thus, $\Ho(\cat{Mod}^G_R)_{\ge0}$ is contained in the preimage under $\mathbb U$ of $\Ho(\cat{$\bm G$-Spectra})_{\ge0}$, i.e.~all objects of $\Ho(\cat{Mod}^G_R)_{\ge0}$ are connective in the above sense. Conversely, if $X$ is connective then we have a distinguished triangle
\begin{equation*}
	X_{\ge 0}\to X\to X_{\le -1}\to \Sigma X_{\ge 0}
\end{equation*}
with $X_{\ge0}\in\Ho(\cat{Mod}^G_R)_{\ge0}$ and $X_{\le -1}\in\Ho(\cat{Mod}^G_R)_{\le-1}$ by the axioms of a t-structure. Then $X_{\ge0}$ is connective by the above and hence so is $X_{\le -1}$ by the long exact sequence. But on the other hand, $X_{\le -1}$ has vanishing non-negative homotopy groups by the above identification of the coconnective part, so $X_{\le -1}=0$ and hence $X\cong X_{\ge0}\in\Ho(\cat{Mod}^G_R)_{\ge0}$ as claimed.
\end{proof}
\end{prop}

As a consequence of Corollary~\ref{cor:modules-stable}, the suspension spectrum-loop space adjunction defines a global Quillen equivalence between $\glo{Mod}_R^G$ and its stabilization $\ul\Sp(\glo{Mod}_R^G)$, so the latter contains no new homotopy theoretic information. Nevertheless, the concrete model will be useful later for comparisons to other stabilizations, and we close this discussion by understanding its weak equivalences a bit better.

\begin{lemma}\label{lemma:g-global-sp-sp-colim}
Let $R$ be a $G$-global ring spectrum. Then the $G$-global weak equivalences in $\Sp(\glo{Mod}^G_R)$ are closed under pushouts along injective cofibrations (i.e.~levelwise injections) as well as under arbitrary filtered colimits.
\begin{proof}
For the first statement, let $i\colon X\to Y$ be an injective cofibration, and let $f\colon X\to Z$ be a $G$-global weak equivalence. We factor $f$ (say, in the projective model structure) as an acyclic cofibration $k$ followed by an acyclic fibration $p$; in particular, $p$ is a $G$-global \emph{level} weak equivalence. Then we have an iterated pushout square
\begin{equation*}
\begin{tikzcd}
X \arrow[d, "i"']\arrow[dr, phantom, "\ulcorner"{very near end}]\arrow[rr, "f", yshift=5pt, bend left=15pt] \arrow[r, "k"', tail, "\sim"] & H\arrow[dr, phantom, "\ulcorner"{very near end}] \arrow[r, "p"', "\sim", two heads]\arrow[d, "j"] & Z\arrow[d]\\
Y\arrow[r, "\ell"'] & K\arrow[r, "q"'] & P
\end{tikzcd}
\end{equation*}
in which $\ell$ is an acyclic cofibration (as a pushout of an acyclic cofibration) and $j$ is again an injective cofibration. Thus, applying left properness of the injective model structure levelwise we see that $q$ is a $G$-global (level) weak equivalence; the claim follows as $q\ell$ is a pushout of $f=pk$ along $i$.

The second statement follows similarly from the corresponding statement for $\cat{Mod}_R^G$, also see \cite[Lemma~A.2.4]{g-global}.
\end{proof}
\end{lemma}

\begin{lemma}\label{lemma:module-transferred}
Let $R$ be a flat $G$-global ring spectrum. Then both adjoints in the Quillen adjunction
\begin{equation*}
\Sp(\ul{R\smashp\blank})\colon\Sp(G\text-\glo{GlobalSpectra}^+)\rightleftarrows\Sp(\glo{Mod}^G_R):\!\Sp(\ul{\mathbb U})
\end{equation*}
are homotopical. Moreover, the projective model structure on $\Sp(\glo{Mod}^R_G)$ is transferred from the projective model structure on $\Sp(G\text-\glo{GlobalSpectra}^+)$ along the above adjunction, and likewise for the flat model structures.
\end{lemma}

Beware that in general stabilization does \emph{not} commute with transferring model structures!

\begin{proof}
By Lemma~\ref{lemma:module-adjunctions} \emph{both} functors are left Quillen (for either model structure), in particular they send flat acyclic cofibrations to weak equivalences. Moreover, both $\Sp(\ul{\mathbb U})$ and $\Sp(\ul{R\smashp\blank})$ preserve level weak equivalences as $\ul{\mathbb U}$ and $\ul{R\smashp\blank}$ are homotopical by the aforementioned lemma. Factoring an arbitrary weak equivalence into a flat cofibration followed by a level weak equivalence we see that both functors are in fact homotopical.

Now let $f\colon X\to Y$ be a map in $\Sp(\glo{Mod}^G_R)$ such that $\Sp(\ul{\mathbb U})(f)$ is a weak equivalence; we have to show that $f$ is a weak equivalence. As $\Sp(\ul{\mathbb U})$ is homotopical, we may assume by $2$-out-of-$3$ that $f$ is a map of fibrant objects. But then also $\Sp(\ul{\mathbb U})(f)$ is a map of fibrant objects, whence a level weak equivalence, so that also $f$ is a (level) weak equivalence.

Finally, to prove that the projective and flat model structures on $\Sp(\glo{Mod}^G_R)$ are transferred along the above adjunction, we will first show that the transferred model structures exist, and then prove that they agree with the given ones.

For the existence, it suffices by local presentability that $\Sp({R\smashp\blank})(J)$-cell complexes are sent by $\Sp(\ul{\mathbb U})$ to weak equivalences, where $J$ is our favourite set of generating acyclic cofibrations. But as $\Sp(\ul{\mathbb U})$ is left Quillen we only have to show that $\Sp(\ul{\mathbb U})$ sends maps in $\Sp(\ul{R\smashp\blank})(J)$ to acyclic cofibrations, which follows at once since also $\Sp(\ul{R\smashp\blank})$ is left Quillen.

Now we simply observe that the transferred model structure has the correct acyclic fibrations (obviously) as well as weak equivalences (by the above), so it agrees with the given model structure as claimed.
\end{proof}

\begin{rk}
The second half of the lemma holds true more generally for all (not necessarily flat) $G$-global ring spectra $R$; however, we will only need the above version, which is slightly easier to prove.
\end{rk}

\subsection{Global model categories of algebras}\label{subsection:model-structures-algebras} Now we come to the key objects of study for the rest of this paper.

\begin{defi}
	An \emph{ultra-commutative $G$-global ring spectrum} is a commutative monoid $R$ (for the smash product) in $\cat{$\bm G$-Spectra}$. We write:
	\begin{enumerate}
	\item $\cat{UCom}^G=\cat{$\bm G$-UCom}$ for the category of commutative monoids
	\item $\cat{UCom}^G_R$ for the category of $R$-algebras (i.e.~commutative monoids in $\cat{Mod}_R^G$ for the relative smash product, which is canonically isomorphic to the slice $R/\cat{UCom}^G$).
	\end{enumerate}
\end{defi}

\begin{prop}\label{prop:t-structure-smash}
	Let $R$ be a connective $G$-global ultra-commutative ring spectrum and equip $\Ho(\cat{Mod}^G_R)$ with the t-structure of Proposition~\ref{prop:t-structure-mod}. Then the derived relative smash product restricts to functors
	\begin{equation*}
		\blank\smashp_R^{\cat L}\blank\colon\Ho(\cat{Mod}^G_R)_{\ge m}\times\Ho(\cat{Mod}^G_R)_{\ge n}\to\Ho(\cat{Mod}^G_R)_{\ge m+n}
	\end{equation*}
	for all $m,n\in\mathbb Z$.
	\begin{proof}
		The relative smash product comes from a left Quillen bifunctor, so it is exact in each variable and moreover preserves coproducts in each variable. We therefore reduce first to the case $m=n=0$, and then (by the characterization of the connective part of the t-structure) to proving that the $G$-global spectrum $(R\smashp\Sigma^\bullet_+ I(A,\blank)\times_\phi G)\smashp_R(R\smashp\Sigma^\bullet_+ I(B,\blank)\times_\phi G)$ is connective. But this is isomorphic to $R\smashp(\Sigma^\bullet_+ I(A,\blank)\times_\phi G)\smashp (\Sigma^\bullet_+ I(A,\blank)\times_\psi G)$, which is connective by the claim in the proof of Proposition~\ref{prop:t-structure-mod} (and flatness of the two suspension spectra).
	\end{proof}
\end{prop}

Next, we turn to model structures on $\cat{UCom}^G$ and more generally on categories of $R$-algebras for $G$-global ultra-commutative ring spectra $R$. These model structures can be constructed using the monoid axiom of \cite{schwede-shipley-monoidal} and the commutative monoid axiom of \cite{white-cmon}. We recall the relevant results and then show that these axioms are satisfied for the positive projective and flat $G$-global model structures.

\begin{constr}\label{constr:iterated-pushout-product}
	Let $(\mathscr C, \otimes, \unit)$ be a cocomplete closed symmetric monoidal category. For any object $X$ in $\mathscr C$, we denote by $\power{m} X = X^{\otimes m}/\Sigma_m$ the $m$-th symmetric power of $X$. We consider a generalization of the pushout product
	\[ f\ppo g\colon A\otimes Y \iP_{A\otimes B} X\otimes B\to X\otimes Y \]
	of two morphisms $f\colon A\to X$ and $g\colon B\to Y$:

	Let $f\colon A\to X$ be a morphism in $\mathscr C$, and $m\geq 1$. We consider the poset category $\mathscr P(\{1, \ldots, m\})$ and the hypercube-shaped diagram
	\begin{align*}
	W(f)\colon \mathscr P(\{1, \ldots, m\}) &\to \mathscr C, \\
	S & \mapsto W_1(f,S)\otimes \ldots \otimes W_m(f,S),
	\end{align*}
	where
	\[ W_k(f,S) = \begin{cases}
	A & \textup{if } k\notin S,\\
	X & \textup{if } k \in S.
	\end{cases}\]
	To an inclusion $i\colon S\hookrightarrow T$, the functor $W(f)$ assigns the morphism $W(f, i) = W_1(f, i)\otimes \ldots \otimes W_m(f,i)$, where
	\[ W_k(f,i) = \begin{cases}
	\id & \textup{if } k\notin T\setminus S,\\
	f & \textup{if } k \in T\setminus S.
	\end{cases}\]
	This diagram defines a morphism
	\[ f^{\ppo m}\colon Q^m := \colim_{\mathscr P(\{1, \ldots, m\})\setminus \{1, \ldots, m\}} W(f) \to X^{\otimes m} = W(f, \{1, \ldots, m\}).\]
	This morphism is $\Sigma_m$-equivariant for the actions induced by the permutation action of $\Sigma_m$ on $\{1, \ldots, m\}$, and hence induces a map
	\begin{equation}\label{eq:symmetric-iterated-pushout-product}
	f^{\ppo m}/\Sigma_m \colon Q^m/\Sigma_m \to \power{m} X = X^{\otimes m}/\Sigma_m.
	\end{equation}
	on coinvariants.
\end{constr}

\begin{defi}\label{def:monoid-axioms}
	Let $(\mathscr C, \otimes, \unit)$ be a symmetric monoidal model category. We say that $\mathscr C$ satisfies the
	\begin{enumerate}
		\item \emph{monoid axiom} if every $\big(\{\textup{acyclic cofibrations}\}\otimes \mathscr C\big)$-cellular map is a weak equivalence.
		\item \emph{commutative monoid axiom} if for any acyclic cofibration $f\colon X\to Y$ and any $n\geq 0$, the map $f^{\ppo n}/\Sigma_n$ is an acyclic cofibration.
		\item \emph{strong commutative monoid axiom} if for any cofibration or acyclic cofibration $f\colon X\to Y$ and any $n\geq 0$, the map $f^{\ppo n}/\Sigma_n$ is a cofibration or an acyclic cofibration, respectively.
	\end{enumerate}
\end{defi}

\begin{prop}\label{prop:model-str-on-algebras}
	Let $(\mathscr C, \otimes, \unit)$ be a combinatorial symmetric monoidal model category satisfying the monoid axiom and commutative monoid axiom, and let $R$ be a commutative monoid in $\mathscr C$. Then:
	\begin{enumerate}
		\item the category $\Mod_R$ of $R$-modules inherits a model structure from $\mathscr C$, transferred along the adjunction
		\[ \begin{tikzcd}[column sep = large]
		\mathscr C \arrow[r, shift left, "R\otimes \_"] & \Mod_R \arrow[l, shift left, "\mathbb U"].
		\end{tikzcd}\]
		\item the category $\CAlg_R$ of commutative $R$-algebras inherits a model structure from $\mathscr C$, transferred along the free-forgetful adjunction
		\[ \begin{tikzcd}[column sep = large]
		\mathscr C \arrow[r, shift left, "{R\otimes\power{} \mathrel{:=}R\otimes\coprod_{n\ge0}\mathbb P^n}"] &[4.2em] \CAlg_R \arrow[l, shift left, "\mathbb U"].
		\end{tikzcd}\]
		\item the category $\NUCA_R$ of non-unital commutative $R$-algebras (i.e.~$R$-modules $M$ equipped with an associative and commutative multiplication $M\smashp_RM\to M$) inherits a model structure from $\mathscr C$, transferred along the free-forgetful adjunction
		\[ \begin{tikzcd}[column sep = large]
		\mathscr C \arrow[r, shift left, "{R\otimes\power{>0}\mathrel{:=}R\otimes\coprod_{n>0}\mathbb P^n}"]&[4.2em] \NUCA_R \arrow[l, shift left, "\mathbb U"].
		\end{tikzcd}\]
	\end{enumerate}
\begin{proof}
The first assertion is part of \cite[Theorem 4.1]{schwede-shipley-monoidal}, the second of \cite[Theorem 3.2, Remark 3.3]{white-cmon} and the last of \cite[Theorem B.2.6]{stahlhauer}.
\end{proof}
\end{prop}

\begin{rk}\label{rk:pushout-analysis-nucas}
	We comment on the proof of the result for non-unital commutative algebras in \cite{stahlhauer}. As for the other results, the main step is an analysis of certain pushouts in the category of non-unital commutative algebras. For simplicity, we restrict to $R=\unit$. Explicitly, for a symmetric monoidal model category $(\mathscr C, \otimes, \unit)$, we consider morphisms $h\colon K\to L$ and $p\colon K\to X$ in $\mathscr C$, where $X$ is a non-unital commutative algebra, and we analyze the pushout
	\begin{equation}\label{diag:pushout-free-nucas}
	\begin{tikzcd}
	\power{>0} K \arrow[r, "\power{>0}h"] \arrow[d, "\tilde p", swap] & \power{>0} L \arrow[d]\\
	X \arrow[r, "f"] & P
	\end{tikzcd}
	\end{equation}
	in the category $\NUCA$ of non-unital commutative algebras. For the proof of \Cref{prop:model-str-on-algebras}, we need to check that if $h$ is an acyclic cofibration, then the pushout morphism $f\colon X\to P$ is a weak equivalence. We show this by defining a filtration
	\[ X = P_0\xrightarrow{f_1} P_1\xrightarrow{f_2} \ldots \to \colim_{n\geq 0} P_n = P. \]
	Here, the $P_n$ are iteratively defined via pushouts
	\[\hskip-3.76pt\hfuzz=4pt\begin{tikzcd}[cramped,column sep = 12em]
	(Q^{n+1}/\Sigma_{n+1}) \iP X\otimes (Q^{n+1}/\Sigma_{n+1}) \arrow[r, "h^{\ppo n+1}/\Sigma_{n+1} \iP X\otimes h^{\ppo n+1}/\Sigma_{n+1}"] \arrow[d, "t_{n+1}", swap]
	& \power{n+1} L \iP X\otimes \power{n+1} L \arrow[d, "T_{n+1}"]\\
	P_n \arrow[r, "f_{n+1}", swap]
	& P_{n+1}.
	\end{tikzcd} \]
	We will never need to know how the vertical maps are defined precisely; informally, $t_n$ is given by mapping all entries in $K$ to $X$ via $p$ and multiplying all resulting terms from $X$ together, whereas all entries in $L$ are collected into $\power{0<\ast\leq n} L$.
\end{rk}

We will also need the following criterion for left properness:

\begin{prop}\label{prop:alg-left-proper}
Let $\mathscr C$ be a combinatorial symmetric monoidal model category satisfying the monoid axiom and strong commutative monoid axiom. Then the transferred model structure on $\CAlg$ is left proper provided all of the following conditions are satisfied:
\begin{enumerate}
	\item $\mathscr C$ is left proper and filtered colimits in it are homotopical.
	\item There exists a set of generating cofibrations for $\mathscr C$ consisting of maps between \emph{cofibrant} objects.
	\item For any $X\in\mathscr C$ and any cofibration $i$, pushouts along $X\otimes i$ are homotopy pushouts.
	\item For any cofibrant $X$, the functor $X\otimes\blank$ is homotopical.
\end{enumerate}
\begin{proof}
This is a special case of \cite[Theorem~4.17]{white-cmon}, also see \cite[Theorem~2.1.34]{g-global} for the reduction to White's result.
\end{proof}
\end{prop}

Let us now specialize this to algebraic structures on $G$-global spectra:

\begin{prop}\label{prop:white-assumptions}
Let $G$ be a finite group. Then the positive projective and positive flat $G$-global model structures on $\cat{$\bm G$-Spectra}$ are symmetric monoidal. Moreover, they satisfy both the monoid axiom as well as the strong commutative monoid axiom.
\begin{proof}
The pushout product axiom is immediate from Corollary~\ref{cor:smash-left-Quillen}. Moreover, if $p\colon\mathbb S^+\to\mathbb S$ is a cofibrant replacement in either of these model structures, and $X$ is any $G$-spectrum, then $X\smashp p$ is a $G$-global weak equivalence by flatness of $\mathbb S$ and Proposition~\ref{prop:flatness-theorem}. Thus, both model structures are symmetric monoidal.

For the monoid axiom, we simply observe once more that for any $G$-global spectrum $R$ and any acyclic cofibration (say, in the positive flat model structure) $j$ the map $R\smashp j$ is a $G$-global weak equivalence and an injective cofibration.

If now $f$ is a positive flat cofibration then $f^{\ppo n}$ is a positive flat cofibration simply by monoidality, also see \cite[Remark~6.9]{hausmann-equivariant}; the strong commutative monoid axiom for positive flat cofibrations follows as $(\blank)/\Sigma_n$ preserves positive flat cofibrations (since $\triv_{\Sigma_n}$ clearly preserves acyclic fibrations).

In the projective case we explicitly compute that for a generating cofibration $f=\bm\Sigma(A,\blank)\smashp_{H} G_+\smashp (\del\Delta^m\hookrightarrow\Delta^m)$ with $H$ acting faithfully on $A\not=\varnothing$ the map $f^{\ppo n}/\Sigma_n$ agrees up to isomorphism with
\begin{equation}\label{eq:com-mon-proj-cof}
\bm\Sigma(\bm n\times A,\blank)\smashp_{\Sigma_n\wr H} G^n_+\smashp(\del\Delta^m\hookrightarrow\Delta^m)^{\ppo n}_+
\end{equation}
where $\Sigma_n\wr H$ acts on $\bm n\times A$ via $(\sigma,h_\bullet).(i,a)\mapsto (\sigma(i), h_i.a)$, which is faithful as $A\not=\varnothing$ and $H$ acts faithfully. Now $(\del\Delta^m\hookrightarrow\Delta^m)^{\ppo n}_+$ is a cofibration of simplicial sets and $G$ acts freely on $G^n$, so $G^n_+\smashp(\del\Delta^m\hookrightarrow\Delta^m)^{\ppo n}_+$ is a $\mathcal G_{\Sigma_n\wr H,G}$-cofibration, whence $(\ref{eq:com-mon-proj-cof})$ is a projective cofibration again.

Finally, for the (strong) commutative monoid axiom for acyclic cofibrations, we note that we can as before pick sets of generating acyclic cofibrations consisting of maps between positively flat $G$-spectra. It therefore suffices by \cite[Corollaries~10 and~23]{sym-powers} that for any $G$-global weak equivalence $f$ of positively flat $G$-spectra also $f^{\smashp n}/\Sigma_n$ is a weak equivalence, which is precisely the content of Corollary~\ref{cor:sym-powers-pos-flat}.
\end{proof}
\end{prop}

\begin{cor}\label{cor:ucomRG}
Let $G$ be a finite group and let $R$ be a $G$-global ultra-commutative ring spectrum. Then the positive projective and flat $G$-global model structures on $\cat{Mod}_R^G$ transfer along the free-forgetful adjunction $\mathbb{P}\colon\cat{Mod}_R^G\rightleftarrows\cat{UCom}_R^G :\!\mathbb U$. The resulting model structures are again combinatorial, proper, simplicial, and filtered colimits in it are homotopical.
\begin{proof}
Identifying $\cat{UCom}_R^G$ with the slice $R/\cat{UCom}^G$, it suffices to consider the case $R=\mathbb S$. Thus, the existence of the model structure follows from White's criterion (Proposition~\ref{prop:model-str-on-algebras}) and the previous proposition.

As the model structure is transferred from $\cat{$\bm G$-Spectra}$, we immediately see that it is combinatorial, right proper, simplicial, and that filtered colimits in it are homotopical. It remains to check left properness, for which we will verify the assumptions of Proposition~\ref{prop:alg-left-proper}:

We know from Propositions~\ref{prop:pos-proj} and~\ref{prop:pos-flat} that both positive model structures on $\cat{$\bm G$-Spectra}$ are left proper, that filtered colimits in them are homotpical, and moreover the standard generating cofibrations are maps of cofibrant objects. As already observed above, for any $X$ and any (generating) cofibration $i$, $X\smashp i$ is an injective cofibration, so pushouts along it are homotopy pushouts. Finally, if $X$ is cofibrant in either of the two model structures, then $X\smashp\blank$ is homotopical by Proposition~\ref{prop:flatness-theorem}.
\end{proof}
\end{cor}

Again, we more generally get positive projective and flat model structures on $\cat{$\bm H$-UCom}^G_R$ for all finite groups $H$.

\begin{lemma}
These make $\cat{UCom}^G_R$ into a global model category $\glo{Comm}_R^G$, and the free-forgetful adjunction defines a global Quillen adjunction $\ul{\mathbb P}\colon\glo{Mod}^G_R\rightleftarrows\glo{Comm}_R^G :\!\ul{\mathbb U}$.
\begin{proof}
For the first statement one argues precisely as for $\glo{Mod}_R^G$ (Lemma~\ref{lemma:module-global}). The second statement is clear.
\end{proof}
\end{lemma}

\begin{defi}
Let $R$ be a $G$-global ultra-commutative ring spectrum, considered as a commutative algebra over itself in the obvious way. An \emph{augmented commutative $R$-algebra} is an object of the slice $\cat{UCom}^G_R/R$, i.e.~a commutative $R$-algebra $A$ together with an $R$-algebra homomorphism $\epsilon\colon A\to R$ (note that $\epsilon$ is automatically a retraction of the unit $R\to A$ by unitality).
\end{defi}

\begin{rk}
Arguing as in Remark~\ref{rk:basepoint}, $\cat{UCom}^G_R/R$ carries a left proper simplicial combinatorial model structure in which a map is a weak equivalence, fibration, or cofibration, if and only if it is so in the positive projective model structure on $\cat{UCom}^G_R$. Analogously, we get a positive flat model structure and via the usual identifications these make $\cat{UCom}^G_R/R$ into a global model category $\glo{Comm}_R^G/R$.
\end{rk}

\begin{defi}\label{defi:nuca}
	Let $R$ be a $G$-global ultra-commutative ring spectrum. We denote the category of non-unital commutative $R$-algebras by $\cat{NUCA}^G_R$.
\end{defi}

\begin{cor}
	Let $G$ be a finite group and let $R$ be a $G$-global ultra-commutative ring spectrum. Then the positive projective and flat $G$-global model structures on $\cat{Mod}_R^G$ transfer along the free-forgetful adjunction $\mathbb{P}^{>0}\colon\cat{Mod}_R^G\rightleftarrows\cat{NUCA}_R^G :\!\mathbb U$. The resulting model structures are again combinatorial, proper, simplicial, and filtered colimits in them are homotopical.
	\begin{proof}
		All statements except for left properness follow as in the case of ordinary commutative algebras (Corollary~\ref{cor:ucomRG}). For left properness, on the other hand, we observe that the (non-full) inclusion $\cat{UCom}^G_R\hookrightarrow\cat{NUCA}^G_R$ admits a left adjoint $K$ given by $X\mapsto (R\hookrightarrow R\vee X)$ with the unique multiplication extending the one on $X$. Then $K$ preserves (generating) cofibrations by direct inspection, and it is homotopical by \cite[Lemma~3.1.43]{g-global}. Finally, it also reflects weak equivalences as on the level of underlying $G$-spectra $X$ is naturally a retract of $R\vee X=K(X)$. Thus, left properness of $\cat{NUCA}^G_R$ follows from left properness of $\cat{UCom}^G_R$ by \cite[Lemma~A.2.15]{g-global}.
	\end{proof}
\end{cor}

Again, we get more generally positive projective and flat model structures on $\cat{$\bm H$-NUCA}^G_R$ for all finite groups $H$.

\begin{lemma}
	These make $\cat{NUCA}^G_R$ into a global model category $\glo{NUCA}_R^G$, and the free-forgetful adjunction defines a global Quillen adjunction $\ul{\mathbb P^{>0}}\colon\glo{Mod}^G_R\rightleftarrows\glo{NUCA}_R^G :\!\ul{\mathbb U}$.
	\begin{proof}
		For the first statement one argues precisely as for $\glo{Mod}_R^G$, while the second statement is clear from the definitions.
	\end{proof}
\end{lemma}

\section{Global topological André-Quillen cohomology\texorpdfstring{\except{toc}{\protect\footnote{The contents of the first two subsections of this section are an adapted version of part of the thesis of the second author \cite[Chapter 2]{stahlhauer}.}}}{}}\label{sec:gtaq}
In this section, we introduce $G$-global versions of topological André-Quillen cohomology. The classical non-equivariant cohomology theory was introduced by Basterra in \cite{basterra-TAQ} as a cohomology theory for (augmented) commutative $S$-algebras. The algebraic predecessor homology theory, as defined by André and Quillen \cite{andre-homology, quillen-cohomology-algebras}, is defined as a derived functor of Kähler differentials which in turn may be constructed as the module of indecomposables of the augmentation ideal of an augmented algebra. Both Basterra's construction of topological André-Quillen (co)homology as well as our $G$-global version below mimic this construction.

\subsection{The cotangent complex in general model categories}\label{subsec:abstract-cotangent}
Our construction works in a general abstract setting: throughout, we fix a combinatorial and stable symmetric monoidal model category $\mathscr C$ satisfying the monoid axiom and commutative monoid axiom, such that finite coproducts in $\mathscr C$ are homotopical; the examples the reader should keep in mind and to which we will later specialize are the positive $G$-global model structures on $\cat{$\bm G$-Spectra}$ constructed in \Cref{section:brave-new-algebra}.

\begin{defi}\label{def:augmentation-ideal}
	Let $R$ be a commutative monoid in $\mathscr C$, and let $S$ be an augmented $R$-algebra, with augmentation $\epsilon\colon S\to R$. Then we define the \emph{augmentation ideal} $I(S)$ of $S$ as the strict pullback
	\[ \begin{tikzcd}
	I(S) \arrow[r]\arrow[dr, phantom, "\lrcorner"{very near start}] \arrow[d] & S \arrow[d, "\epsilon"] \\
	\ast \arrow[r] & R
	\end{tikzcd}\]
	in the category $\Mod_R$ of $R$-modules.
\end{defi}
This inherits the structure of a non-unital $R$-algebra from $S$, yielding a functor $I\colon \CAlg_R/R \to \NUCA_R$. There also is a \emph{unitalization} functor in the other direction, given as
\[ K\colon \NUCA_R \to \CAlg_R/R, \quad J\mapsto R\iP J. \]
This object is equipped with a multiplication using the multiplications on $R$ and $J$ as well as the $R$-module structure on $J$.

Additionally, we define a functor taking a non-unital commutative algebra to its module of indecomposables.

\begin{defi}\label{def:indecomposables}
	Let $R$ be a commutative monoid in $\mathscr C$, and $J$ be a non-unital commutative $R$-algebra with multiplication map $\mu\colon J\otimes_R J\to J$. Then we define the \emph{module of indecomposables} $Q(J)$ of $J$ as the strict pushout
	\[ \begin{tikzcd}
	J\otimes_R J \arrow[r, "\mu"]\arrow[dr, phantom, "\ulcorner"{very near end}] \arrow[d] & J \arrow[d]\\
	\ast \arrow[r] & Q(J)
	\end{tikzcd}\]
	in $\Mod_R$.
\end{defi}
This defines a functor $Q\colon \NUCA_R \to \Mod_R$. In the other direction, we can equip an $R$-module with the zero multiplication to obtain a functor
\[ Z\colon \Mod_R \to \NUCA_R.\]

\begin{prop}\label{prop:aug-ideal-Quillen-equivalence}
	The functors
	\[ \begin{tikzcd}
	\CAlg_R/R \arrow[r, "I", swap, shift right] & \NUCA_R \arrow[l, "K", swap, shift right] \arrow[r, "Q", shift left] & \Mod_R \arrow[l, "Z", shift left]
	\end{tikzcd}\]
	define two Quillen adjunctions for the induced model structures from \Cref{subsection:model-structures-algebras}, with left adjoints the top arrows. The Quillen adjunction $K \dashv I$ is a Quillen equivalence. Moreover, the functors $K$ and $Z$ are homotopical.
\end{prop}
\begin{proof}
	The fact that the two pairs of functors are adjoint can be seen from explicitly considering the unit and counit. For the first adjunction, the unit is given by the square
	\[ \begin{tikzcd}
	J \arrow[r, "\textup{incl}"] \arrow[d] & R\iP J \arrow[d, "\pr"] \\
	\ast \arrow[r] & R
	\end{tikzcd}\]
	and the counit by $(\eta,\text{incl})\colon R\iP I(A) \to A$, where $\eta\colon R\to A$ is the unit map for the algebra $A$. For the second adjunction, the unit is given by the projection $J\to Q(J)$ from the defining pushout, considered as a map of non-unital commutative monoids for the zero multiplication on $Q(J)$. The counit is defined by the square
	\[ \begin{tikzcd}
	M\otimes_R M \arrow[r, "\ast"] \arrow[d] & M \arrow[d, "\id"] \\
	\ast \arrow[r] & M
	\end{tikzcd}\]
	for any $R$-module $M$, where the top morphism is the zero map.

	The first adjunction then is a Quillen adjunction, since $K$ preserves both cofibrations and acyclic cofibrations of non-unital commutative algebras. For this, it suffices to check the generating (acyclic) cofibrations, which are of the form $R\otimes \power{>0} f$ for $f$ a generating (acyclic) cofibration of $\mathscr C$. These morphisms are sent by $K$ to $R\otimes \power{} f$, which are precisely the generating (acyclic) cofibrations for the model structure on the category of augmented algebras. We conclude that $K$ is left Quillen and hence the first adjunction is a Quillen adjunction. Furthermore, since coproducts are homotopical in $\mathscr C$, also $K$ is homotopical.

	For the second adjunction, we observe that the right adjoint $Z$ is the identity on underlying objects and morphisms. Since for both the model category of non-unital algebras and modules, the fibrations and weak equivalences are defined on underlying objects, we see that $Z$ is right Quillen and homotopical.

	Finally, we need to check that $K\dashv I$ is a Quillen equivalence. We check the following criterion from \cite[Definition 1.3.12]{hovey}: Let $J\in \NUCA_R$ be a cofibrant non-unital commutative $R$-algebra and $S\in \CAlg_R/R$ be a fibrant augmented algebra. Then we have to show that a morphism $f\colon KJ\to S$ is a weak equivalence if and only if its adjoint $\tilde{f} \colon J\to IS$ is a weak equivalence. Since $S$ is fibrant, the augmentation $\epsilon \colon S\to R$ is a fibration in $\mathscr C$, and hence the defining diagram
	\[ \begin{tikzcd}
	I(S) \arrow[r, "\text{incl}"] \arrow[d] & S \arrow[d, "\epsilon"] \\
	\ast \arrow[r] & R
	\end{tikzcd}\]
	is a homotopy pullback square in $\mathscr C$. The map $\eta\colon R\to A$ defines a section to $\epsilon$, so the distinguished triangle
	\[
		I(S)\xrightarrow{\text{incl}} S\xrightarrow{\;\epsilon\;} R\to \Sigma I(S)
	\]
	in $\Ho(\mathscr C)$ splits and the counit $(\eta,\text{incl})$ is a weak equivalence. Moreover, since $\tilde{f}$ is a retract of $R\iP \tilde{f}$ and $R\iP (\_)$ is homotopical, one is a weak equivalence if and only if the other is one. In total, since we have the relation $f = (\eta, i) \circ (R\iP \tilde{f})$, we see that $K\dashv I$ is a Quillen equivalence.
\end{proof}

We also have a base-change adjunction
\begin{equation}\label{eq:base-change-adj}
\begin{tikzcd}[column sep = large]
\CAlg_R /S \arrow[r, shift left, "S\otimes_R \_"] & \CAlg_S /S \arrow[l, shift left, "\textup{forget}"]
\end{tikzcd}
\end{equation}
for any $R$-algebra $S$. Since the forgetful functor is the identity on underlying objects and morphisms, it is right Quillen, so also this adjunction is a Quillen adjunction.

\begin{defi}
	Let $R$ be a commutative monoid in $\mathscr C$ and let $S$ be a commutative $R$-algebra. Then we define the \emph{abelianization} functor as the composite
	\[ \Ab_{S/R} \colon \Ho(\CAlg_R /S) \xrightarrow{S\otimes_R^{\cat{L}} \_ } \Ho(\CAlg_S /S) \xrightarrow{\cat{R}I} \Ho(\NUCA_S) \xrightarrow{\cat{L}Q} \Ho(\Mod_S). \]
\end{defi}

\begin{rk}
	The name \emph{abelianization} for the above construction originates in algebra, where it is an observation going back to Beck \cite{beck-cohomology} that abelian group objects in the category of augmented algebras over a commutative ring $R$ are equivalently $R$-modules, with the equivalence given by square-zero extensions in one direction and the augmentation ideal in the other. One then observes that Kähler differentials define a left adjoint functor to the inclusion of abelian group objects, interpreted as modules, into all augmented algebras.

	In our context of $G$-global homotopy theory, we show in Theorem~\ref{thm:stabilization-aug-algebras} that $\Ab_{R/R}=\cat{L}Q\cat{R}I$ can be interpreted as a global stabilization.
\end{rk}

\begin{prop}\label{prop:abelianization-adj}
	Let $R$ be a commutative monoid in $\mathscr C$ and $S$ be a commutative $R$-algebra. The functors
	\[ \begin{tikzcd}[column sep = large]
	\Ho(\CAlg_R/S) \arrow[r, shift left, "\Ab_{S/R}"] & \Ho(\Mod_S) \arrow[l, shift left, "\cat{L}K\circ \cat{R}Z"]
	\end{tikzcd}\]
	are adjoint, with abelianization being the left adjoint.
\begin{proof}
	We use the Quillen adjunctions from \Cref{prop:aug-ideal-Quillen-equivalence} and \eqref{eq:base-change-adj} to calculate for an $S$-module $M$ and an $R$-algebra $T$ with augmentation to $S$:\phantom{\qedhere}
	\begin{align*}
	\Ho(\Mod_S)(\Ab_{S/R}(T), M) & \cong \Ho(\NUCA_S) ((\cat{R}I)(S\otimes_R^{\cat{L}} T), (\cat{R}Z)(M))\\
	& \cong \Ho(\CAlg_S/S) ((\cat{L}K)(\cat{R}I)(S\otimes_R^{\cat{L}} T), (\cat{L}K)(\cat{R}Z)(M))\\
	& \cong \Ho(\CAlg_R/S) (T, (\cat{L}K)(\cat{R}Z)(M)).\pushQED{\qed}\qedhere\popQED
	\end{align*}
\end{proof}
\end{prop}

Using these adjunctions, we now define the cotangent complex:

\begin{defi}\label{def:cotangent-complex}
	Let $R$ be a commutative monoid in $\mathscr C$ and $S$ be an $R$-algebra. The cotangent complex of $S$ over $R$ is defined as the object
	\[ \Omega_{S/R} = \Ab_{S/R} (S) =(\cat{L}Q) (\cat{R}I)(S \otimes_R^{\cat{L}} S)\]
	in $\Ho(\Mod_S)$.
\end{defi}

\begin{ex}\label{ex:R=S}
	If $R\to S$ is a weak equivalence, then $S$ is initial in $\Ho(\CAlg_R/S)$. As $\Ab_{S/R}$ is a left adjoint by the proposition above, we conclude that $\Omega_{S/R}=0$.
\end{ex}

For any square
\[\begin{tikzcd}
R\arrow[r]\arrow[d] & R^\prime\arrow[d]\\
S \arrow[r] & S^\prime
\end{tikzcd}\]
in the model category $\CAlg$, we obtain an induced morphism on cotangent complexes $\Omega_{S/R}\to \Omega_{S^\prime/R^\prime}$ as $S$-modules. This arises from the universal property of pushouts and pullbacks.

For any commutative $R$-algebra $S$, we also obtain a natural \emph{universal derivation}
\begin{equation}\label{eq:universal-derivation}
d_{S/R}\colon S\to \Omega_{S/R}
\end{equation}
by composing the unit $S\to S\iP \Omega_{S/R} = (KZ)(\Ab_{S/R}(S))$ of the adjunction in \Cref{prop:abelianization-adj} with the projection to $\Omega_{S/R}$.

This cotangent complex behaves just as in the classical case of commutative algebras, in that is comes with a transitivity exact sequence and a base-change formula. These properties allow us to consider the cotangent complex as defining a cohomology theory for augmented commutative algebras. In order to establish these properties, we first need to consider how the augmentation ideal and indecomposables behave under base changes.

\begin{lemma}\label{lemma:aug-ideal-base-change}
	Let $R$ be a commutative monoid in $\mathscr C$, let $S$ be a commutative $R$-algebra, and let $T$ be an augmented commutative $R$-algebra. Then the commutative square
	\begin{equation}\label{diag:aug-ideal-base-change}
	\begin{tikzcd}
	I_R(T) \otimes_R S \arrow[r]\arrow[d] & T\otimes_R S \arrow[d]\\
	\ast \arrow[r] & S
	\end{tikzcd}
	\end{equation}
	induces a morphism
	\[ \cat{R}I_R (T) \otimes_R^{\cat{L}} S \to \cat{R}I_S(T\otimes_R^{\cat{L}} S), \]
	which is an isomorphism in the homotopy category of non-unital commutative $S$-algebras.
\end{lemma}
\begin{proof}
	The square \eqref{diag:aug-ideal-base-change} induces a natural transformation $I_R(\_) \otimes_R S \Rightarrow I_S(\_\otimes_R S)$ by considering pullbacks. Here, the functor $(\_)\otimes_R S$ can be left derived, and the functors $I$ can be right derived. We thus consider the double category of model categories and left and right derivable functors. In this context, \cite[Theorem 7.6]{shulman-left-right-derived} shows that taking homotopy categories and derived functors is a double pseudofunctor. In particular, the natural transformation above induces a transformation
	\[ \cat{R}I_R(\_) \otimes_R^{\cat{L}} S \Rightarrow \cat{R}I_S(\_\otimes_R^{\cat{L}} S)\]
	as desired.

	Since $I$ is a Quillen equivalence with inverse $K$, we can show that this transformation is a natural isomorphism in the homotopy category by considering its mate
	\[ \cat{L}K_S(\_\otimes_R^{\cat{L}} S) \Rightarrow \cat{L}K_R(\_)\otimes_R^{\cat{L}} S. \]
	This transformation is induced from the natural isomorphism $S\iP (\_\otimes_R S) \cong (R\iP \_) \otimes_R S$ of left Quillen functors, and hence is a natural isomorphism.
\end{proof}

\begin{lemma}\label{lemma:indec-base-change}
	Let $R$ be a commutative monoid in $\mathscr C$, let $S$ be a commutative $R$-algebra and $J$ be a non-unital commutative $R$-algebra. Then the commutative square
	\begin{equation}\label{diag:indec-base-change}
	\begin{tikzcd}
	(J\otimes_R J) \otimes_R S \arrow[r]\arrow[d] & J\otimes_R S \arrow[d]\\
	\ast \arrow[r] & Q_R(J) \otimes_R S
	\end{tikzcd}
	\end{equation}
	induces a morphism
	\[ \cat{L}Q_S (J\otimes_R^{\cat{L}} S) \to \cat{L}Q_R(T) \otimes_R^{\cat{L}} S, \]
	which is an isomorphism in the homotopy category of $S$-modules.
\end{lemma}
\begin{proof}
	As both the functors labelled $Q$ and the base change $(\_)\otimes_R S$ are left Quillen functors, the desired morphism can be obtained from the natural morphism $Q_S(J \otimes_R^{\cat{L}} S) \to Q_R(J) \otimes_R^{\cat{L}} S$ induced on pushouts by the diagram \eqref{diag:indec-base-change}, by considering cofibrant replacements. Moreover, as $(\_)\otimes_R S$ is a left adjoint, it preserves pushouts, which implies that this morphism is an isomorphism.
\end{proof}

\begin{lemma}\label{lemma:cofib-base-change}
	Let $R$ be a commutative monoid in $\mathscr C$ and let $S$ and $T$ be commutative $R$-algebras with a map $S\to T$ of commutative $R$-algebras. Assume moreover that $\CAlg(\mathscr C)$ is left proper or that $S$ is cofibrant as a commutative $R$-algebra. Then a homotopy cofiber of $S\otimes_R^{\cat{L}} T\to T\otimes_R^{\cat{L}} T$ in $\CAlg_T/T$ is given by $T\otimes_S^{\cat{L}} T$.
\end{lemma}
\begin{proof}
	This statement may be shown by considering cofibrant replacements. In order to calculate $S\otimes_R^{\cat{L}} T$, we consider a cofibrant replacement $\Gamma_RS\to S$ of $S$ as a commutative $R$-algebra. Moreover, in order to calculate $T \otimes_R^{\cat{L}} T$, we decompose the morphism $\Gamma_R S\to S\to T$ into a cofibration followed by a weak equivalence as $\Gamma_R S\to \Gamma_R T\to T$.

	We now take the homotopy cofiber of
	\[ \Gamma_R S\otimes_R T\to \Gamma_R T\otimes_R T.\]
	This morphism arises from applying the left Quillen functor $\_\otimes_R T\colon \CAlg_R/T \to \CAlg_T/T$ to the cofibration $\Gamma_RS\to \Gamma_RT$ between cofibrant objects, so it is itself a cofibration between cofibrant objects. Hence, its homotopy cofiber is represented by the $1$-categorical cofiber, which can be calculated as
	\[ (\Gamma_R T\otimes_R T)\otimes_{\Gamma_R S\otimes_R T} T \cong (\Gamma_R T\otimes_{\Gamma_R S} \Gamma_R S\otimes_R T)\otimes_{\Gamma_R S\otimes_R T} T \cong \Gamma_R T\otimes_{\Gamma_R S} T. \]
	We now need to compare this with $\Gamma_S T\otimes_S T$, where $\Gamma_ST$ is a cofibrant replacement of $T$ as a commutative $S$-algebra. For this, it suffices to establish a weak equivalence $\Gamma_RT \otimes_{\Gamma_RS} S \to \Gamma_ST$ of commutative $S$-algebras. As both $\Gamma_RT \otimes_{\Gamma_RS} S$ and $\Gamma_ST$ are cofibrant $S$-algebras (since $\_\otimes_{\Gamma_RS}S$ is left Quillen and by definition, respectively), the left Quillen functor $\_\otimes_ST$ then preserves this weak equivalence.

	In order to establish this weak equivalence, we consider the diagram
	\[\begin{tikzcd}[column sep = large]
	\Gamma_RS \arrow[r, "\simeq"] \arrow[dd] & S \arrow[r] \arrow[d] & \Gamma_S T\arrow[dd, "\simeq"]\\
	& \Gamma_RT\otimes_{\Gamma_RS} S \arrow[ru, dashed] \\
	\Gamma_R T \arrow[rr, "\simeq", swap] \arrow[ru] \arrow[rruu, dashed, bend right=20] && T.
	\end{tikzcd}\]
	Here, $S\to \Gamma_ST \to T$ is the factorization of $S\to T$ into a cofibration followed by an acyclic fibration provided by the cofibrant replacement $\Gamma_S T$. As $\Gamma_R S\to \Gamma_R T$ is a cofibration, so is $S\to \Gamma_RT\otimes_{\Gamma_RS} S$ and the diagonal dashed morphisms exist by the lifting property. Moreover, since either the categories in question are left proper or $S$ is cofibrant, we conclude that the morphism $\Gamma_R T\to \Gamma_RT\otimes_{\Gamma_RS} S$ is a weak equivalence. Hence we indeed have a weak equivalence $\Gamma_RT \otimes_{\Gamma_RS} S \to \Gamma_ST$ as desired by the 2-out-of-3 property.
\end{proof}

From this cofiber calculation, we deduce the transitivity sequence for the cotangent complex.

\begin{thm}\label{thm:transitivity_sequence_topological_cotangent_complex}
	Let $R\to S\to T$ be a sequence of commutative monoids in $\mathscr C$. Moreover, assume that either $\CAlg(\mathscr C)$ is left proper or that $S$ is cofibrant as a commutative $R$-algebra. Then the sequence
	\[ \Omega_{S/R}\otimes_S^{\cat{L}} T \to \Omega_{T/R} \to \Omega_{T/S}, \]
	induced from functoriality of $\Omega$, admits the structure of a homotopy cofiber sequence of $T$-modules.
\end{thm}
\begin{proof}
	We consider the morphism $S\otimes_R^{\cat{L}} T \to T\otimes_R^{\cat{L}} T$ of $T$-algebras augmented to $T$. By \Cref{lemma:cofib-base-change}, the homotopy cofiber of this morphism in $\CAlg_T/T$ is given by $T\otimes_S^{\cat{L}} T$. Since $I$ is a Quillen equivalence and $Q$ is left Quillen, applying $\cat{L} Q_T\circ \cat{R} I_T$ to the resulting homotopy cofiber sequence yields a homotopy cofiber sequence of $T$-modules. This takes the form
	\[ \cat{L} Q_T(\cat{R} I_T(S\otimes_R^{\cat{L}} T)) \to \cat{L} Q_T(\cat{R} I_T(T\otimes_R^{\cat{L}} T)) \to \cat{L} Q_T(\cat{R} I_T(T\otimes_S^{\cat{L}} T)). \]
	The last two terms are by definition the cotangent complexes $\Omega_{T/R}$ and $\Omega_{T/S}$. We thus only have to identify the first term as $\Omega_{S/R} \otimes_S^{\cat{L}} T$.

	For this, we observe that $S\otimes_R^{\cat{L}} T \cong (S\otimes_R^{\cat{L}} S) \otimes_S^{\cat{L}} T$ and use \Cref{lemma:aug-ideal-base-change,lemma:indec-base-change} to calculate
	\begin{align*}
	\cat{L} Q_T(\cat{R} I_T(S\otimes_R^{\cat{L}} T)) & \cong \cat{L} Q_T(\cat{R} I_T((S\otimes_R^{\cat{L}} S)\otimes_S^{\cat{L}} T))\\
	& \cong \cat{L} Q_T( \cat{R} I_S(S\otimes_R^{\cat{L}} S) \otimes_S^{\cat{L}} T)\\
	& \cong \cat{L} Q_S(\cat{R} I_S(S\otimes_R^{\cat{L}} S)) \otimes_S^{\cat{L}} T\\
	& \cong \Omega_{S/R} \otimes_S^{\cat{L}} T.
	\end{align*}
	In total, we obtain the desired homotopy cofiber sequence.
\end{proof}

As a special case of the previous theorem we can now make precise that the cotangent complex is homotopy invariant in each variable.

\begin{cor}\label{cor:cotang-htpy-inv-R}
	Assume $\CAlg(\mathscr C)$ is left proper, let $R'\to R$ be a weak equivalence of commutative monoids in $\mathscr C$, and let $S$ be a commutative $R$-algebra. Then the induced map $\Omega_{S/R'}\to\Omega_{S/R}$ in the homotopy category of $S$-modules is an isomorphism.
	\begin{proof}
		By the previous theorem, the homotopy fiber of the this map is given by $\Omega_{R/R'}\otimes^{\cat L}_RS$, which is trivial by Example~\ref{ex:R=S}. The claim follows by stability of $\mathscr C$.
	\end{proof}
\end{cor}

In the same way one shows:

\begin{cor}\label{cor:cotang-htpy-inv-S}
	Assume $\CAlg(\mathscr C)$ is left proper, let $R$ be a commutative monoid, and let $S\to S'$ be a weak equivalence of commutative $R$-algebras. Then the induced map $\Omega_{S/R}\otimes^\cat{L}_SS'\to\Omega_{S'/R}$ is an isomorphism in the homotopy category of $S'$-modules.\qed
\end{cor}

Moreover, we also get base change and additivity results.

\begin{prop}\label{prop:base_change-cotangent-complex}
	Let $R$ be a commutative monoid in $\mathscr C$ and $S$ and $T$ be two commutative $R$-algebras. Assume that $\CAlg(\mathscr C)$ is left proper or that $S$ and $T$ are cofibrant in $R/\CAlg(\mathscr C)$. Then there are natural isomorphisms
	\begin{align*}
	\Omega_{S\otimes_R^{\cat{L}} T/T} &\cong \Omega_{S/R} \otimes_R^{\cat{L}} T \qquad \textrm{and}\\
	\Omega_{S\otimes_R^{\cat{L}} T/R} &\cong (\Omega_{S/R} \otimes_R^{\cat{L}} T) \iP (S\otimes_R^{\cat{L}}\Omega_{T/R}).
	\end{align*}
\end{prop}
\begin{proof}
	For the first assertion, we calculate
	\begin{align*}
	\Omega_{S\otimes_R^{\cat{L}} T/T} & \cong \cat{L} Q_{S\otimes_R^{\cat{L}} T}(\cat{R} I_{S\otimes_R^{\cat{L}} T}((S\otimes_R^{\cat{L}} T)\otimes_T^{\cat{L}} (S\otimes_R^{\cat{L}} T)))\\
	& \cong \cat{L} Q_{S\otimes_R^{\cat{L}} T}(\cat{R} I_{S\otimes_R^{\cat{L}} T}((S\otimes_R^{\cat{L}} S)\otimes_S^{\cat{L}} (S\otimes_R^{\cat{L}} T)))\\
	& \cong \cat{L} Q_{S\otimes_R^{\cat{L}} T}(\cat{R} I_S(S\otimes_R^{\cat{L}} S) \otimes_S^{\cat{L}} (S\otimes_R^{\cat{L}} T)) \\
	& \cong \cat{L} Q_S(\cat{R} I_S(S\otimes_R^{\cat{L}} S)) \otimes_S^{\cat{L}} (S\otimes_R^{\cat{L}} T)\\
	& \cong \Omega_{S/R}\otimes_R^{\cat{L}} T.
	\end{align*}
	The second assertion follows from observing that the transitivity cofiber sequences for $R\to S\to S\otimes_R^{\cat{L}} T$ and $R\to T\to S\otimes_R^{\cat{L}} T$ fit together to define splittings for each other, and applying the first assertion to the resulting cotangent complexes.
\end{proof}

\subsection{Topological André-Quillen cohomology of \texorpdfstring{$\bm G$}{G}-global ring spectra}
By Proposition~\ref{prop:white-assumptions}, the positive $G$-global model structures on $\cat{$\bm G$-Spectra}$ are symmetric monoidal and satisfy the monoid and strong commutative monoid axiom; moreover, coproducts in them are homotopical by \cite[Lemma~3.1.43]{g-global}. We can therefore specialize the above discussion to this setting, yielding:

\begin{defi}\label{def:TAQ}
	Let $R$ be a $G$-global ultra-commutative ring spectrum and $S$ be an $R$-algebra. The cotangent complex of $S$ over $R$ is defined as
	\[ \Omega_{S/R} = (\cat{L}Q) (\cat{R}I)(S \smashp_R^{\cat{L}} S).\]
	This is an $S$-module, and the homology and cohomology theories represented by it are called \emph{($G$-global) topological André-Quillen (co)homology}. Explicitly, for an $S$-module $M$, topological André-Quillen (co)homology of $S$ over $R$ with coefficients in $M$ is defined as
	\begin{align*}
		\ul{\TAQ}_\ast (S, R;M) & = \ul{\hat\pi}_\ast (\Omega_{S/R}\smashp_S^{\cat{L}} M) \\
		\ul{\TAQ}^\ast (S, R;M) & = \ul{\hat\pi}_{-\ast} (\cat{R} F(\Omega_{S/R}, M)).
	\end{align*}
	Here, $F$ denotes the internal hom in the category of $S$-modules.
\end{defi}

As a consequence of \Cref{thm:transitivity_sequence_topological_cotangent_complex} and \Cref{prop:base_change-cotangent-complex}, we obtain a transitivity sequence and base change for these. We explicitly state the transitivity sequence in homotopy groups.

\begin{cor}\label{cor:transitivity-seq-homotopy-groups}
	Let $R\to S\to T$ be morphisms of $G$-global ultra-commutative ring spectra, and let $M$ be a $T$-module. Then there are long exact sequences
	\[\hskip-11.67pt\hfuzz=12pt \cdots \to \ul{\TAQ}_{n+1}(T, S; M) \to \ul{\TAQ}_n(S,R; M)\to \ul{\TAQ}_n(T,R; M) \to \ul{\TAQ}_{n}(T,S;M) \to \cdots \]
	and
	\[\hskip-11.67pt\hfuzz=12pt \cdots \to \ul{\TAQ}^n(T, S; M) \to \ul{\TAQ}^n(T,R; M)\to \ul{\TAQ}^n(S,R; M) \to \ul{\TAQ}^{n+1}(T,S;M) \to \cdots \]
	induced from the cofibre sequence from \Cref{thm:transitivity_sequence_topological_cotangent_complex}. \qed
\end{cor}

We now present two applications of the above theory to the study of $G$-global ultra-commutative ring spectra. The first is a Hurewicz theorem, which says that vanishing of the (relative) André-Quillen homology detects equivalences of $G$-global ultra-commutative ring spectra. The other application is a construction of Postnikov towers for global ultra-commutative ring spectra, with $k$-invariants in global topological André-Quillen cohomology. These applications are analogous to the usage of topological André-Quillen homology in \cite[Chapter 8]{basterra-TAQ}, but the proofs are simplified by the use of t-structures.

Both of these results need a connectivity hypothesis, and the first step is to consider how the indecomposables functor $Q$ interacts with connectivity.

\begin{lemma}\label{lemma:Q-connectivity}
	Let $R$ be a connective $G$-global ultra-commutative ring spectrum and $J$ be a cofibrant non-unital commutative $R$-algebra. Suppose moreover that $J$ is $n$-connected for $n\geq 0$. Then also $Q(J)$ is $n$-connected, and the adjunction unit $q\colon J\to Q(J)$ induces an isomorphism on $\ul{\hat\pi}_k$ for $n+1\leq k\leq 2n+1$.
\end{lemma}
\begin{proof}
	The module of indecomposables is defined via the cofiber sequence
	\[ J\smashp_R J \to J\xrightarrow{\eta} Q(J) \]
	of $R$-modules. Here, $J$ is $n$-connected by assumption, while $J\smashp_RJ$ is $(2n+1)$-connected by Proposition~\ref{prop:t-structure-smash}; the claim follows from the long exact sequence.
\end{proof}

We now come to the Hurewicz theorem. To state this, recall the universal derivation $d_{S/R}\colon S\to\Omega_{S/R}$ for any $R$-algebra $\eta\colon R\to S$. By naturality, the composite
\begin{equation*}
	R\xrightarrow{\;\eta\;}S\xrightarrow{d_{S/R}}\Omega_{S/R}
\end{equation*}
in $\Ho(\cat{Mod}^G_R)$ will factor through $\Omega_{R/R}$, so it is zero by Example~\ref{ex:R=S}. Thus, $d_{S/R}$ always \emph{non-canonically} factors through an $R$-module map $C(\eta)\to\Omega_{S/R}$.

\begin{thm}[Hurewicz theorem]\label{thm:Hurewicz}
	Let $R$ be a connective $G$-global ultra-commuta\-tive ring spectrum and let $S$ be a connective commutative $R$-algebra such that the unit map $\eta\colon R\to S$ is an $n$-equivalence for $n\geq 1$. Then $\Omega_{S/R}$ is $n$-connected, the factorization of $d_{S/R}$ through $C(\eta)$ is unique, and this induces an isomorphism $\ul{\hat\pi}_{n+1} (C(\eta)) \cong \ul{\hat\pi}_{n+1}(\Omega_{S/R})$.
\end{thm}
\begin{proof}
	After cofibrant replacement, we may assume that $S$ is a cofibrant commutative $R$-algebra. We consider the diagram
	\[ \begin{tikzcd}
	R \arrow[r, "\eta"] \arrow[d, "\eta", swap] & S \arrow[r] \arrow[d] & C(\eta)\arrow[d, "\iota"]\\
	S \arrow[r] & S\smashp_R S \arrow[r] & S\smashp_R C(\eta).
	\end{tikzcd}\]
	The first line of the diagram is a cofiber sequence by definition, and the second line arises from it by applying $S\smashp_R \_$. Since $S$ is cofibrant, this again is a cofiber sequence. The vertical morphisms are the inclusions as the right factors. By \Cref{prop:aug-ideal-Quillen-equivalence}, the counit
	\[ S\vee \cat{R}I(S\smashp_R S) \to S\smashp_R S\]
	is an equivalence, and thus by comparing cofibers we obtain $S\smashp_R C(\eta) \cong \cat{R}I (S\smashp_R S)$. We now consider the composition
	\[ C(\eta)\xrightarrow{\iota} S\smashp_R C(\eta) \cong \cat{R}I(S\smashp_R S) \xrightarrow{q} \Omega_{S/R},\]
	where the last map is an instance of the unit map $J\to \cat{L}Q(J)$ considered in \Cref{lemma:Q-connectivity}, for the non-unital algebra $J= S\smashp_R C(\eta)\cong \cat{R}I(S\smashp_R S) $. Since $\eta$ is $n$-connected, so is $C(\eta)$. As $S$ is connective, also $S\smashp_R C(\eta)$ is $n$-connected, and the morphism $\iota$ is an isomorphism on $\ul{\hat\pi}_k$ for $k\leq {n+1}$. By \Cref{lemma:Q-connectivity}, also $q$ induces an isomorphism on $\ul{\hat\pi}_k$ for $k\leq n+1$, finishing the proof that $\Omega_{S/R}$ is $n$-connected and that $\ul{\hat\pi}_{n+1}(C(\eta))\cong\ul{\hat\pi}_{n+1}(\Omega_{S/R})$ via the above composite $C(\eta)\to\Omega_{S/R}$.

	Unravelling the definition of this morphism $C(\eta)\to \Omega_{S/R}$, we observe that it is indeed induced by the universal derivation $d_{S/R}\colon S\to \Omega_{S/R}$. It then only remains to prove uniqueness of the factorization through $C(\eta)$, for which we consider the following portion of the long exact sequence in cohomology associated to the cofiber sequence $R\to S\to C(\eta)$:
	\begin{equation*}
		\cdots\gets[S,\Omega_{S/R}]\gets[C(\eta),\Omega_{S/R}]\gets[\Sigma R,\Omega_{S/R}]\gets\cdots
	\end{equation*}
	Here $[\,{,}\,]$ denotes hom sets in $\Ho(\cat{Mod}^G_R)$. The claim amounts to saying that $[S,\Omega_{S/R}]\gets [C(\eta),\Omega_{S/R}]$ is injective, for which it suffices that $[\Sigma R,\Omega_{S/R}]=0$. But indeed, by adjunction $[\Sigma R,\Omega_{S/R}]\cong\hat\pi_1^{\id_G}(\Omega_{S/R})$ which vanishes as $\Omega_{S/R}$ is $n$-connected by the above and we have assumed $n\ge1$.
\end{proof}

\begin{cor}\label{cor:TAQ-detects-equiv}
	Let $R$ be a connective $G$-global ultra-commutative ring spectrum, and let $S$ be a connective $R$-algebra such that the unit map $\eta\colon R\to S$ is $1$-connected and $\Omega_{S/R}\simeq0$. Then $\eta\colon R\to S$ is already an equivalence.
\end{cor}
\begin{proof}
	Suppose $\eta$ is not an equivalence, and let $\ul{\hat\pi}_k(C(\eta))$, $k\ge2$ be the first non-trivial homotopy group of the cone. Then also $\ul{\hat\pi}_k(\Omega_{S/R})\neq 0$ by \Cref{thm:Hurewicz}, in contradiction to triviality of the cotangent complex.
\end{proof}

Next, we construct Postnikov towers for global ultra-commutative ring spectra. For this, we explain how elements of the topological André-Quillen cohomology parametrize extensions of algebras.

\begin{constr}\label{constr:extension-k-inv}
	Let $R$ be a $G$-global ultra-commutative ring spectrum, $S$ be a commutative $R$-algebra and $M$ be an $S$-module. By adjunction and the definition of topological André-Quillen cohomology, we have the identifications
	\begin{align*}
		\TAQ_{\id_G}^n(S, R; M) &\cong
		\Ho(\cat{$\bm G$-Spectra})\big(\mathbb S, \Sigma^n\cat{R}F(\Omega_{S/R},M)\big)
		\\&\cong\Ho(\cat{$\bm G$-Spectra})\big(\mathbb S, \cat{R}F(\Omega_{S/R},\Sigma^nM)\big)\\&\cong
		\Ho(\cat{Mod}^G_S)\big(S, \cat{R}F(\Omega_{S/R},\Sigma^nM)\big)\\&\cong
		\Ho(\cat{Mod}^G_S)(\Omega_{S/R}, \Sigma^n M)\\
		&\cong \Ho(\cat{UCom}^G_R/S) (S, S\vee \Sigma^nM).
	\end{align*}
	In particular, for a given class $k\in \TAQ_{\id_G}^n(S, R; M)$, we may interpret it as a morphism $S\to S\vee \Sigma^nM$ of commutative $R$-algebras over $S$. We form the homotopy pullback
	\[ \begin{tikzcd}
	S[k] \arrow[r] \arrow[d] & S \arrow[d, "\iota"]\\
	S \arrow[r, "k", swap] & S\vee \Sigma^nM
	\end{tikzcd}\]
	in $R$-algebras and call it the extension of $S$ by $k$. Here, the right vertical map is the inclusion of $S$ as the first wedge summand.

	Topological André-Quillen cohomology can be related to usual cohomology represented by a $G$-global spectrum by the universal derivation $d_{S/R}\colon S\to \Omega_{S/R}$ constructed in \eqref{eq:universal-derivation}. Precomposition with this derivation takes a map $\tilde{k}\colon \Omega_{S/R}\to \Sigma^n M$ representing an element in $\TAQ_{\id_G}^n(S, R; M)$ to the map $\tilde{k}\circ d_{S/R}\colon S\to \Sigma^n M$. By the definition of the universal derivation, this agrees (in the homotopy category) with the composition
	\[ \pr\circ k\colon S\to S\vee \Sigma^n M \to \Sigma^n M,\]
	where $k$ is the adjoint to $\tilde{k}$ under the adjunction between square-zero extensions and Kähler differentials.

	Using this translation, we observe that $S[k]$ is also the homotopy pullback in the total square
	\[ \begin{tikzcd}
	S[k] \arrow[r] \arrow[d] & S \arrow[d, "\iota"] \arrow[r] & \ast \arrow[d]\\
	S \arrow[r, "k", swap] & S\vee \Sigma^nM \arrow[r, "\pr", swap] & \Sigma^n M
	\end{tikzcd}\]
	and hence the homotopy fiber of $\tilde{k}\circ d_{S/R}$ in the category of $R$-modules.
\end{constr}

If $X$ is any ultra-commutative ring spectrum, then its zeroth homotopy groups $\ul{\hat\pi}_0X$ come with additional norm maps, defined via the multiplication on smash powers, giving them the structure of a so-called \emph{global power functor} \cite[Definition~5.1.6]{schwede-book}. In the next theorem, we will need that conversely any global power functor $F$ gives rise to an \emph{ultra-commutative} Eilenberg-MacLane spectrum $HF$, i.e.~an ultra-commutative global ring spectrum with $\ul{\hat\pi}_0(HF)\cong F$ as global power functors and $\ul{\hat\pi}_k(HF)=0$ for $k\not=0$, see \cite[Theorem~5.4.14]{schwede-book}. As the corresponding result in $G$-global homotopy theory for general $G$ has not been established yet, this means we have to restrict to $G=1$ here; however, once the corresponding theory of Eilenberg-MacLane spectra is set up, the same result will hold for arbitrary $G$, with the same proof.

\begin{thm}\label{thm:Postnikov-towers}
	Let $R$ be a connective global ultra-commutative ring spectrum. Then there is a sequence $R_0, R_1, \ldots$ of commutative $R$-algebras together with maps $R_{n+1}\to R_n$ of $R$-algebras and classes $k_n\in \TAQ_e^{n+1}(R_n, R; H\ul{\hat\pi}_{n+1}(R))$, such that the following properties are satisfied:
	\begin{enumerate}
		\item $R_0\cong H\ul{\hat\pi}_0(R)$ and $R_{n+1} \cong R_{n}[k_n]$,
		\item $\ul{\hat\pi}_k(R_n)=0$ for $k>n$,
		\item the unit maps $\eta_n\colon R\to R_n$ are $(n+1)$-equivalences.
	\end{enumerate}
\end{thm}
\begin{proof}
	We define $R_0= H\ul{\hat\pi}_0(R)$ as an Eilenberg-MacLane spectrum for the global power functor $\ul{\hat\pi}_0(R)$. This is a global ultra-commutative ring spectrum, and it comes with a morphism $\eta_0\colon R\to H\ul{\hat\pi}_0(R)$ of ultra-commutative ring spectra inducing an isomorphism on $\ul{\hat\pi}_0$ (see the remark before \cite[Theorem 5.4.14]{schwede-book}). Hence $R_0$ is a possible first stage of the Postnikov tower.

	Now suppose that we have constructed the tower up to level $n$. In particular, we have a morphism $\eta_n\colon R\to R_n$ of ultra-commutative ring spectra that is an $(n+1)$-equivalence, and $\ul{\hat\pi}_{n+2}(R_n)= \ul{\hat\pi}_{n+1}(R_n) = 0$. The Hurewicz theorem \ref{thm:Hurewicz} shows that thus $\Omega_{R_n/R}$ is $(n+1)$-connected and $\ul{\hat\pi}_{n+2}(\Omega_{R_n/R})\cong \ul{\hat\pi}_{n+1}(R)$. This isomorphism defines a morphism $\tilde{k}_n\colon \Omega_{R_n/R}\to \Sigma^{n+2} H\ul{\hat\pi}_{n+1}(R)$ of $R$-modules, which corresponds to an element $k_n\in \TAQ_e^{n+1}(R_n, R; H\ul{\hat\pi}_{n+1}(R))$.

	Using this element and \Cref{constr:extension-k-inv}, we define $R_{n+1}= R_n[k_n]$. This comes with a map $R_{n+1}\to R_n$ of $R$-algebras. Furthermore, as an $R$-module, the algebra $R_{n+1}$ is the homotopy fiber of the map \[R_n\xrightarrow{d_{R_n/R}} \Omega_{R_n/R} \xrightarrow{\tilde{k}_n} \Sigma^{n+2} H\ul{\hat\pi}_{n+1}(R).\] Hence the morphism $\eta_{n+1}\colon R\to R_{n+1}$ is indeed an $(n+2)$-equivalence and all higher homotopy groups of $R_{n+1}$ vanish. Thus, the theorem follows by induction.
\end{proof}

\subsection{Equivariant topological André-Quillen (co)homology}\label{subsec:equivariant-TAQ}
By \cite[Theorem~4.7]{hausmann-equivariant} the $G$-equivariant flat model structure recalled in Theorem~\ref{thm:equiv-stable} also admits a `positive' version, with the same weak equivalences, but with cofibrations the \emph{positive} flat cofibrations. In this subsection, we will specialize the cotangent complex formalism to this setting, yielding a definition of $G$-equivariant topological André-Quillen (co)homology, and we will explain how it relates to the $G$-global theory developed in the previous section.

\begin{prop}
	The positive $G$-equivariant flat model structure on $G$-spectra satisfies all the assumptions for the construction of the the abstract cotangent complex, i.e.
	\begin{enumerate}
		\item It is a combinatorial stable model category.
		\item It is a symmetric monoidal model category and it satisfies both the monoid axiom and the commutative monoid axiom.
		\item Finite coproducts are homotopical.
	\end{enumerate}
	Moreover, the induced model structure on commutative monoids (i.e.~commutative $G$-ring spectra) is left proper.
	\begin{proof}
		By \cite[discussion after Theorem~4.8]{hausmann-equivariant}, the positive flat equivariant model structure is combinatorial. Moreover, it is stable as the usual equivariant flat model structure is so.

		By \cite[Subsection~6.1]{hausmann-equivariant}, the model structure is moreover symmetric monoidal and satisfies the monoid axiom, while the strong commutative monoid axiom is verified in Proposition~6.22 of \emph{op.~cit.}

		The third property is an instance of \cite[Proposition~4.2-(1)]{hausmann-equivariant}.

		Finally, to show that the induced model structure on commutative monoids is left proper, it suffices to verify the assumptions of Proposition~\ref{prop:alg-left-proper}.

		The positive flat model structure is left proper by \cite[Proposition~4.2-(4)]{hausmann-equivariant}, and more generally \emph{loc.~cit.} shows that for any cofibration $i$ pushouts along $X\smashp i$ are homotopy pushouts. To prove that filtered colimits in it are homotopical, it suffices to consider the usual equivariant flat model structure. By \cite[Lemma~A.2.4]{g-global} it then further suffices to prove the analogous statement for the equivariant flat \emph{level} model structure, where this is clear.

		The standard generating cofibrations are maps of cofibrant objects. Finally, the equivariant Flatness Theorem \cite[Proposition~6.2-(i)]{hausmann-equivariant} in particular shows that smashing with any fixed cofibrant object is homotopical. This completes the verification of the assumptions of Proposition~\ref{prop:alg-left-proper} and hence the proof that the transferred model structure on commutative monoids is left proper.
	\end{proof}
\end{prop}

We therefore obtain a $G$-equivariant cotangent complex for any map $R\to S$ of commutative $G$-ring spectra, that we will denote by $\Omega_{S/R}^\textup{$G$-equiv}$ to avoid confusion with the $G$-global cotangent complex $\smash{\Omega_{S/R}^\textup{$G$-gl}}$ of the previous section. Explicitly:

\begin{defi}
	Let $R$ be a commutative $G$-ring spectrum, and let $S$ be a commutative $R$-algebra. We define
	\begin{equation*}
		\Omega_{S/R}^\textup{$G$-equiv}\mathrel{:=} \cat{L}Q\cat{R}I(S\smashp_R^\cat{L}S)
	\end{equation*}
	and call it the \emph{$G$-equivariant cotangent complex}. Here $Q$ is the indecomposables functor as before, $I$ denotes the augmentation ideal, and all functors are derived with respect to the (lifts of the) \emph{equivariant} positive flat model structure.

	For any $S$-module $M$ we then define the \emph{$G$-equivariant topological André-Quillen (co)homology groups} as the true equivariant homotopy groups
	\begin{align*}
		\ul\TAQ_*^\textup{equiv}(S,R;M)&=\ul{\hat\pi}_*(\Omega_{S/R}\smashp_S^\cat{L} M)\\
		\ul\TAQ^*_\textup{equiv}(S,R;M)&=\ul{\hat\pi}_{-*}(\cat{R}F(\Omega_{S/R}, M)).
	\end{align*}
\end{defi}

\begin{rk}
	For $G=1$, this recovers Basterra's original construction of topological André-Quillen cohomology \cite{basterra-TAQ} (translated from $S$-modules to symmetric spectra).
\end{rk}

The results of Subsection~\ref{subsec:abstract-cotangent} then specialize to show that this is homotopy invariant, comes with a transitivity long exact sequence, and satisfies the base-change formula.

\begin{rk}\label{rk:proj-transfer}
	There is also a positive analogue of the equivariant \emph{projective} model structure \cite[Theorem 4.8]{hausmann-equivariant}, with the usual weak equivalences and with generating cofibrations the maps
	\[
		(\bm\Sigma(A,\blank)\smashp G_+)/H\smashp(\del\Delta^n\hookrightarrow\Delta^n)_+
	\]
	for all finite \emph{non-empty} $A$, $H\subset G\times\Sigma_A$, and $n\ge0$. While this does \emph{not} satisfy the commutative monoid axiom, the transferred model structures on commutative algebras and NUCAs still exist as the Crans-Kan transfer criterion for the positive projective model structures is a formal consequence from the one for the positive \emph{flat} model structures. As both $Z$ and $I$ are still right Quillen for these model structures by direct inspection, we can then equivalently define the equivariant cotangent complex by deriving with respect to the positive equivariant \emph{projective} model structures everywhere.
\end{rk}

\subsubsection{$G$-global vs.~$G$-equivariant topological André-Quillen cohomology}
In order to compare these construction to their $G$-global analogues, we first have to understand how to embed commutative $G$-ring spectra into $G$-global ultra-commutative ring spectra.

\begin{lemma}
	The identity constitutes a Quillen adjunction
	\[
		\cat{$\bm G$-Spectra}_\textup{$G$-equivariant pos.~proj.}\rightleftarrows
		\cat{$\bm G$-Spectra}_\textup{$G$-global pos.~flat}.
	\]
	\begin{proof}
		By Proposition~\ref{prop:equivariant-vs-global}, (acyclic) cofibrations in the projective $G$-equivariant model structure are also (acyclic) cofibrations in the $G$-global flat model structure. As the positive equivariant projective model structure has fewer cofibrations than the usual equivariant projective model structure, it then only remains to show that the generating cofibrations of the \emph{positive} equivariant projective model structure are \emph{positive} flat cofibrations. This follows immediately from Lemma~\ref{lemma:positive-vs-absolute} as the above generating cofibrations are clearly isomorphisms in degree $\varnothing$.
	\end{proof}
\end{lemma}

In particular, we have an exact adjunction
\[
	\mathcal{L}\colon\Ho(\cat{$\bm G$-Spectra}_\textup{$G$-equiv})\rightleftarrows\Ho(\cat{$\bm G$-Spectra}_\textup{$G$-global}):\!\forget
\]
of triangulated categories in which the right adjoint is a localization and $\mathcal L$ is therefore fully faithful; we call the objects in its essential image \emph{left induced}.

Using the transferred model structures (Remark~\ref{rk:proj-transfer}), the identity then also defines Quillen adjunctions between commutative monoids, yielding an adjunction
\begin{equation*}
	\mathcal{L}^\text{com}\colon\Ho(\cat{UCom}^\textup{$G$-equiv})\rightleftarrows
	\Ho(\cat{UCom}^\textup{$G$-gl}):\!{\forget}.
\end{equation*}
Here the right adjoint is still a localization, so that $\mathcal L^\text{com}$ is fully faithful; again, we call the objects in its essential image \emph{left induced}.

\begin{warn}
	The diagram
	\begin{equation*}
		\begin{tikzcd}
			\Ho(\cat{UCom}^\textup{$G$-equiv})\arrow[r, "\mathcal{L}^\text{com}"]\arrow[d, "\forget"'] & \Ho(\cat{UCom}^\textup{$G$-gl})\arrow[d, "\forget"]\\
			\Ho(\cat{$\bm G$-Spectra}_\textup{$G$-equiv})\arrow[r, "\mathcal{L}"']&\Ho(\cat{$\bm G$-Spectra}_\textup{$G$-gl})
		\end{tikzcd}
	\end{equation*}
	does \emph{not} commute, even for $G=1$.

	For example, the spectrum $X\mathrel{:=}\Sigma^\bullet_+I(\bm1,\blank)\cong\bm\Sigma(\bm1,\blank)\smashp S^1$ is cofibrant in the positive non-equivariant projective model structure, so the free commutative ring spectrum $\mathbb P(X)\mathrel{:=}\bigvee_{n\ge 0}\Sigma^\bullet_+I(\bm n,\blank)/\Sigma_n$ is cofibrant in the corresponding model structure on commutative monoids, hence represents a left induced commutative ring spectrum. However, the global stable homotopy type represented by its underlying spectrum is \emph{not} left induced: for example, it contains $\Sigma^\bullet_+I(\bm2,\blank)/\Sigma_2$ as a retract, which is non-equivariantly but not globally weakly equivalent to $\mathcal L\Sigma^\infty_+B\Sigma_2$.
\end{warn}

\begin{constr}
	Let $R$ be a commutative $G$-ring spectrum. In the same way as before, we get an adjunction $\Ho(\cat{Mod}_R^\text{$G$-equiv})\rightleftarrows\Ho(\cat{Mod}_R^\text{$G$-gl})$. We will now explain how we can modify this to instead obtain an adjunction that lands in modules over $\mathcal L^\text{com}R$.

	For this note that on the pointset level, $\mathcal L^\text{com}R$ is simply a cofibrant replacement in the positive projective equivariant model structure, and in particular it comes with an equivariant weak equivalence $\epsilon\colon\mathcal L^\text{com}R\to R$ (modelling the counit of the derived adjunction). We can therefore consider the composite
	\begin{equation*}
		\mathcal L_R\colon\Ho(\cat{Mod}^\text{$G$-equiv}_R)\xrightarrow{\epsilon^*}
		\Ho(\cat{Mod}^\text{$G$-equiv}_{\mathcal L^\text{com}R})\xrightarrow{\mathcal L}\Ho(\cat{Mod}^\text{$G$-gl}_{\mathcal L^\text{com}R}).
	\end{equation*}
\end{constr}

\begin{lemma}
	The functor $\epsilon^*$ is an equivalence. In particular, $\mathcal L_R$ is fully faithful.
	\begin{proof}
		The functor $\epsilon^*$ is conservative and it has a left adjoint $\epsilon_!= \cat{L}(R\smashp_{\mathcal L^\text{com}R}\blank)$. It therefore suffices that the unit $X\to\epsilon^*\epsilon_!X$ is an isomorphism for every $X$. As both adjoints preserve arbitrary coproducts, it suffices to check this on generators of the triangulated category of $R$-modules.

		The homotopy category of $G$-spectra is generated by the objects $\Sigma^\infty_+G/H$ for $H\subset G$ by \cite[Proposition~4.9-(3)]{hausmann-equivariant}, so $\smash{\Ho(\cat{Mod}_{R}^\text{$G$-equiv})}$ is generated by the objects $R\smashp\Sigma^\infty_+G/H$ by the same argument as in the proof of Proposition~\ref{prop:t-structure-mod}. But on these the unit can be identified with the map $\mathcal L^\text{com}R\smashp\Sigma^\infty_+G/H\to R\smashp\Sigma^\infty_+G/H$, so the claim follows by flatness of $\Sigma^\infty_+G/H$.
	\end{proof}
\end{lemma}

Using this, we can now make precise in which sense the $G$-equivariant cotangent complex can be recovered from the $G$-global one via passing to left-induced $G$-global ultra-commutative ring spectra:

\begin{prop}
	Let $R\to S$ be a map of commutative $G$-ring spectra. Then there exists a canonical isomorphism
	\[
		\mathcal L_S\Omega_{S/R}^\textup{$G$-equiv}\cong\Omega_{\mathcal L^\textup{com}S/\mathcal L^\textup{com}R}^\textup{$G$-gl}
	\]
	in $\Ho(\cat{Mod}^\textup{$G$-gl}_{\mathcal L^\textup{com}S})$. In particular, $\Omega_{S/R}^\textup{$G$-equiv}$ and $\Omega_{\mathcal L^\textup{com}S/\mathcal L^\textup{com}R}^\textup{$G$-gl}$ become canonically isomorphic in $\Ho(\cat{Mod}^\textup{$G$-equiv}_{\mathcal L^\textup{com}S})$.

	\begin{proof}
		It suffices to prove the first statement; the second one will then follow by applying the forgetful functor from $G$-global $\mathcal L^\text{com} S$-modules to $G$-equivariant ones.

		By Corollary~\ref{cor:cotang-htpy-inv-R} we may assume that $R$ is cofibrant as a commutative $G$-ring spectrum (so that $\mathcal L^\text{com}R=R$). Using Corollary~\ref{cor:cotang-htpy-inv-S} and the definition of $\mathcal L^\text{com}$ and $\mathcal L_S$ we are then further reduced to showing that $\mathcal L\Omega_{S/R}^\textup{$G$-equiv}\cong\Omega^\text{$G$-gl}_{S/R}$ as $G$-global $S$-modules whenever $R\to S$ is a cofibration.

		We now have a commutative diagram
		\begin{equation*}\hskip-39.98pt\hfuzz=39.98pt
			\begin{tikzcd}[cramped]
				\Ho(\cat{UCom}_R^\textup{$G$-gl}/S)\arrow[d,"\forget"']\arrow[from=r, "\;\forget"'] &[1em] \arrow[d,"\forget"']\Ho(\cat{UCom}_S^\textup{$G$-gl}/S)\arrow[from=r, "\Ho(K)"',"\simeq"] & \arrow[d,"\forget"]\Ho(\cat{NUCA}_S^\textup{$G$-gl})\arrow[from=r,"\Ho(Z)"'] & \Ho(\cat{Mod}_S^{\textup{$G$-gl}})\arrow[d,"\forget"]\\
				\Ho(\cat{UCom}_R^\textup{$G$-equiv}/S)\arrow[from=r, "\;\forget"] &[1em] \Ho(\cat{UCom}_S^\textup{$G$-equiv}/S)\arrow[from=r, "\Ho(K)","\simeq"'] & \Ho(\cat{NUCA}_S^\textup{$G$-equiv})\arrow[from=r,"\Ho(Z)"] & \Ho(\cat{Mod}_S^{\textup{$G$-equiv}})
			\end{tikzcd}
		\end{equation*}
		coming from the analogous diagram of commuting homotopical functors. Passing to left adjoints yields a natural isomorphism ${\Ab_{S/R}^\text{$G$-gl}}\circ{\cat{L}\id}\cong{\mathcal L}\circ{\Ab_{S/R}^\text{$G$-equiv}}$, and chasing through the (positively projectively cofibrant) object $\id_S$ of $\cat{UCom}_R^\textup{$G$-equiv}/S$ yields the claim.
	\end{proof}
\end{prop}

\begin{cor}
	Let $H\subset G$ be any subgroup and let $M\in\Ho(\cat{Mod}^\textup{$G$-equiv}_S)$. Then we have a canonical isomorphism
	\begin{equation*}
		\ul\TAQ_*^H(S,R;M)\cong \ul\TAQ_*^{H\hookrightarrow G}(\mathcal L^\textup{com} S,\mathcal L^\textup{com} R;\mathcal L_S M)
	\end{equation*}
	between the equivariant and global topological André-Quillen homology groups.
	\begin{proof}
		We have canonical identifications
		\begin{align*}
			\mathcal L_S(\Omega_{S/R}^\textup{$G$-equiv}\smashp_S^\cat{L} M)
			=\mathcal L\epsilon^*(\Omega_{S/R}^\textup{$G$-equiv}\smashp_S^\cat{L} M)
			&\cong\mathcal L(\epsilon^*\Omega_{S/R}^\textup{$G$-equiv}\smashp_{\mathcal L^\textup{com}S}^\cat{L}\epsilon^*M)\\
			&\cong
			\mathcal L_S\Omega_{S/R}^\textup{$G$-equiv}\smashp_{\mathcal L^\textup{com}S}^\cat{L}\mathcal L_SM
		\end{align*}
		where the final isomorphism comes from the equality of underived functors, while the second one follows from the fact that the inverse $\epsilon_!$ of $\epsilon^*$ is induced by a strong symmetric monoidal functor. Combining this with the previous proposition, we obtain $\mathcal L_S(\Omega_{S/R}^\textup{$G$-equiv}\smashp_S^\cat{L} M)\cong\Omega_{\mathcal L^\textup{com}S/\mathcal L^\textup{com}R}^\textup{$G$-gl}\smashp^\cat{L}_{\mathcal L^\textup{com}S}\mathcal L_SM$. Forgetting structure, we conclude that the underlying $G$-equivariant spectra of $\Omega^\textup{$G$-gl}_{\mathcal L^\textup{com}S/\mathcal L^\textup{com}R}\smashp^\cat{L}_{\mathcal L^\textup{com}S}\mathcal L_SM$ and $\Omega^\textup{$G$-equiv}_{S/R}\smashp^\cat{L}_SM$ agree. The claim follows by taking homotopy groups.
	\end{proof}
\end{cor}

Similarly one shows:

\begin{cor}
	Let $H\subset G$ be any subgroup and let $M\in\Ho(\cat{Mod}^\textup{$G$-gl}_S)$. Then we have a canonical isomorphism
	\[
		\ul\TAQ^*_H(S,R;\forget M)\cong \ul\TAQ^*_{H\hookrightarrow G}(\mathcal L^\textup{com} S,\mathcal L^\textup{com}R; \epsilon^* M).
	\]
	In particular, if $N\in\Ho(\cat{Mod}^\textup{$G$-equiv}_S)$, then we have a canonical isomorphism
	\[
		\ul\TAQ^*_H(S,R; N)\cong \ul\TAQ^*_{H\hookrightarrow G}(\mathcal L^\textup{com} S,\mathcal L^\textup{com} R; \mathcal L_S N).\pushQED{\qed}\qedhere\popQED
	\]
\end{cor}

\subsubsection{The Postnikov tower for commutative $G$-ring spectra} We close this subsection by discussing equivariant analogues of the Hurewicz theorem and the Postnikov tower.

We start with the Hurewicz theorem. Since this is entirely analogous to the $G$-global statements proven above, we will be very brief here.

\begin{lemma}
	Let $R$ be a connective $G$-ring spectrum. Then $\Ho(\cat{Mod}_R^\textup{$G$-equiv})$ carries a t-structure with connective part $\Ho(\cat{Mod}_R^\textup{$G$-equiv})_{\ge0}$ given by those $R$-modules that are connective as $G$-equivariant spectra, and with coconnective part those spectra whose positive homotopy groups vanish.

	Moreover, if $R$ is commutative, then the derived smash product restricts to
	\[
		\Ho(\cat{Mod}_R^\textup{$G$-equiv})_{\ge m}\times\Ho(\cat{Mod}_R^\textup{$G$-equiv})_{\ge n}\to\Ho(\cat{Mod}_R^\textup{$G$-equiv})_{\ge m+n}
	\]
	for all $m,n\in\mathbb Z$.
	\begin{proof}
		We will argue for $R=\mathbb S$. The general case then follows from this by the same arguments as in the proofs of Propositions~\ref{prop:t-structure-mod} and~\ref{prop:t-structure-smash}.

		For the first statement, we recall once more from \cite[Proposition~4.9-(3)]{hausmann-equivariant} that the triangulated category $\Ho(\cat{$\bm G$-Spectra}_\textup{$G$-equiv})$ is compactly generated with generators the objects $\Sigma^\infty_+G/H$ for $H\subset G$. Thus, \cite[Theorem~A.1]{t-struc-compact} yields a t-structure with connective part the subcategory generated by the $\Sigma^\infty_+G/H$ under coproducts, supensions, and extensions. Arguing as in the proof of Proposition~\ref{prop:t-structure-mod} we then see that this t-structure admits the above description.

		For the second statement we may reduce as before to proving that the derived smash product sends the generators to connective spectra. As suspension spectra are flat, the derived smash product is given by the non-derived smash product and we immediately compute $\Sigma^\infty_+G/H\smashp\Sigma^\infty_+G/K\cong \Sigma^\infty_+(G/H\times G/K)$, which is evidently connective.
	\end{proof}
\end{lemma}

\begin{thm}[Equivariant Hurewicz Theorem]
	Let $R$ be a connective commutative $G$-ring spectrum, and let $\eta\colon R\to S$ be an $n$-equivalence for some $n\ge1$. Then $\Omega_{S/R}^\textup{$G$-equiv}$ is $n$-connected, the factorization of the universal derivation $d_{S/R}$ through $C(\eta)$ is unique, and it induces an isomorphism $\smash{\ul{\hat\pi}_{n+1}(C(\eta))\cong\ul{\hat\pi}_{n+1}(\Omega^\textup{$G$-equiv}_{S/R})}$.
	\begin{proof}
		With the previous lemma at hand, we can copy the proofs of Lemma~\ref{lemma:Q-connectivity} and Theorem~\ref{thm:Hurewicz} verbatim.
	\end{proof}
\end{thm}

Now we come to the Postnikov tower for connective $G$-equivariant commutative algebras.

\begin{constr}
	Let $R\to S$ be a map of commutative $G$-ring spectra and let $S$ be an $R$-module. Analogously to Construction~\ref{constr:extension-k-inv} we can interpret any $k\in\TAQ_G^n(S,R;M)$ as a morphism $\tilde k\colon\Omega_{S/R}^\text{$G$-equiv}\to\Sigma^nM$ of $S$-modules, or equivalently (via the adjunction between abelianization and square-zero extensions) as a map $k\colon S\to S\vee\Sigma^nM$ of commutative $G$-equivariant $R$-algebras over $S$.

	To such a class we can then associate a commutative $R$-algebra $S[k]$ over $S$ via the homotopy pullback
	\[
		\begin{tikzcd}
			S[k]\arrow[dr,phantom,"\lrcorner"{very near start}]\arrow[d]\arrow[r] & S\arrow[d,hook]\\
			S\arrow[r, "k"'] & S\vee\Sigma^nM.
		\end{tikzcd}
	\]
	As explained in the global case, the underlying $R$-module map $S[k]\to S$ can be equivalently viewed as the homotopy fiber of the composite $\tilde k\circ d_{S/R}\colon S\to\Sigma^nM$.
\end{constr}

If $R$ is a commutative $G$-ring spectrum, then its zeroth homotopy groups $\ul{\hat\pi}_0(R)$ acquire the structure of a so-called \emph{$G$-Tambara functor}. By \cite[Theorem~5.1]{ullman}, this construction admits a section, sending any $G$-Tambara functor $\ul T$ to a suitable Eilenberg-MacLane commutative $G$-ring spectrum $H\ul T$, refining the usual equivariant Eilenberg-MacLane spectra for Mackey functors.

\begin{thm}
	Let $R$ be a connective commutative $G$-ring spectrum. Then there exists an infinite tower $\dots\to R_2\to R_1\to R_0$ of commutative $R$-algebras together with classes $k_n\in\TAQ^{n+1}_G(R_n,R;H\ul{\hat\pi_{n+1}}(R))$ satisfying the following properties:
	\begin{enumerate}
		\item $R_0\cong H\ul{\hat\pi}_0(R)$ and $R_{n+1}\cong R_n[k_n]$.
		\item $\ul{\hat{\pi}}_k(R_n)=0$ for $k>n$
		\item The unit maps $R\to R_n$ are $(n+1)$-equivalences.
	\end{enumerate}
	\begin{proof}
		By \cite[Theorem~5.2]{ullman}, $H$ is right adjoint to $\ul{\hat\pi_0}$ when we restrict to \emph{connective} commutative $G$-ring spectra. In particular, we obtain a map $R\to R_0\mathrel{:=}H\ul{\hat\pi}_0(R)$ yielding the bottom stage of the tower.

		Now assume we have already constructed the tower up to $R_n$. In particular the unit $R\to R_n$ is an $(n+1)$-equivalence, and $\ul{\hat{\pi}}_k(R_n)=0$ for $k>n$. By the equivariant Hurewicz Theorem and the long exact sequence we therefore see that $\Omega_{R_n/R}^\text{$G$-equiv}$ is $(n+1)$-connected with
		\[
			\ul{\hat\pi}_{n+2}(\Omega_{R_n/R})\cong \ul{\hat\pi}_{n+2}(C(\eta)) \cong\ul{\hat\pi}_{n+1}(R).
		\]
		This isomorphism then defines a map of $R$-modules $\tilde k_n\colon\Omega_{R_n/R}\to\Sigma^{n+2}H\ul{\hat\pi}_{n+1}(R)$ corresponding to a class $k_n\in\TAQ_G^{n+1}(R_n,R;H\ul{\hat\pi}_{n+1}(R))$. The previous construction therefore yields an $R$-algebra $R_{n+1}=R_n[k_n]$ together with an $R$-algebra map $R_{n+1}\to R_n$, whose underlying $R$-module map exhibits $R_{n+1}$ as homotopy fiber of the composite $R_n\to\Omega_{R_n/R}^\text{$G$-equiv}\to\Sigma^{n+1}H\ul{\hat\pi}_{n+1}(R)$. The long exact sequence then shows that $R\to R_{n+1}$ is an $(n+2)$-equivalence and that $\ul{\hat\pi_{k}}(R_{n+1})=0$ for $k>n+1$, finishing the inductive step.
	\end{proof}
\end{thm}

\subsection{(Co)Homology as a stabilization}
We now explain how $G$-global topological André-Quillen cohomology can be interpreted as a global stabilization:

\begin{thm}\label{thm:stabilization-aug-algebras}
	Let $R$ be a flat $G$-global ultra-commutative ring spectrum. Then all the global Quillen adjunctions
	\begin{equation}
	\begin{tikzcd}
	\glo{Mod}_R\arrow[r, "\Sigma^\infty", shift left] &\arrow[l, "\Omega^\infty", shift left] \glo{Sp}(\glo{Mod}_R) \arrow[r, "\glo{Sp}(\ul{Z})", shift right, swap] &[2em]\arrow[l, "\glo{Sp}(\ul{Q})", shift right, swap] \glo{Sp}(\glo{NUCA}_R) \arrow[r, "\glo{Sp}(\ul{K})", shift left] &[2em] \arrow[l, "\glo{Sp}(\ul{I})", shift left] \glo{Sp}(\glo{Comm}_R/R)
	\end{tikzcd}
	\end{equation}
	are global Quillen equivalences. In particular, the composite
	\begin{equation*}
		\glo{Mod}_R\xrightarrow{\;Z\;}\glo{NUCA}_R\xrightarrow{I^{-1}}\glo{Comm}_R/R
	\end{equation*}
	in $\cat{GLOBMOD}$ (which on homotopy categories gives the right adjoint to $\Ab_{R/R}$) is the universal map from a stable global model category to $\glo{Comm}_R/R$.
\end{thm}

For the proof, the main step is to show that another stabilization of a global Quillen adjunction is an equivalence, namely the free-forgetful adjunction for non-unital commutative $R$-algebras:

\begin{thm}\label{thm:stabilization-nucas}
	In the above situation, the global Quillen adjunction
	\begin{equation}\label{eq:stab-free-forgetful-nuca}
	\begin{tikzcd}
	\glo{Sp}(\glo{Mod}_R) \arrow[r, "\glo{Sp}(\ul{\mathbb P^{>0}})", yshift=3pt] &[2em]\arrow[l, "\glo{Sp}(\ul{\mathbb U})", yshift=-3pt] \glo{Sp}(\glo{NUCA}_R)
	\end{tikzcd}
	\end{equation}
	is a global Quillen equivalence.
\end{thm}

The proof of this requires a substantial amount of work, and we devote all of \Cref{section:stab-nucas} to it. For now, let us use it to deduce \Cref{thm:stabilization-aug-algebras}:

\begin{proof}[Proof of \Cref{thm:stabilization-aug-algebras}]
	Since $\ul K\dashv\ul I$ is a global Quillen equivalence, so is the globally stabilized adjunction $\ul{\Sp}(\ul K)\dashv\ul{\Sp}(\ul I)$ by \Cref{prop:sp-induced-adjunction}. Moreover, the global model category $\glo{Mod}_R$ is globally stable (Corollary~\ref{cor:modules-stable}), so the stabilization adjunction
	\begin{equation*}
	\begin{tikzcd}
	\glo{Mod}_R\arrow[r, "\Sigma^\infty", shift left] &\arrow[l, "\Omega^\infty", shift left] \glo{Sp}(\glo{Mod}_R)
	\end{tikzcd}
	\end{equation*}
	is a global Quillen equivalence by \Cref{thm:stabilization-preserves-stable}.

	Finally, we have to consider how the adjunction $\ul Q\dashv\ul Z$ behaves upon stabilization. For this, we use that by \Cref{thm:stabilization-nucas} the adjunction $(\ref{eq:stab-free-forgetful-nuca})$ is a global Quillen equivalence. We now observe that the composite $\ul Q\circ \ul{\power{>0}}\colon \glo{Mod}_R\to \glo{NUCA}_R \to \glo{Mod}_R$ of left Quillen functors is naturally isomorphic to the identity. Thus, the same is true after stabilization, and also the associated composite of left derived functors is naturally isomorphic to the identity. Hence, $\ul Q\dashv\ul Z$ induces a global Quillen equivalence after stabilization by $2$-out-of-$3$, finishing the proof.
\end{proof}

\section{The stabilization of global NUCAs}\label{section:stab-nucas}
\subsection{Computing the stabilization} This section is devoted to the proof of Theorem~\ref{thm:stabilization-nucas}. We begin by recasting the left adjoint $\ul{\Sp}(\ul{\mathbb P^{>0}})$ into a more convenient form:

\begin{constr}\label{constr:diamond-sp-sp}
Let $X,Y\in \Sp(\glo{GlobalSpectra})$. We define $X\diamond Y$ as the bispectrum with $(X\diamond Y)(B)=X(B)\smashp Y(B)$ (where the right hand side is the usual Day convolution smash product on $\cat{Spectra}$) and structure maps
\begin{align*}
S^A\smashp (X\diamond Y)(B)&=S^A\smashp X(B)\smashp Y(B)\xrightarrow{\delta} S^A\smashp X(B)\smashp S^A \smashp Y(B)\\&\xrightarrow{\sigma\smashp\sigma} X(A\amalg B)\smashp Y(A\amalg B)=(X\diamond Y)(A\amalg B)
\end{align*}
where the first map is induced by the diagonal $S^A\to S^A\smashp S^A$. This becomes a functor in the obvious way.

If $G$ acts on $X$ and $H$ acts on $Y$, then $G\times H$ acts on $X\diamond Y$ via functoriality; if $G=H$, we will typically equip $X\diamond Y$ with the diagonal $G$-action, yielding a bifunctor on $G\text-\Sp(\glo{GlobalSpectra})=\Sp(G\text-\glo{GlobalSpectra})$.

Finally, let $X\in G\text-\Sp(\glo{GlobalSpectra})$ and let $n\ge 1$. Then we obtain $n$ commuting $G$-actions on the $n$-fold $\diamond$-product $X^{\diamond n}$, which together with the $\Sigma_n$-action via permuting the factors assemble into a $(\Sigma_n\wr G)$-action. We will also frequently view $X^{\diamond n}$ as an object in $(\Sigma_n\times G)\text-\Sp(\glo{GlobalSpectra})$ via the diagonal $G$-action.

More generally, if $R$ is a ($G$-)global ultra-commutative ring spectrum, then we define $\diamond_R$ as the pointwise smash product over $R$, yielding functors $\Sp(\glo{Mod}^G_R)\times\Sp(\glo{Mod}^G_R)\to\Sp(\glo{Mod}^G_R)$.
\end{constr}

\begin{rk}
For later use we record that the same construction can be applied more generally in the categories $\cat{Fun}(\bm\Sigma,\cat{Spectra})$ and $\cat{Fun}(\bm\Sigma,\cat{Mod}_R^G)$ of all $\cat{SSet}$-enriched (as opposed to $\cat{SSet}_*$-enriched) functors $\bm\Sigma\to\cat{Spectra}$ or $\bm\Sigma\to\cat{Mod}_R^G$, respectively. For these larger categories, $\diamond$ and $\diamond_R$ are the tensor product of a symmetric monoidal structure with unit the constant functor at $\mathbb S$ or $R$, respectively (while unitality fails in $\Sp(\glo{GlobalSpectra})$ and $\Sp(\glo{Mod}^G_R)$).
\end{rk}

If now $X\in \Sp(\glo{GlobalSpectra})$, then comparing right adjoints yields a natural isomorphism $\Sp(\mathbb P_R^{>0})\Sp(R\smashp\blank)(X)\cong\Sp(R\smashp\blank)\Sp(\mathbb P_{\mathbb S}^{>0})(X) =\Sp(R\smashp\blank)\big(\bigvee_{n\ge1} X^{\diamond n}/\Sigma_n\big)$. Non-equivariantly, Basterra and Mandell \cite[Theorem~2.8]{basterra-mandell-stab} proved that the summands indexed by $n>1$ vanish (for suitably cofibrant $X$), and this is the key non-formal ingredient used to compare spectra of $R$-NUCAs with $R$-modules (Theorem~3.7 of \emph{op.~cit.}). Similarly, the following $G$-global comparison will be the key computational ingredient to the proof of our Theorem~\ref{thm:stabilization-nucas}:

\begin{thm}\label{thm:diamond-sp-sp-trivial}
\begin{enumerate}
\item Let $X,Y\in G\text{-}\Sp(\glo{GlobalSpectra})$ levelwise flat. Then $X\diamond Y$ is $G$-globally weakly equivalent to $0$.
\item Let $X\in G\text{-}\Sp(\glo{GlobalSpectra})$ levelwise \emph{positively} flat. Then $X^{\diamond n}/\Sigma_n$ is $G$-globally weakly equivalent to $0$ for every $n>1$.
\item If $i\colon X\to Y$ is a levelwise positive flat cofibration, then $i^{\ppo n}/\Sigma_n\colon Q^n/\Sigma_n\to Y^{\diamond n}/\Sigma_n$ is a levelwise positive flat cofibration and a {$G$-global weak equivalence} for every $n>1$. In particular, if $Y$ is levelwise positively flat, then $Q^n/\Sigma_n$ is $G$-globally weakly equivalent to $0$.
\end{enumerate}
\end{thm}

The proof will be given below after some preparations; for now, let us use it to deduce the theorem:

\begin{proof}[Proof of Theorem~\ref{thm:stabilization-nucas}]
Fix a finite group $H$. We have to show that $H\text-\Sp(\ul{\mathbb P^{>0}})\dashv H\text-\Sp(\ul{\mathbb U})$ is a Quillen equivalence for the flat model structures. Replacing $G$ by $G\times H$ (and letting $H$ act trivially on $R$), we may assume without loss of generality that $H=1$.

\begin{claim*}
The positive flat global model structure on $\Sp(\glo{Mod}_R^G)$ transfers to $\Sp(\glo{NUCA}_R^G)$ along $\Sp(\ul{\mathbb P}^{>0})\dashv \Sp(\ul{\mathbb U})$.
\begin{proof}
As the positive flat global model structure on $\Sp(\glo{Mod}_R^G)$ is in turn transferred from $\Sp(G\text-\glo{GlobalSpectra}^+)$ (Lemma~\ref{lemma:module-transferred}) it suffices to consider the composite adjunction
\begin{equation*}
\Sp(\ul{R\smashp\mathbb P^{>0}_\mathbb S})\colon\Sp(G\text-\glo{GlobalSpectra}^+)\rightleftarrows\Sp(\glo{NUCA}_R^G) :\!\Sp(\ul{\mathbb U}).
\end{equation*}
By local presentability, we only have to show that every $\Sp(\ul{R\smashp\mathbb P^{>0}_\mathbb S})(J)$-cell complex is sent by $\Sp(\ul{\mathbb U})$ to a $G$-global weak equivalence, where $J$ is our favourite set of generating acyclic cofibrations of $G\text-\Sp(\glo{GlobalSpectra})$.

For this, we consider a pushout square
\begin{equation*}
\begin{tikzcd}
\Sp(\ul{R\smashp\mathbb P}^{>0})(A)\arrow[r, "{\Sp(\ul{R\smashp\mathbb P}^{>0})(j)}"]\arrow[d] &[3em] \Sp(\ul{R\smashp\mathbb P}^{>0})(B)\arrow[d]\\
X \arrow[r, "k"'] & Y
\end{tikzcd}
\end{equation*}
in $\Sp(\glo{NUCA}_R^G)$ with $j\in J$. We want to express $\Sp(\ul{\mathbb U})(k)$ as a transfinite composition of weak equivalences, for which we note that we can identify the $1$-category $\Sp(\glo{NUCA}_R^G)$ over $\Sp(\glo{GlobalSpectra})$ with non-unital commutative monoids in $\Sp(\glo{Mod}_R^G)$ with respect to the non-unital symmetric product $\diamond_R$. While we cannot literally apply Remark~\ref{rk:pushout-analysis-nucas} in this setting (for lack of unitality), we can do this inside the larger category $\cat{Fun}(\bm\Sigma,\cat{Mod}^G_R)$, which (as spectrum objects are closed inside $\cat{SSet}$-enriched functors under all colimits) then factors $k$ as a map in $G\text-\Sp(\glo{GlobalSpectra})$ into a transfinite composition of maps $k_n$ fitting into pushout squares
\begin{equation*}\hskip-30.71pt\hfuzz=31pt
\begin{tikzcd}
\Sp(\ul{R\smashp\blank})(Q^n/\Sigma_n \vee X\diamond Q^{n}/\Sigma_{n})\arrow[r, "\Sp(\ul{R\smashp\blank})(j^{\ppo n}/\Sigma_n\vee X\diamond j^{\ppo n}/\Sigma_{n})"]\arrow[d] &[9em] \Sp(\ul{R\smashp\blank})(L^{\diamond n}/\Sigma_n\vee X\diamond L^{\diamond n}/\Sigma_n)\arrow[d]\\
\cdot\arrow[r, "k_n"'] & \cdot
\end{tikzcd}
\end{equation*}
in $G\text-\Sp(\glo{GlobalSpectra})$. Appealing to Theorem~\ref{thm:diamond-sp-sp-trivial} above, the top map is a map between weakly contractible objects, hence a weak equivalence, and moreover a levelwise positive flat cofibration (in particular an injective cofibration). Thus, also $k_n$ is a weak equivalence, and hence so is $\Sp(\ul{\mathbb U})(k)$ since filtered colimits in $\Sp(\glo{Mod}_R^G)$ are homotopical (Lemma~\ref{lemma:g-global-sp-sp-colim}). As $\Sp(\ul{\mathbb U})$ preserves filtered colimits, it then follows by the same argument that also any transfinite compositions of maps of the above form are sent to $G$-global weak equivalences as desired.
\end{proof}
\end{claim*}

We now observe that the transferred model structure and the global model structure on $\Sp(\glo{NUCA}_R^G)$ have the same acyclic fibrations (namely, those maps $f$ for which each $f(A)$ is a fibration in the positive flat $(G\times\Sigma_A)$-global model structure on $(G\times\Sigma_A)\text-\Sp(\glo{GlobalSpectra}^+)$) and the same fibrant objects (namely, those $X$ for which each $X(A)$ is $(G\times\Sigma_A)$-globally positively flatly fibrant and for which $X(A)\to\Omega^BX(A\amalg B)$ is a $(G\times H)$-global weak equivalence of NUCAs or equivalently of spectra for all $H$-sets $A,B$). Thus, the two model structures actually agree, and in particular we see that $\Sp(\mathbb U)\colon \Sp(\glo{NUCA}_R^G)\to \Sp(\glo{Mod}_R^G)$ preserves and reflects weak equivalences. To complete the proof, it suffices now to show that the ordinary unit $X\to\Sp(\ul{\mathbb U\mathbb P^{>0}})(X)$ is a weak equivalence for every cofibrant $X$.

This is again a standard cell induction argument \cite[Corollary~1.2.65]{g-global} using \Cref{thm:diamond-sp-sp-trivial}: we let $\mathscr H$ denote the class of all objects for which the above is a weak equivalence. If $Y$ is a levelwise positively flat $G$-bispectrum and $X=\Sp(\ul{R\smashp\blank})(Y)$, then the unit for $X$ agrees up to isomorphism with the inclusion of the first summand of $\bigvee_{n\ge 1} \Sp(\ul{R\smashp\blank})(Y^{\diamond n}/\Sigma_n)$, so it is a weak equivalence by Theorem~\ref{thm:diamond-sp-sp-trivial}, i.e.~$\Sp(\ul{R\smashp\blank})(Y)\in\mathscr H$; in particular, $\mathscr H$ contains the initial object $0$ as well as all sources and targets of the usual generating cofibrations $I$. However, $\mathscr H$ is closed under pushouts along cofibrations (as the right Quillen functor $\Sp(\mathbb U)$ sends these to homotopy pushouts by stability and left properness) as well as filtered colimits (as these are homotopical by Lemma~\ref{lemma:g-global-sp-sp-colim} and preserved by $\Sp(\mathbb U))$, so $\mathscr H$ more generally contains all $I$-cell complexes, whence all cofibrant objects by Quillen's Retract Argument.
\end{proof}

\subsection{The levelwise smash product is trivial} In this subsection we will complete the proof of Theorem~\ref{thm:stabilization-nucas} by proving Theorem~\ref{thm:diamond-sp-sp-trivial}.

\subsubsection{The case of ordinary spectra} For this, we first consider the analogues of Construction~\ref{constr:diamond-sp-sp} and Theorem~\ref{thm:diamond-sp-sp-trivial} for ordinary $G$-spectra:

\begin{constr}
Let $X,Y$ be spectra. We define $X\diamond Y$ as the spectrum with $(X\diamond Y)(B)=X(B)\smashp Y(B)$ (smash product of pointed simplicial sets) and structure maps
\begin{align*}
S^A\smashp (X\diamond Y)(B)&=S^A\smashp X(B)\smashp Y(B)\xrightarrow{\delta} S^A\smashp X(B)\smashp S^A \smashp Y(B)\\&\xrightarrow{\sigma\smashp\sigma} X(A\amalg B)\smashp Y(A\amalg B)=(X\diamond Y)(A\amalg B)
\end{align*}
where the first map is induced by the diagonal $S^A\to S^A\smashp S^A$. This becomes a functor in the obvious way.

If $X$ and $Y$ are $G$-spectra, then we equip $X\diamond Y$ with the induced action. Given any $n\ge1$, we view $X^{\diamond n}$ as a $(\Sigma_n\wr G)$- or $(\Sigma_n\times G)$-spectrum.
\end{constr}

\begin{prop}\label{prop:diamond-sp-trivial}
Let $X,Y$ be $G$-spectra.
\begin{enumerate}
\item $X\diamond Y$ is $G$-globally weakly contractible.
\item Let $n>1$. Then $X^{\diamond n}$ is $(\Sigma_n\wr G)$-globally weakly contractible.\label{item:dst-iterated}
\end{enumerate}
\begin{proof}
We will prove the second statement, the proof of the first one being similar but easier. We will show that it is even $\ul\pi_*$-isomorphic to $0$, for which we will need the following easy geometric input:

\begin{claim*}
Let $H$ be a finite group and let $F,A$ be finite $H$-sets such that $F$ is non-empty and free while $|A| > 1$. Then the diagonal embedding $\delta\colon S^F\to S^{A\times F}$ of $H$-\emph{topological} spaces is $H$-equivariantly based nullhomotopic.
\begin{proof}
The map $\delta$ is the $1$-point compactification of the diagonal map $\delta\colon\mathbb R[F]\to \mathbb R[A\times F]$. Now the $H$-fixed points of the source have dimension $|F/H|$ while the $H$-fixed points of the target have dimension $|(A\times F)/H| = |A|\cdot|F/H| > |F/H|$ where the first equality uses freeness of $F$. In particular, there exists an $H$-fixed point $p\in \mathbb R[A\times F]$ outside the diagonal. But then
\begin{align*}
(0,1]\times \mathbb R[F]&\to \mathbb R[A\times F]\\
(t,x)&\mapsto p + t^{-1}(\delta(x)-p)
\end{align*}
is $H$-equivariant and one easily checks that this extends to an equivariant based homotopy from the map constant at $\infty$ to the diagonal embedding $S^F\to S^{A\times F}$.
\end{proof}
\end{claim*}
Fix now a finite group $H$ and a homomorphism $\phi\colon H\to\Sigma_n\wr G$; we will show that $\pi_0^\phi(X^{\diamond n})=0$, the argument in other dimensions being analogous but requiring slightly more notation. For this we pick an exhaustive sequence $B_0\subset B_1\subset\cdots\subset\mathcal U_H$ of our favourite complete $H$-set universe $\mathcal U_H$ such that each $B_{k+1}\setminus B_k$ contains a free $H$-orbit. Then
\begin{equation*}
\pi_0^\phi(X^{\diamond n})=\colim_{B\in s(\mathcal U_H)}[S^B, |\phi^*(X^{\diamond n})(B)|]^H_*\cong \colim_k [S^{B_k},|\phi^*(X^{\diamond n})(B_k)|]^H_*
\end{equation*}
by cofinality. But on the other hand, if $F\subset B_{k+1}\setminus B_k$ is a free $H$-orbit, then the transition map $[S^{B_k},|\phi^*(X^{\diamond n})(B_k)|]^H_*\to [S^{B_{k+1}},|\phi^*(X^{\diamond n})(B_{k+1})|]^H_*$ factors by definition through
\begin{equation*}\hskip-20.384pt\hfuzz=21pt
[S^F\smashp S^{B_k}, S^F\smashp |\phi^*(X^{\diamond n})(B_{k})|]^H_*\xrightarrow{\delta\smashp |\phi^*(X^{\diamond n})(B_k)|} [S^F\smashp S^{B_k}, S^{(\pr_{\Sigma_n}\circ\phi)^*\bm{n}\times F}\smashp |\phi^*(X^{\diamond n})(B_{k})|]^H_*
\end{equation*}
which is null by the claim.
\end{proof}
\end{prop}
Beware that Theorem~\ref{thm:diamond-sp-sp-trivial} does not follow simply by applying the proposition levelwise; in particular, taking $X=Y=\Sigma^\infty\mathbb S$ in the first item we have $(X\diamond Y)(A)\cong\Sigma^{2\cdot A}\mathbb S$, so $X\diamond Y$ is not globally \emph{level} equivalent to $0$. Instead, the reduction will require further preparation.

\subsubsection{The external smash product}
We begin by introducing yet another smash product:

\begin{constr}
Let $X,Y\in\cat{Spectra}$. Then we define the \emph{external smash product} $X\extsmash Y\in\Sp(\cat{Spectra})$ as the bispectrum with $(X\extsmash Y)(A)(B)= X(A)\smashp Y(B)$ and with the evident structure maps, i.e.~$(X\extsmash Y)(A)=X(A)\smashp Y$ and $(X\extsmash Y)(\blank)(B)=X\smashp Y(B)$ as spectra. Again, we extend this to $G$-spectra by pulling through the $G$-action.
\end{constr}

Note that the above agrees with the functor $\blank\smashp\blank\colon\cat{Spectra}\times\mathscr C\to\Sp(\ul{\mathscr C})$ considered in Subsection~\ref{subsec:spectrification}, specialized to $\mathscr C=\cat{Spectra}$. However, in our setting using `$\smashp$' again would be highly ambiguous, which is why we introduced the above notation.

\begin{thm}\label{thm:ext-smash}
Let $G$ be any finite group. Then the external smash product
\begin{equation*}
\blank\extsmash\blank\colon\cat{$\bm G$-Spectra}_\textup{$G$-global}\times\cat{$\bm G$-Spectra}_\textup{$G$-global}\to G\text-\Sp(\glo{GlobalSpectra})_\textup{$G$-global}
\end{equation*}
is homotopical in both variables.
\end{thm}

The statement for the second variable is actually quite easy:

\begin{lemma}\label{lemma:ext-smash-easy}
Let $X$ be a $G$-spectrum and let $f$ be a $G$-global weak equivalence of $G$-spectra. Then $X\extsmash f$ is a $G$-global weak equivalence in $G\text-\Sp(\glo{GlobalSpectra})$.
\begin{proof}
If $A$ is any finite set, then $f$ is a $(G\times\Sigma_A)$-global weak equivalence (with respect to the trivial $\Sigma_A$-actions), and hence so is $X(A)\smashp f$. Thus, $X\extsmash f$ is even a $G$-global level weak equivalence.
\end{proof}
\end{lemma}

The proof that the external smash product is also homotopical in the first variable is much harder. We begin with some further closure properties of the $G$-global weak equivalences of $G\text-\Sp(\glo{GlobalSpectra})$ similar to Lemma~\ref{lemma:g-global-sp-sp-colim}:

\begin{lemma}\label{lemma:g-global-sp-sp-quotient}
Let $f\colon X\to Y$ be a $G$-global weak equivalence in $G\text-\Sp(\glo{GlobalSpectra})$ and let $\alpha\colon G\to G'$ be a homomorphism of finite groups. Assume that $\ker(\alpha)$ acts levelwise freely outside the basepoint on $X$ and $Y$. Then $\alpha_!f$ is a $G'$-global weak equivalence.
\begin{proof}
Employing functorial factorizations in the $G$-global projective \emph{level} model structure we obtain a commutative diagram
\begin{equation*}
\begin{tikzcd}
X'\arrow[d, "\sim"']\arrow[r, "f'"] & Y'\arrow[d, "\sim"]\\
X\arrow[r, "f"'] & Y
\end{tikzcd}
\end{equation*}
such that the vertical maps are $G$-global level weak equivalences and $X',Y'$ are projectively cofibrant. Then $f'$ is a $G$-global weak equivalence by $2$-out-of-$3$, and hence $\alpha_!f'$ is a $G'$-global weak equivalence by Ken Brown's Lemma. By another application of $2$-out-of-$3$, it will then be enough to show that $\alpha_!$ sends the vertical maps to $G'$-global weak equivalences. However, $G$ (and hence in particular $\ker\alpha$) acts levelwise freely outside the basepoint on $X'$ and $Y'$ by \cite[Remark~3.1.22]{g-global} together with Lemma~\ref{lemma:projective-cofibrations-levelwise} (for the trivial $G$-action on the indexing set), whence the claim follows from Proposition~\ref{prop:free-quotient-spectra}.
\end{proof}
\end{lemma}

Moreover, one proves in the same way:

\begin{lemma}\label{lemma:G-sset-bispectra-homotopical}
The levelwise smash product
\begin{equation*}
(\cat{$\bm G$-SSet}_*)_\textup{$G$-equivariant}\times G\text-\Sp(\glo{GlobalSpectra})_\textup{$G$-global}\to G\text-\Sp(\glo{GlobalSpectra})_\textup{$G$-global}
\end{equation*}
is homotopical in each variable.\qed
\end{lemma}

\begin{proof}[Proof of Theorem~\ref{thm:ext-smash}]
Write $\mathscr H$ for the class of $G$-global spectra $X$ such that $\blank\extsmash X$ is homotopical. Our goal is to show that $\mathscr H$ consists of all objects, which will be done in several steps.

\medskip\noindent\textit{Step 1. For every $Y\in\cat{$\bm G$-$\bm{\mathcal I}$-SSet}_*$ we have $\Sigma^\bullet Y\in\mathscr H$.}\\
Plugging in the definitions, we have for every $G$-global spectrum $T$ and all finite sets $A,B$ a natural isomorphism
\begin{align*}
(T\extsmash \Sigma^\bullet Y)(A)(B)=T(A)\smashp S^B\smashp Y(B)&\cong S^B\smashp T(A)\smashp Y(B)\\&= \big(G\text-\Sp(\ul{\Sigma^\bullet})(T\extotimes Y)\big)(A)(B)
\end{align*}
where $\extotimes$ denotes the `external tensor product' of a $G$-spectrum with a $G$-$\mathcal I$-simplicial set, i.e.~the $G$-spectrum object in $\cat{$\bm{\mathcal I}$-SSet}_*$ given by $(T\extotimes Y)(A)(B)=T(A)\smashp Y(B)$ with the obvious functoriality. Letting $A$ and $B$ vary, one then easily checks that the above induces a natural isomorphism $T\extsmash \Sigma^\bullet Y\cong G\text-\Sp(\ul{\Sigma^\bullet})(T\extotimes Y)$. However, $G\text-\Sp(\ul{\Sigma^\bullet})$ is left Quillen (say, for the projective model structures) and preserves $G$-global level weak equivalences (Proposition~\ref{prop:suspension-loop-G-gl}), so it is in fact fully homotopical. Thus, it will be enough to show that $\blank\extotimes Y$ sends $G$-global weak equivalences of $G$-spectra to $G$-global weak equivalences in $G\text-\Sp(\glo{GlobalSpaces})$. However, by Proposition~\ref{prop:Delta-star-homotopical} the latter are detected by the diagonal restriction $\Delta^*$, and $\Delta^*(\blank\extotimes Y)=\blank\otimes Y$ preserves $G$-global weak equivalences of $G$-spectra by Proposition~\ref{prop:tensor-homotopical}.

\medskip\noindent\textit{Step 2. For every $Y\in\cat{$\bm G$-$\bm{I}$-SSet}_*$ we have $\Sigma^\bullet Y\in\mathscr H$.}\\
By Lemma~\ref{lemma:ext-smash-easy} and $2$-out-of-$3$, $\mathscr H$ is closed under $G$-global weak equivalences. But \cite[Theorem~1.4.31 and~Proposition~3.2.2]{g-global} yield a $G$-global weak equivalence $\Sigma^\bullet Y\simeq \Sigma^\bullet(\mathcal I\times_IY)$ for some pointed $G$-$\mathcal I$-simplicial set $\mathcal I\times_I Y$, so the claim follows from the previous step.

\medskip\noindent\textit{Step 3. For every finite $G$-set $A$ and every $K\in\cat{$\bm G$-SSet}_*$, $\bm\Sigma(A,\blank)\smashp K\in\mathscr H$.}\\
The endofunctor $S^A\smashp\blank$ of $G\text-\Sp(\glo{GlobalSpectra})$ is homotopical (by the previous lemma) and part of a Quillen equivalence (by stability), so it reflects weak equivalences. Thus, it suffices to show that $S^A\smashp (f\extsmash \bm\Sigma(A,\blank)\smashp K)$ is a $G$-global weak equivalence for every $G$-global weak equivalence $f$. However, this is conjugate to $f\extsmash \big(S^A\smashp\bm\Sigma(A,\blank)\big)\smashp K$, and by a simple Yoneda argument $S^A\smashp\bm\Sigma(A,\blank)\cong\Sigma^\bullet_+ I(A,\blank)$ naturally in $A$ (hence $G$-equivariantly). Thus, the claim follows from the previous step.

\medskip\noindent\textit{Step 4. Let $H$ be a finite group, $\phi\colon H\to G$ a homomorphism, and $A$ a finite faithful $H$-set. Then $\bm\Sigma(A,\blank)\smashp_\phi G_+\smashp K\in\mathscr H$ for every pointed $G$-simplicial set $K$.}\\
Applying the previous step with $G$ replaced by $G\times H$ shows that the functor $\blank\extsmash\bm\Sigma(A,\blank)\smashp G_+\smashp K$ sends $G$-global weak equivalences to $(G\times H)$-global weak equivalences, where $H$ acts on $A$ in the given way, on $G$ from the right via $\phi$, and trivially everywhere else. However, as $A$ is faithful, $H$ acts freely on $\bm\Sigma(A,B)$ outside the basepoint for every finite set $B$. Thus, Lemma~\ref{lemma:g-global-sp-sp-quotient} immediately implies that $(\blank\extsmash\bm\Sigma(A,\blank)\smashp G_+\smashp K)/H\cong \blank\extsmash\bm\Sigma(A,\blank)\smashp_\phi G_+\smashp K$ is homotopical.

\medskip\noindent\textit{Step 5. Every projectively cofibrant $G$-global spectrum is contained in $\mathscr H$.}\\
Fix a $G$-global weak equivalence $f\colon T\to U$, which induces a natural transformation $T\extsmash\blank\Rightarrow U\extsmash\blank$. We want to show that this is a weak equivalence on all projectively cofibrant objects. This is again a standard cell induction argument \cite[Lemma~1.2.64]{g-global}: by the previous step the claim is true for the sources and targets of the standard generating cofibrations, and moreover for any $G$-spectrum $V$ the functor $V\extsmash \blank$ preserves colimits as well as injective cofibrations; the claim therefore follows from Lemma~\ref{lemma:g-global-sp-sp-colim}.

\medskip\noindent\textit{Step 6. All $G$-global spectra belong to $\mathscr H$.}\\
Every $G$-global spectrum is weakly equivalent to a projectively cofibrant one. The claim therefore follows from the previous step together with Lemma~\ref{lemma:ext-smash-easy}.
\end{proof}

\subsubsection{Proof of triviality} Using this, we can prove a key special case of Theorem~\ref{thm:diamond-sp-sp-trivial}:

\begin{prop}\label{prop:diamond-power-corep}
Let $A$ be any finite set, let $X$ be a positively flat $(G\times\Sigma_A)$-spectrum, and let $n>1$. Then $(\bm\Sigma(A,\blank)\extsmash_{\Sigma_A} X)^{\diamond n}/\Sigma_n$ is $G$-globally weakly contractible.
\begin{proof}
Reordering factors we have
\begin{equation*}
(\bm\Sigma(A,\blank)\extsmash X)^{\diamond n}(B)(C)\cong\bm\Sigma(A,B)^{\smashp n}\smashp {X^{\smashp n}}(C)
\end{equation*}
inducing a $G\times(\Sigma_n\wr\Sigma_A)$-equivariant isomorphism of bispectra $(\bm\Sigma(A,\blank)\extsmash X)^{\diamond n}\cong \bm\Sigma(A,\blank)^{\diamond n}\extsmash X^{\smashp n}$. By Proposition~\ref{prop:diamond-sp-trivial}, $\bm\Sigma(A,\blank)^{\diamond n}$ is $\Sigma_n\wr(G\times\Sigma_A)$- and hence also $G\times(\Sigma_n\wr\Sigma_A)$-globally weakly contractible, so Theorem~\ref{thm:ext-smash} shows that $\bm\Sigma(A,\blank)^{\diamond n}\extsmash X^{\smashp n}$ is $G\times(\Sigma_n\wr\Sigma_A)$-globally weakly contractible. Now $\Sigma_A^n$ acts levelwise freely on this (as it already does so on $\bm\Sigma(A,\blank)^{\smashp n}$), whence Lemma~\ref{lemma:g-global-sp-sp-quotient} shows that $(\bm\Sigma(A,\blank)\extsmash_{\Sigma_A}X)^{\diamond n}\cong \bm\Sigma(A,\blank)^{\diamond n}\extsmash_{\Sigma_A^n} X^{\smashp n}$ is $(G\times\Sigma_n)$-globally weakly contractible. But $\bm\Sigma(A,\blank)\extsmash_{\Sigma_A} X$ is levelwise positively flat, so the $\Sigma_n$-action on its $n$-th $\diamond$-power is free by Lemma~\ref{lemma:smash-power-free}, and the proposition follows by another application of Lemma~\ref{lemma:g-global-sp-sp-quotient}.
\end{proof}
\end{prop}

In the same way one shows:

\begin{prop}
Let $A,B$ be finite sets, let $X$ be a flat $(G\times\Sigma_A)$-spectrum, and $Y$ a flat $(G\times\Sigma_B)$-spectrum. Then $(\bm\Sigma(A,\blank)\extsmash_{\Sigma_A}X)\diamond (\bm\Sigma(B,\blank)\extsmash_{\Sigma_B} Y)$ is $G$-globally weakly contractible.\qed
\end{prop}

From this we can immediately easily deduce the following slight strenghtening of the first part of Theorem~\ref{thm:diamond-sp-sp-trivial}:

\begin{prop}\label{prop:diamond-sp-sp-trivial-first-half}
Let $X,Y\in G\text-\Sp(\glo{GlobalSpectra})$ and assume at least one of them is levelwise flat. Then $X\diamond Y$ is $G$-globally weakly contractible.
\begin{proof}
We first observe the following closure properties:
\begin{claim*}
For any (levelwise) flat $X\in G\text-\Sp(\glo{GlobalSpectra})$ the class of objects $Y$ for which $X\diamond Y$ is weakly contractible is closed under (a) filtered colimits \emph{and} (b) pushouts along levelwise flat cofibrations.
\begin{proof}
We will prove the second statement, the argument for the first one being similar. Consider a pushout in $G\text-\Sp(\glo{GlobalSpectra})$ as on the left
\begin{equation*}
\begin{tikzcd}
A\arrow[d]\arrow[dr, phantom, "\ulcorner"{very near end}]\arrow[r, "i"] & B\arrow[d]\\
C\arrow[r]&D
\end{tikzcd}
\qquad\qquad
\begin{tikzcd}
X\diamond A\arrow[dr, phantom, "\ulcorner"{very near end}]\arrow[d]\arrow[r, "X\diamond i"] & X\diamond B\arrow[d]\\
X\diamond C\arrow[r]& X\diamond D
\end{tikzcd}
\end{equation*}
such that $i$ is levelwise flat. Applying $X\diamond\blank$ to this yield a pushout as on the right (as $X\diamond\blank$ is cocontinuous), and $X\diamond i$ is a levelwise flat cofibration (in particular an injective cofibration). Thus, if $X\diamond A,X\diamond B,X\diamond D$ are weakly contractible then so is $X\diamond D$ by Lemma~\ref{lemma:g-global-sp-sp-colim}.
\end{proof}
\end{claim*}

Using this, the previous proposition immediately proves the special case that $Y$ is flat and $X=\bm\Sigma(A,\blank)\extsmash_{\Sigma_A}Z$ for some flat $Z$. But for any flat $Y$, the class of $X$ for which $X\diamond Y$ is weakly contractible is again closed under filtered colimits and pushouts along flat cofibrations (by symmetry), proving the case that both $X$ and $Y$ are flat. As for any levelwise flat $Z$ both $\blank\diamond Z$ and $Z\diamond\blank$ preserve $G$-global \emph{level} weak equivalences, the claim now follows by cofibrant replacement.
\end{proof}
\end{prop}

\begin{proof}[Proof of Theorem~\ref{thm:diamond-sp-sp-trivial}]
The first statement is a special case of the previous proposition. For the second statement, we observe that the class of \emph{levelwise positively flat} $X\in G\text-\Sp(\glo{GlobalSpectra})$ such that $X^{\diamond n}/\Sigma_n$ is $G$-globally weakly contractible contains $0$ and is closed under filtered colimits. To complete the proof it is then enough to show that it is also closed under pushouts along generating cofibrations, for which we more generally consider any pushout
\begin{equation}\label{diag:po-sp-sp-trivial}
\begin{tikzcd}
\bm\Sigma(A,\blank)\extsmash_{\Sigma_A} X\arrow[r, "{\bm\Sigma(A,\blank)\extsmash_{\Sigma_A}i}"]\arrow[d] &[3em] \bm\Sigma(A,\blank)\extsmash_{\Sigma_A} Y\arrow[d]\\
Z\arrow[r, "j"'] & P\arrow[from=ul,phantom, "\ulcorner"{very near end,xshift=.75em,yshift=2pt}]
\end{tikzcd}
\end{equation}
such that $i\colon X\to Y$ is a positive flat cofibration of $(G\times\Sigma_A)$-spectra and $Z$ is levelwise positively flat with $Z^{\diamond n}/\Sigma_n\simeq0$ for all $n>1$; we will prove that also $P^{\diamond n}/\Sigma_n\simeq0$ for all $n>1$. For this, we apply \cite[Theorem~22]{sym-powers} (to the larger category of all $\cat{SSet}$-enriched functors) yielding factorizations
\begin{equation*}
Z^{\diamond n}/\Sigma_n=Q^n_0\to Q^n_1\to\cdots\to Q^n_{n-1}\to Q^n_n=P^{\diamond n}/\Sigma_n
\end{equation*}
of $j^{\diamond n}/\Sigma_n$ for all $n$ such that we have for $k<n-1$ a pushout
\begin{equation}\label{diag:qkn-po}
\begin{tikzcd}
Z^{\diamond (n-k)}/\Sigma_{n-k}\diamond Q^k_{k-1} \arrow[r, "Z^{\diamond(n-k)}/\Sigma_{n-k}\diamond j^{\ppo k}/\Sigma_k"]\arrow[d] &[6.6em] Z^{\diamond (n-k)}/\Sigma_{n-k}\diamond P^{\diamond k}/\Sigma_k\arrow[d]\\
Q^n_k\arrow[r] & Q^n_{k+1}\arrow[from=ul,phantom, "\ulcorner"{very near end, xshift=1.7em,yshift=2pt}]
\end{tikzcd}
\end{equation}
(where for $k=0$ the empty $\diamond$-power has to be interpreted as a formal unit) while $Q^n_{n-1}\to Q^n_n$ recovers $j^{\ppo n}/\Sigma_n$.

\begin{claim*}
For all $0\le k<n$ the map $Q^n_k\to Q^{n}_{k+1}$ is a levelwise flat cofibration.
\begin{proof}
For $k=n-1$ this is an instance of the strong commutative monoid axiom for the (say, non-equivariant) positive flat model structure on symmetric spectra.

For $k<n-1$, on the other hand, we observe that the top arrow in the pushout square $(\ref{diag:qkn-po})$ is a positive levelwise flat cofibration by the above special case and levelwise flatness of $Z^{\diamond(n-k)}/\Sigma_{n-k}$. The claim follows immediately.
\end{proof}
\end{claim*}
By (ordinary) stability, it therefore suffices to show that the $1$-categorical cofiber of $Q^n_k\to Q^n_{k+1}$ is trivial for all $0\le k<n$. However by the aforementioned \cite[Theorem~22]{sym-powers}, this cofiber is isomorphic to $Z^{\diamond(n-k)}/\Sigma_{n-k}\diamond \cofib(j)^{\diamond k}/\Sigma_{k}$. For $0<k<n$ this is $G$-globally weakly contractible by Proposition~\ref{prop:diamond-sp-sp-trivial-first-half}, while for $k=0$ this is $G$-globally weakly contractible by assumption on $Z$. Finally, if $k=n$, then we observe that $\cofib(j)\cong \bm\Sigma(A,\blank)\extsmash_{\Sigma_A}\cofib(i)$ because of the pushout $(\ref{diag:po-sp-sp-trivial})$; as $\cofib(i)$ is levelwise positive flat, the claim now follows from Proposition~\ref{prop:diamond-power-corep}.

Finally, for the third statement, we simply argue as before to see that for any levelwise positive flat cofibration $i$ the pushout product $i^{\ppo n}/\Sigma_n$ is again a levelwise positive flat cofibration, and that its ($1$-categorical) cofiber agrees with $\cofib(i)^{\diamond n}/\Sigma_n$, which is $G$-globally weakly contractible by the second statement.
\end{proof}

\frenchspacing
\bibliographystyle{amsalpha}
\bibliography{literature.bib}
\end{document}